\documentclass[reqno]{amsart}
\numberwithin{equation}{section}
\usepackage[all]{xy}
\usepackage{eurosym}
\usepackage[colorlinks=false,hidelinks]{hyperref}
\usepackage{amsmath}
\usepackage{amssymb}
\usepackage{amsfonts}
\usepackage[mathscr]{eucal} 
\usepackage{enumerate}
\usepackage{graphicx}
\newtheorem{theorem}{Theorem}[section]
\newtheorem{corollary}[theorem]{Corollary}
\newtheorem{lemma}[theorem]{Lemma}
\newtheorem{proposition}[theorem]{Proposition}

\theoremstyle{definition}
\newtheorem{definition}[theorem]{Definition}
\theoremstyle{remark}

\newtheoremstyle{mystyle}{2mm}{0mm}{}{}{\bfseries}{}{ }{\thmnumber{#2}.\thmnote{
#3}}
\theoremstyle{mystyle}
\newtheorem{fact}[theorem]{}
\newtheoremstyle{myremark}{2mm}{0mm}{}{}{\bfseries}{}{ }{\thmname{#1}
\thmnumber{#2}. \thmnote{ #3}}
\theoremstyle{myremark}
\newtheorem{remark}[theorem]{Remark}

\def\bG{\boldsymbol{\Gamma}}
\def\bU{\boldsymbol{\Upsilon}}
\def\kG{\Bbbk\boldsymbol{\Gamma}}

\def\ot{\otimes}

\def\wot{{\,\widehat{\otimes}\,}}

\def\Id{\mathrm{I}} 
\DeclareMathOperator{\Ker}{Ker} \DeclareMathOperator{\h}{H}
\DeclareMathOperator{\hh}{HH}
\DeclareMathOperator{\Ho}{H}
\DeclareMathOperator{\Hom}{Hom}\DeclareMathOperator{\tor}{Tor}
\DeclareMathOperator{\coker}{Coker}
\DeclareMathOperator{\ext}{Ext}
\DeclareMathOperator{\im}{Im}
\DeclareMathOperator{\K}{K}
\DeclareMathOperator{\Hd}{Hdim}
\DeclareMathOperator{\Pd}{pd}

\DeclareMathOperator{\Sh}{Sh}

\def\pd{\partial}

\def\htp{{\,\widehat{\otimes}\,}}

\def\fc{\mathfrak{c}}\def\bK{\Bbbk}
\email{}
\thanks{}

\DeclareGraphicsExtensions{.pdf}
\graphicspath{{./ktp/}}

\begin{document}
\title{Koszul Pairs and Applications}
\author{Pascual Jara}
	\address{University of Granada, Department of Algebra, Granada E-18071, Spain}
	\email{pjara@ugr.es}
\author{Javier L\'{o}pez Pe\~na}
	\address{Department of Mathematics, University College London, Gower Street, London WC1E 6BT, UK}
	\email{jlp@math.ucl.ac.uk}
\author{Drago\c s \c Stefan}
	\address{University of Bucharest, Department of Mathematics, 14 Academiei Street, Bucharest Ro-010014, Romania}
	\email{dragos.stefan@fmi.unibuc.ro}
\date{}

\begin{abstract}
  Let $R$ be a semisimple ring. A pair $(A,C)$ is called almost-Koszul if $A$ is a connected graded $R$-ring and $C$ is a compatible  connected graded $R$-coring. To an almost-Koszul pair one associates three chain complexes and three cochain complexes such that one of them is exact if and only if the others are so. In this situation $(A,C)$ is said to be Koszul. One proves that a connected $R$-ring $A$ is Koszul if and only if there is a connected $R$-coring $C$ such that $(A,C)$ is Koszul. This result allows us to investigate the Hochschild (co)homology of Koszul rings. We apply our method to show that the  twisted tensor product of two Koszul rings is Koszul. More examples and applications of Koszul pairs, including a generalization of Fr\"oberg Theorem \cite{Fr}, are discussed in the last part of the paper.
\end{abstract}

\subjclass[2010]{Primary 16E40; Secondary 16T10 and 16T15}
  \maketitle
  
\section*{Introduction}

  Koszul algebras were introduced by Priddy in \cite{Pr}. By definition,  the $\mathbb{N}$-graded algebra $A:=\oplus_{n\in\mathbb{N}}A^n$ over a field $\Bbbk$ is said to be (left) Koszul  if $A$ is connected, that is $A^{0}=\Bbbk$, and there is a resolution $P_{\ast}$ of $A^0$ by projective graded left $A$-modules such that each $P_{n}$ is generated by homogeneous elements of degree $n$. This class of algebras has outstanding applications in numerous fields of Mathematics, including Representation Theory, Algebraic Geometry, Algebraic Topology, Quantum Groups and Combinatorics; see \cite{PP} and the references therein. 

  Koszul algebras have been generalized by Beilinson, Ginzburg and Soergel. Following \cite{BGS}, we say that a graded ring $A$ is Koszul if $A^0$ is a semisimple ring and it has a resolution $P_\ast$ with the same properties as above. Many fundamental properties of Koszul algebras still hold in this more general setting. For instance, such a ring is always quadratic. The Koszulity of a ring is equivalent to the exactness of the Koszul complex. Moreover,  for any  Koszul ring $A$ such that $A^n$ is a finitely generated left $A^0$-module, the graded ring $A^{\#}:=\mathrm{Ext}_A^\ast(A^0,A^0)$ is Koszul as well and $(A^{\#})^{\#}=A$.  Here the functors $\mathrm{Ext_A^\ast(-,-)}$ are defined on the category of all left $A$-modules. The opposite ring $A^{op}$ of a left Koszul ring $A$ is left Koszul over $(A^0)^{op}$. Thus, a left Koszul ring is right Koszul, and conversely.

  The Hochschild cohomology of an algebra $A$ with coefficients in an $A$-bimodule $M$ is defined by the relation $\hh^\ast(A,M):=\ext_{A\ot A^{op}}(A,M)$. Hochschild cohomology may also be defined using the standard Hochschild complex \cite[Chapter 9.1]{We}.  Rather recently, a lot of effort has been paid to the explicit calculation  of these cohomology groups, see for example \cite{Ci1}, \cite{Ci2}, \cite{Ci3}, \cite{GG2}, \cite{ShW1}, \cite{ShW2} and \cite{Wa}. Although the standard complex is an important tool for the study of Hochschild cohomology, it is not helpful for computational purposes. In general, in order to compute Hochschild cohomology, ad hoc complexes are constructed, depending on the algebra that one works with. One of our main goals is to show that in the case of Koszul rings such complexes can be obtained using Coring Theory.  
  
  In order to explain our approach we need some terminology and notation. Let $R$ be a given ring and let $({}_R\mathrm{Mod}_R,\otimes,R)$ denote the tensor category of $R$-bimodules with respect to the tensor product $\ot$ of $R$-bimodules. Note that a graded ring $A:=\oplus_{n\in\mathbb{N}}A^n$ with $A^0=R$ may be seen as a connected graded algebra in this tensor category, and conversely. For short, we shall say that $A$ is a connected $R$-ring. A connected $R$-ring is said to be strongly graded if, in addition, $A^nA^m=A^{n+m}$, for any $n,m\in\mathbb{N}$. Connected and strongly graded $R$-corings can be defined by duality, as coalgebras in the tensor category $_R\mathrm{Mod}_R$. Since we work with graded structures, the multiplication of a ring $A$ and the comultiplication of a coring $C$ are uniquely determined by some $R$-bimodule morphisms $m^{p,q}:A^{p}\otimes A^{q}\rightarrow A^{p+q}$ and $\Delta_{p,q}:C_{p+q}\rightarrow C_{p}\otimes C_{q}.$

  We can now introduce \textit{almost-Koszul pairs}, one of the the main tools that we use for studying Koszul rings. By definition, such a pair consists of a connected $R$-ring $A$ and a connected $R$-coring $C$, together with an isomorphism of $R$-bimodules ${\theta_{C,A}}:C_1\to A^1$ which satisfies the relation 
\begin{equation}\label{ec:almost-koszul-intro}
  m^{1,1}\circ (\theta _{C,A}\otimes \theta _{C,A})\circ \Delta_{1,1}=0.
\end{equation}
  To every strongly graded $R$-ring $A$ corresponds a canonical almost-Koszul pair $\left(A, T(A)\right)$. By construction, the homogeneous component of degree $n$ of $T(A)$ is the $R$-bimodule $T_n(A):=\mathrm{Tor}_n^A(R,R)$. Note that $\mathrm{Tor}_\ast^A(R,R)$ is the homology of the chain $R$-coring $R\otimes_A \beta_\ast^{\,l}(A)$, where $\beta_\ast^{\,l}(A)$ denotes the normalized bar  resolution of $R$ regarded as a left $A$-module. Thus $T(A)$ has a natural connected $R$-coring structure. In this example, the fact that $A$ is strongly graded guarantees the existence of the $R$-bimodule isomorphism $\theta_{T(A),A}:T_1(A)\to A^1$. By duality, to every strongly graded $R$-coring $C$ corresponds an almost-Koszul pair $(E(C),C)$, where $E(C):=\mathrm{Ext}_C^\ast(R,R)$. Here, the functors $\mathrm{Ext}_C^\ast(-,-)$ are defined on the category of  right $C$-comodules.

  For each almost-Koszul pair $(A,C)$ we associate three chain complexes: $\mathrm{K}_\ast^l(A,C)$, $\mathrm{K}_\ast^r(A,C)$ and $\mathrm{K}_\ast(A,C)$. The first and the second are complexes of graded left and right $A$-modules, respectively. The third one lives in the category of graded $A$-bimodules. By duality, we also define three cochain complexes $\mathrm{K}^\ast_l(A,C)$, $\mathrm{K}^\ast_r(A,C)$ and $\mathrm{K}^\ast(A,C)$ in the categories of left, right and two-sided comodules over $C$. 

  By Theorem \ref{thm: Koszul}, all six complexes associated to an almost-Koszul pair $(A,C)$ are exact, provided that one of them is so. In this case we shall say that $(A,C)$ is a \textit{Koszul pair}. Note that, for any Koszul pair $(A,C)$, the complexes $\mathrm{K}_\ast^l(A,C)$ and $\mathrm{K}_\ast^r(A,C)$  provide projective resolutions of $R$ in the categories of  left and right graded $A$-modules, respectively. Similarly, $\mathrm{K}^\ast_l(A,C)$ and $\mathrm{K}^\ast_r(A,C)$ are injective resolutions of $R$ in suitable categories of $C$-comodules. Supposing in addition that $R$ is a separable algebra over a field $\Bbbk$, then $\mathrm{K}_\ast(A,C)$ is a resolution of $A$ by projective graded $A$-bimodules and $\mathrm{K}^\ast(A,C)$ is a resolution of $C$ by injective graded $C$-bicomodules.

  Some useful properties of Koszul pairs are investigated in the second section of the paper. In Theorem \ref{te:alg-echiv}  one shows that, for such a pair $(A,C)$, both components are strongly graded and, moreover, $(A,T(A))$ and $(E(C),C)$ are Koszul as well. Moreover, in this situation, it follows that $C$ and $T(A)$ are isomorphic as graded corings. The relationship between Koszul pairs and Koszul rings is explained in Theorem \ref{te:caracterizare_Koszul}: $A$ is such a ring if and only if there exists a Koszul pair $(A,C)$. Taking into account that the components of a Koszul pair uniquely determine each other, it is easy to see that
\[
  E(T(A))\cong A\text{\quad and\quad}T(E(C))\cong C,
\] 
  without any finiteness condition imposed on $A$ or $C$. These isomorphisms suggest that the {coring} $T(A)$ and the {ring} $E(C)$ may be thought of as  (Koszul) duals of $A$ and $C$, respectively. For example, the Koszul dual of a tensor $R$-ring $T^a_R(V)$ is the unique connected $R$-coring $C:=R\oplus V$, which is concentrated in degree $0$ and $1$. In the last part of the paper we compute the dual coring for other  Koszul $R$-rings, such as: trivial extension, multiparametric quantum spaces and quotients of quiver algebras by ideals generated by $2$-paths.
  
  On the other hand, in view of Theorem \ref{te:caracterizare_Koszul}, we say that  a strongly graded $R$-coring $C$ is Koszul if and only if $(E(C),C)$ is a  Koszul pair. With this definition in hand, it follows that the functors $T$ and $E$ preserve Koszulity. Multiparametric quantum spaces are Koszul both as a ring and a coring, cf. Theorem \ref{te: Koszul Pair} and \S\ref{fa:O_q}. Since Koszul corings might be useful for the study of other quantum groups, their properties will be investigated in a subsequent article.  

  We have already noticed that  $\mathrm{K}_{\ast }(A,C)$ is a resolution of $A$ as a projective (graded) bimodule over itself, for any separable algebra $R$ over a field $\bK$ and any Koszul pair $(A,C)$. In the third section we use this resolution to get a new (co)chain complex that computes the Hochschild (co)homology of $A$ with coefficients in an arbitrary bimodule. As an immediate corollary we show that, for any Koszul $R$-ring $A$, the projective dimension of $A$ in the category of $A$-bimodules (\textit{i.e.}  the Hochschild dimension of $A$) can be computed using the formula
  \begin{equation*}
    \mathrm{Hdim}(A)=\mathrm{sup}\{n\mid T_{n}(A)\neq 0\}.
  \end{equation*}%
  
  It is well known that the class of Koszul algebras is closed under twisted tensor products. As another  application of Koszul pairs we prove a similar result for Koszul rings. Let $A$ and $B$ be two strongly graded $R$-rings. To every invertible  graded twisting map $\sigma:B\otimes A\to A\otimes  B$ we associate an invertible graded twisting map  of $R$-corings $\tau:T(A)\otimes T(B)\to T(B)\otimes T(A) $. Thus the twisted tensor products $A\otimes_\sigma B$ and $T(A)\otimes_\tau T(B)$ make sense and, in fact, they always define an almost-Koszul pair $ (A\otimes_\sigma B,T(A)\otimes_\tau T(B))$. By Theorem~\ref{thm:koszul} this pair is Koszul, provided that $A$ and $B$ are Koszul $R$-rings. In particular, $A\ot_\sigma B$ is Koszul too. We have already mentioned that, for any Koszul pair $(A,C)$, the coring structure of $T(A)$ is captured on the fly by the isomorphism $T(A)\cong C$. In particular, for a twisted tensor product of two Koszul rings we get $T(A\ot_\sigma 
B)\cong T(A)\otimes_\tau T(B)$.  
	
  We next use the above Koszul pair to identify  the Hochschild (co)homology of a twisted tensor product $A\otimes_\sigma B$ with the total (co)homology of a certain double complex. In homology, the column-wise and row-wise filtrations of the double complex lead us to two spectral sequences converging to the Hochschild homology of $A\ot_\sigma B$, see Theorem \ref{thm:DoubleComplexHom}. Under some additional conditions, similar spectral sequences are obtained in cohomology. By definition generalized Ore extensions are examples of twisted tensor products. We specialize our results on Hochschild homology  to this more particular setting in the last part of the fifth section. 
	
  Our method based on coring techniques is also useful for the investigation of the Gerstenhaber structure of $\hh^\ast(A,A)$, in the case when $A$ is a  Koszul ring (for example the smashed product between a Koszul ring and a finite dimensional group algebra over a field of characteristic zero). Details about these results will be given in a sequel of this paper.
	
  Some more examples of Koszul pairs, related to braided bialgebras in the tensor category $_R\mathrm{Mod}_R$, are discussed in the last section. First of all, in Theorem  \ref{te: Koszul Pair} we prove that any couple of strongly graded braided commutative bialgebras in $_R\mathrm{Mod}_R$, under some conditions on their braidings, defines a Koszul pair. In particular we prove that any symmetric braided bialgebra in $_R\mathrm{Mod}_R$ is Koszul, provided that the braiding is an involution. The incidence algebra of the power set of a finite set is a nontrivial example of such bialgebras, cf.  Theorem \ref{thm:power_set}. For a different approach to Koszulity of (reduced) incidence algebras the reader is referred to \cite{RS}. As a last application, in Theorem \ref{thm:Froberg}, we extend  a result of Fr\"oberg \cite{Fr}.    
	
\section{Almost-Koszul pairs}

  In this section we introduce almost-Koszul pairs and we investigate their basic properties. We start by fixing the terminology and the notation that we use. Throughout, $R$ will denote a  semisimple ring.

\begin{fact}[$R$-rings.]\label{R-ring}
  The main objects that we work with are (co)unital and (co)associative (co)algebras in the tensor category of $R$-bimodules $(_{R}\mathcal{M}_{R},\ot,R)$. For the tensor product of two $R$-bimodules we shall always use the unadorned tensor product symbol $\otimes$. The unit object in $_{R}\mathcal{M}_{R}$ is $R$, regarded as a bimodule with respect to the left and right actions induced by the multiplication of $R$.  
  
  By definition, an $R$-\emph{ring} is an associative and unital algebra in $_{R}\mathcal{M}_{R}.$ Therefore an $R$-ring consists of an associative and unital ring $A$ together with a morphism of unital rings $u:R\rightarrow A.$  An $R$-ring $A$ is \emph{graded }if it is equipped with a decomposition $A=\textstyle\oplus_{n\in \mathbb{N}}A^{n}$ in $_{R}\mathcal{M}_{R},$ such that the multiplication $m:A\otimes A\rightarrow A$ maps $A^{p}\otimes A^{q}$ to $A^{p+q}.$ If  $A^{0}=R$, then we shall say that $A$ is \emph{connected}. The multiplication $m$ induces an $R$-bilinear map $m^{p,q}:A^{p}\otimes A^{q}\rightarrow A^{p+q},$ for all non-negative integers $p$ and $q$. The  $R$-ring $A$ is said to be \emph{strongly graded} if and only if it is connected and all maps $m^{p,q}$ are surjective. Obviously, $A$ is strongly graded if and only if $m^{p,1}$ is surjective, for all $p$.  The projection of $A$ onto $A^{n}$ will be denoted by $\pi _{A}^{n}$. 
    
  We denote the ideal $\oplus_{n>0}A^{n}$ by $\overline{A}$. The multiplication of $A$ induces a bimodule map $\overline{m}:\overline{A}\otimes \overline{A}\rightarrow \overline{A}$. Let $R^{op}$ be the opposite ring of $R$. If $V$ and $W$ are $R$-bimodules, then they become $R^{op}$-bimodules by interchanging their left and right module structures. The group morphism $\Lambda_{V,W}:V\otimes_{R^{op}}W \rightarrow W\otimes _{R}V$ that maps $v\otimes _{R^{op}}w$ to $w\otimes _{R}v$ is an isomorphism. If $(A,m,u)$ is an $R$-ring, then the multiplication and the unit of the \emph{opposite $R^{op}$-ring} $A^{op}$ are the maps $m^{op}:=m\circ \Lambda_{A,A}$ and $u$, respectively.
\end{fact}

\begin{fact}[$R$-corings.]
  An $R$-\emph{coring} is a coassociative and counital coalgebra in $_{R}\mathcal{M}_{R}$. Thus, an $R$-coring is an $R$-bimodule $C$ together with a coassociative \emph{comultiplication }$\Delta :C\rightarrow C\otimes C$ and a \emph{counit} $\varepsilon:C\rightarrow R,$ which are morphisms in $_{R}\mathcal{M}_{R}$.  An $R$-coring $(C,\Delta ,\varepsilon )$ is said to be \emph{graded} if, in addition, $C$ is the direct sum of a family $\{C_n\}_{n\in\mathbb{N}}$ of sub-bimodules, such that the counit vanishes on $C_m$, for any $m>0$, and 
  \[
    \Delta(C_{n})\subseteq \bigoplus_{p=0}^{n}C_{p}\otimes C_{n-p}. 
  \] 
  The comultiplication of a graded coring  is defined by a family of $R$-bilinear maps $\Delta_{p,q}:C_{p+q}\rightarrow C_{p}\otimes C_{q}.$ In the graded case we shall use a special form of Sweedler notation, namely $\Delta _{p,q}(c)=\sum\limits c_{(1,{p})}\otimes c_{(2,{q})}.$ Therefore, in a graded coring, coassociativity is equivalent to the following relations
\begin{equation}
  \sum\limits {c_{(1,p+q)}}_{(1,p)}\otimes {c_{(1,{p+q})}}_{(2,q)}\otimes c_{(2,{r})}=\sum\limits c_{(1,{p})}\otimes {c_{(2,{q+r})}}_{(1,{q})}\otimes {c_{(2,{q+r})}}_{(2,{r})},  \label{c1}
\end{equation}%
  where $p,q,r$ are arbitrary non-negative integers and $c\in C_{p+q+r}$. For short, we shall write the sums from equation  \eqref{c1} as $\sum\limits c_{(1,p)}\otimes c_{(2,q)}\otimes c_{(3,r)}$.  The counit satisfies the relations
\begin{equation}
  \sum\limits \varepsilon (c_{(1,{0})})c_{(2,{n})}=c=\sum\limits c_{(1,{n})}\varepsilon (c_{(2,{0})}),  \label{c2}
\end{equation}%
  for every $c\in C_{n}$. By definition, a graded $R$-coring $C$ is \emph{connected } if $C_{0}$ is isomorphic as an $R$-bimodule to $R$.  Thus, $\Delta_{0,0}$ and the restriction of $\varepsilon $ to $C_{0}$ are uniquely determined by
\begin{equation*}
  \Delta_{0,0}(c_0)=r_0 c_0\otimes c_0\quad\text{and}\quad\varepsilon(c_0)=r^{-1}_0,
\end{equation*}
  where $c_0$ is a certain element in $C_0$ that commutes with all $r\in R$, and $r_0$ is an invertible element in the center of $R$. In this paper we shall	always assume that $c_0$ is a group-like element, that is $r_0=1.$ For such a connected coring we shall identify $C_0$ with $R$ via the counit of $C$. In conclusion, without loss in generality, we may assume that $C_0=R$ and that the restriction of the counit to the zero degree component is the identity map of $R$. Of course, the counit vanishes on all other homogeneous components of $C$. It is easy to see that for any $c\in C_{n}$ we have
\begin{equation*}
  \Delta _{0,n}(c)=1\otimes c \quad\text{and}\quad\Delta _{n,0}(c)=c\otimes 1.
\end{equation*}
  For any graded  $R$-coring $(C,\Delta ,\varepsilon )$ the comultiplication $\Delta$ factors through a map $\overline{\Delta }:\overline{C}\rightarrow \overline{C}\otimes\overline{C}$, where $\overline{C}:=C/C_{0}$.  Note that $\overline{\Delta }$ is coassociative. Let $p_{C}:C\to\overline{C}$ and $\pi_n^C:C\to C_n$ denote the canonical projections. The comultiplication and the counit of the \emph{opposite $R^{op}$-coring} $C^{op}$ are the bimodule maps $\Delta ^{op}:=\Lambda_{C,C}^{-1}\circ \Delta $ and $\varepsilon$, respectively.
  
  For a connected  $R$-coring $C$ one defines the maps $\Delta (n):C_{n}\rightarrow C_{1}^{\otimes n}$ by setting $\Delta (1):={\Id}_{C_{1}}$ and then using the recursive relation 
\begin{equation}\label{ec:D(n)}
  \Delta (n)=\left({\Id}_{C_{1}}\otimes \Delta (n-1)\right)\circ \Delta_{1,n-1}.
\end{equation}%
  A graded coring $C$ is said to be \emph{strongly graded} if and only if it is connected and $\Delta(n)$ is injective for all $n>0$. By induction, it follows that $\Delta(p+q)=(\Delta (p)\otimes \Delta (q))\circ \Delta _{p,q}$, hence $C$ is strongly graded if and only if $\Delta _{p,q}$ is injective for all $p$ and $q$. To check that a coring is strongly graded it is enough to prove that $\Delta _{1,n}$ is injective for all $n$ or, equivalently, that $\Delta _{n,1}$ is injective for all $n$.    
\end{fact}

\begin{fact}[Almost-Koszul pairs.]\label{de:quasi-Koszul}
  We can now introduce one of the main tools that we shall use to investigate Koszul rings. An \emph{almost-Koszul pair} $(A,C)$ consists of  a connected $R$-ring $A$ and  a connected $R$-coring $C$, together with an  isomorphism of $R$-bimodules $\theta_{C,A}:C_{1}\rightarrow A^{1}$ that satisfies the relation  \eqref{ec:almost-koszul-intro}. 
  Using Sweedler notation this is equivalent to the fact that,  for any $c\in C_{2}$,  
\begin{equation}\label{ec:almost-koszul}
  \sum\limits\theta _{C,A}(c_{(1,{1})})\theta _{C,A}(c_{(2,{1})})=0.
\end{equation}
\end{fact}

\begin{remark}\label{re:Opposite _pair}
  Let $(A,C)$ be an almost-Koszul pair. We have noticed that $A^{op}$ and $C^{op}$ are an $R^{op}$-ring and an $R^{op}$-coring, respectively. Obviously, $(A^{op},C^{op})$ is an almost-Koszul pair over the ring $R^{op}$, with respect to $\theta_{C^{op},A^{op}}:=\theta_{C,A}$, regarded as an $R^{op}$-bimodule map in the canonical way.
\end{remark}

\begin{fact}[The normalized bar resolution of $R$ (as a right $A$-module).]  \label{fa:bar_resolution}
  We now want to show that, for every strongly graded $R$-ring $A$, there is a canonical graded coring $C$ such that $(A,C)$ is almost-Koszul. By  \cite[p. 283]{We}, the groups $\tor_{\ast }^{A}(R,R)$ may be computed using the \emph{normalized right bar resolution} $\beta_{\ast}^{\,r}(A)$, that is the exact sequence
\begin{equation}\label{ec:bar-resolution}
  0\longleftarrow R\overset{\delta_{0}}{\longleftarrow }A \overset{\delta_1}{\longleftarrow}\overline{A}\otimes A\longleftarrow\cdots \longleftarrow \overline{A}^{\,\otimes n-1}\otimes A\overset{\delta_{n}}{\longleftarrow}\overline{A}^{\,\otimes n}\otimes A\longleftarrow \cdots,
\end{equation}%
  where $\delta_{0}=\pi _{A}^{0}$ and, for $n>0$, the arrows are given by
\begin{equation*}
  \delta_{n}(a_{1}\otimes \cdots \otimes a_{n}\otimes a_{n+1}):=\sum\limits\limits_{i=1}^{n}(-1)^{i}a_{1}\otimes \cdots \otimes a_{i}a_{i+1}\otimes \cdots \otimes a_{n+1}.
\end{equation*}%
  Note that, since $R$ is semisimple, $\overline{A}^{\,\otimes n}\otimes A$ is a projective right $A$-module. Hence $\tor_{\ast }^{A}(R,R)$ is the homology of the \emph{normalized bar complex }$(\Omega _{\ast }(A),\partial_{\ast })$
\begin{equation}\label{ec:Complex-Tor}
  0\longleftarrow \Omega _{0}(A)\overset{\partial _{1}}{\longleftarrow }\Omega_{1}(A) \longleftarrow\cdots \longleftarrow \Omega _{n-1}(A)\overset{\partial _{n}}{\longleftarrow}\Omega _{n}(A)\longleftarrow \cdots ,
\end{equation}%
  where $\Omega _{0}(A)=R$ and $\Omega _{n}(A):=\overline{A}^{\,\otimes n},$ for $n>0$. The differential map $\partial _{1}$ is zero, and $\partial _{n}$ is the restriction of $\delta_{n-1}$ to $\Omega_n(A)\subseteq \overline{A}^{\,\otimes n-1}\otimes A$. We shall use the notation
\begin{equation*}
  T_{n}(A):=\h_{n}(\Omega _{\ast }(A))\quad \text{and\quad }T(A):= \textstyle\bigoplus\limits_{n\in \mathbb{N}}T_{n}(A).
\end{equation*}%
  The homology class of $x\in \overline{A}$ in $T_{1}(A)$ will be denoted by $ [x]$.

  The \emph{normalized left bar resolution} is defined by $\beta{}_{\ast}^{\,l}(A):=\beta{}_{\ast }^{\,r}(A^{op})$. Note that $\beta{}_{n}^{\,l}(A)\cong A\otimes \overline{A}^{\,\otimes n}$.
\end{fact}

\begin{fact}[The $R$-corings $T_{R}^{c}(V)$ and $\Omega(A)$.]\label{fa:Tc}
  Let $V$ denote an $R$-bimodule. On $T_{R}^{c}(V):=\oplus_{n\in\mathbb{N}}V^{\otimes n}$ one defines a graded $R$-coring structure such that $\Delta _{p,q}$ is the isomorphism $V^{\otimes p+q}\cong V^{\otimes p}\otimes V^{\otimes q}.$ The counit of $T_{R}^{c}(V)$ is the projection onto $V^{\otimes 0}:=R$, the zero degree homogeneous component of $T_{R}^{c}(V)$.
\end{fact}

  The  coring $\Omega(A):=T_{R}^{c}(\overline{A})$ will play an important role in our work. Let us first show that $\Omega(A)$ is a chain coring, \emph{i.e.} it is a coalgebra in the tensor category of chain complexes of $R$-bimodules.
	
\begin{lemma}\label{le:T(A)}
  The comultiplication and the counit of  $\Omega(A)$ are chain maps. In particular, $T(A)$ is a connected $R$-coring, and $T_{1}(A)$ is the cokernel of $\overline{m}:\Omega _{2}(A)\rightarrow \Omega _{1}(A).$
\end{lemma}

\begin{proof}
  By definition, the $n$-chains set of $ \Omega_{\ast }(A)\otimes \Omega_{\ast }(A)$ is the $R$-bimodule $\bigoplus_{p=0}^{n}\Omega_{p}(A)\otimes \Omega_{n-p }(A)$, while the restriction of the differential $d_n$ to $\Omega_{p}(A)\otimes \Omega_{n-p}(A)$ is given by 
\begin{equation*}
  d_{n}=\partial_{p}\otimes \Id_{\Omega_{n-p}(A)}+(-1)^{p} \Id_{\Omega_{p}(A)} \otimes \partial_{n-p}.
\end{equation*}%
  Proceeding as in \cite[Chapter 1.1]{PP}  one shows that the comultiplication of $\Omega(A)$ is a chain map. Clearly the counit is a chain map, regarding  $R$  as a complex concentrated in degree zero. In conclusion, $\Omega (A)$ is a chain $R$-coring and $T(A)$ is a graded $R$-coring, as the homology of a chain coring always inherits a graded coring structure. Since $\Omega_0(A)=R$, the map $\partial_{1}$ is trivial and  $\partial_{2}=\overline{m}$, it follows that $T(A)$ is connected and $T_{1}(A)=\coker\overline{m}$. 
\end{proof}

\begin{proposition}\label{pr:Tor}
  If $A$ is a connected strongly graded $R$-ring, then $(A,T(A))$ is almost-Koszul.
\end{proposition}

\begin{proof}
  We have proved that $T(A)$ is connected. Since $A$ is strongly graded, $T_{1}(A)=\overline{A}/\overline{A}^{\,2}=A^1$ and the projection $\pi _{A}^{1}$ induces an $R$-bimodule isomorphism $\theta _{T(A),A}:T_{1}(A)\to A^{1}.$ Every $\omega \in T_{2}(A)$ is the homology class of a certain $\zeta \in \Ker\overline{m} $. Hence $\zeta=\sum_{i=1}^{n}x_{i}\otimes y_{i},$ for some $x_{1},\dots ,x_{n}$ and $y_{1},\dots ,y_{n}$ in $\overline{A}$ that satisfy the equation $\sum_{i=1}^{n}x_{i}y_{i}=0.$ Note that  $x_{i}y_{i}-\pi _{A}^{1}\left( x_{i}\right) \pi_{A}^{1}\left( y_{i}\right) $ belongs to $\sum_{n>2}A^{n}$ and
\begin{equation*}
  \Delta _{1,1}(\omega )=\sum\limits_{i=1}^{n}\left[ x_{i}\right]\otimes\left[ y_{i}\right].
\end{equation*}%
  Thus, the relation \eqref{ec:almost-koszul} follows by the following computation
\begin{equation*}
  \sum\limits\limits_{i=1}^{n}\theta _{T(A),A}(\left[ x_{i}\right])\theta_{T(A),A}(\left[y_{i}\right])=\sum\limits\limits_{i=1}^{n}\pi_{A}^{1}\left(x_{i}\right) \pi _{A}^{1}\left( y_{i}\right) =\pi _{A}^{2}(\sum\limits\limits_{i=1}^{n}x_{i}y_{i}) =0.\qedhere
\end{equation*}
\end{proof}
  Our goal now is to associate to an almost-Koszul pair $(A,C)$ three cochain complexes: one in the category of graded left $C$-comodules and, symmetrically, one in the category of graded right $C$-comodules. By combining these constructions, we shall get the third cochain complex, that lives in the category of graded $C$-bicomodules.

\begin{fact}[The categories $\mathfrak{M}^{C}$, $^{C}\mathfrak{M}$ and $^{C}\mathfrak{M}^{C}$.]
  Let $C$ be an $R$-coring. The pair $(M,\rho ^{M})$ is a right $C$-\emph{comodule} if $M$ is a right $R$-module and $\rho^{M}:M\rightarrow M\otimes C $ is a morphisms of right $R$-modules such that, using the Sweedler notation $\rho ^{M}(m)=\sum\limits m_{\left\langle 0\right\rangle }\otimes m_{\left\langle 1\right\rangle }$, the relations below hold true for any $m\in M$
\begin{align*}
  \sum\limits {m_{\left\langle 0\right\rangle}} _{\left\langle  0\right\rangle}\otimes {m_{\left\langle 0\right\rangle}}_{\left\langle 1\right\rangle}\otimes m_{\left\langle 1\right\rangle } =\sum\limits m_{\left\langle 0\right\rangle }\otimes {m_{\left\langle 1\right\rangle}}_{(1)}\otimes {m_{\left\langle 1\right\rangle}}_{(2)} \quad\text{and}\quad 
  \sum\limits m_{\left\langle 0\right\rangle }\varepsilon \left(m_{\left\langle 1\right\rangle }\right)  =m.
\end{align*}%
  A morphism of comodules is a right $R$-linear map that commutes with the comodule structure maps. The category $\mathfrak{M}^{C}$ of right $C$-comodules is Grothendieck, as $C$ is flat as a left $R$-module \cite[p. 264]{Br}. 
	
  A right comodule $M$ is \emph{graded} if $M:=\textstyle\oplus_{n\in\mathbb{N}}M_{n}$, and $\rho ^{M}(M_{n})\subseteq\textstyle\oplus_{p=0}^{n}M_{p}\otimes C_{n-p}$, for all $n$. Then  $\rho^M$ is uniquely defined by the induced $R$-linear maps $\rho_{p,q}^{M}:M_{p+q}\rightarrow M_{p}\otimes C_{q}$. 
	
  The category $^{C}\mathfrak{M}$ of left $C$-comodules is constructed in a similar way. For a left $C$-comodule $(N,\rho ^{N})$ we use the Sweedler notation $\rho ^{N}(n)=\sum n_{\langle-1\rangle }\otimes n_{\langle 0\rangle }.$

  A $C$-\emph{bicomodule} is a triple $(M,\rho _{l}^{M},\rho_{r}^{M})$ such that $M$ is an $R$-bimodule, $(M,\rho _{l}^{M})$ is a left comodule, $(M,\rho _{r}^{M})$ is a right comodule and, for every $m\in M$,
\begin{equation*}
  \sum\limits m_{\left\langle -1\right\rangle }\otimes {m_{\left\langle 0\right\rangle}} _{\left\langle 0\right\rangle}\otimes {m_{\left\langle 0\right\rangle}}_{\left\langle 1\right\rangle }=\sum\limits {m_{\left\langle 0\right\rangle}} _{\left\langle -1\right\rangle}\otimes {m_{\left\langle 0\right\rangle}}_{\left\langle 0\right\rangle }\otimes m_{\left\langle 1\right\rangle }.
\end{equation*}%
  Note that, by definition, the structure maps $\rho _{l}^{M}$ and $\rho_{r}^{M}$ must be morphisms of $R$-bimodules, otherwise the above compatibility relation does not make sense. A morphism of $C$-bicomodules is a map which is left and right $C$-colinear. For the category of $C$-bicomodules we shall use the notation $^{C}\mathfrak{M}^{C}$.
\end{fact}

\begin{lemma}\label{le:injective}
  If $V$ is a right $R$-module, then $V\otimes C$ is an injective right $C$-comodule. A similar result holds for left $C$-comodules and $C$-bicomodules.
\end{lemma}

\begin{proof}
  For any $C$-comodule $(M,\rho ^{M})$ the natural transformation
\begin{equation}\label{ec:Theta_M}
  \Theta _{M,V}:\Hom_{R}(M,V)\rightarrow \Hom^{C}(M,V\otimes C),\quad\Theta_{M,V}(f):=(f\otimes {\Id}_{C})\circ \rho ^{M}
\end{equation}%
  is an isomorphism. The inverse of $\Theta_{M,V}$ maps  $g\in\Hom^C(M, V\otimes C)$ to $\left(\Id_V\otimes \varepsilon \right)\circ g.$ Let $U$ denote the functor that associates to a right $C$-comodule the underlying $R$-module structure, forgetting the coaction. Therefore the functors $\Hom^{C}(-,V\otimes C)$ and $\Hom_{R}(-,V)\circ U$ are isomorphic.  As $R$ is semisimple, $\Hom_{R}(-,V)\circ U$ is exact. Thus $\Hom^{C}(-,V\otimes C)$ is exact as well meaning that $V\ot C$ is injective.
  
  For any injective $R$-bimodule $W$ the functorial isomorphism
\begin{equation*}
  \Theta' _{M,W}:\Hom_{R,R}(M,W)\rightarrow \Hom^{C-C}(M,C\ot W\otimes C),\quad\Theta'_{M,W}(f):=(\Id_C\ot f\otimes {\Id}_{C})\circ (\rho_l^M\ot\Id_C)\circ  \rho_r ^{M}
\end{equation*}
	can be used as above to show that $C\ot W\ot C$ is injective as a bicomodule. 
\end{proof}

\begin{fact}[The complexes $\mathrm{K}_{l}^{\ast }(A,C)$ and $\mathrm{K}_{r}^{\ast }(A,C).$] \label{fa:K^*}
  Let $(A,C)$ be an almost-Koszul pair. We set
\begin{equation*}
  \K_{l}^{-1}(A,C):=R	\qquad\text{and}\qquad\K_{l}^{n}(A,C):=C\otimes A^{n},
\end{equation*}%
  for any $n\geq 0$. We regard $R$ as a graded left $C$-comodule with respect to the trivial coaction. On the other hand,  it is easy to see that $\Delta \otimes \Id_{A^{n}}$ defines a graded $C$-comodule structure on $\K_{l}^{n}(A,C)$, whose homogeneous component of degree $p$ is $C_{p-n}\otimes A^{n}$ (by convention, $C_{k}=0$ for any $k<0$). The differential maps $d_{l}^{n}:\K_{l}^{n}(A,C)\rightarrow \K_{l}^{n+1}(A,C)$ are defined as follows. If $n=-1$, then we take $d_l^{n}$ to be the canonical bimodule morphisms $R\to C\otimes A^0$ that maps $1$ to $1\ot 1\in C_0\otimes A^0$. For $n\geq 0$ the map $d_l^n$ vanishes on $C_{0}\otimes A^{n}$ and, if $p>0$ and $c\otimes a\in C_{p}\otimes A^{n}$, then
\begin{equation*}
  d_{l}^{n}(c\otimes a):=\sum\limits c_{(1,{p-1})}\otimes \theta_{C,A}(c_{(2,{1})})a.
\end{equation*}%
  Obviously, $d_{l}^{n}$ respects the gradings on $\K_{l}^{n}(A,C)$ and  $\K_{l}^{n+1}(A,C)$.

  Recall that $(A^{op},C^{op})$ is an almost-Koszul pair (over $R^{op}$), see Remark \ref{re:Opposite _pair}. Hence we may consider
\begin{equation*}
  \K_{r}^{\ast }(A,C):=\K_{l}^{\ast }(A^{op},C^{op}).
\end{equation*}%
  For any $n\geq 0$, in view of the isomorphism $C^{op}\otimes_{R^{op}}(A^{op})^{n}\cong A^{n}\otimes _{R}C,$ we identify $\K_{r}^{n}(A,C)$ and $A^{n}\otimes _{R}C$. Through this identification the  differential of $\K_{l}^{\ast }(A^{op},C^{op})$ in degree $n$ corresponds to the map $d_{r}^{n}:A^{n}\otimes C\rightarrow A^{n+1}\otimes C$ which is zero on $A^{n}\otimes C_{0}$ and, for $p>0$ and $a\otimes c\in A^{n}\otimes C_{p},$ is given by
\begin{equation*} 
  d_{r}^{n}(a\otimes c)=\sum\limits a\theta _{C,A}(c_{(1,{1})})\otimes c_{(2,{p-1})}.
\end{equation*}
  Note that $d_r^{-1}:R\to A^0\ot C$ is uniquely defined by the relation $d_r^{-1}(1)=1\ot 1$. 
\end{fact}

\begin{proposition}\label{pr:GradedComodules}
  If $(A,C)$ is an almost-Koszul pair, then $(\K_{l}^{\ast}(A,C),d_{l}^{\ast})$ and $(\K_{r}^{\ast}(A,C),d_{r}^{\ast})$ are cochain complexes of graded  $C$-comodules (left and right, respectively).
\end{proposition}

\begin{proof}
  Clearly $d_{l}^{0}\circ d_{l}^{-1}=0$, as $d_l^0$ vanishes on $C_0\ot A^0$. Let us show that $d_{l}^{n+1}\circ d_{l}^{n}(c\otimes a)=0,$ for any $n\geq 0$ and $c\otimes a\in C_{p}\otimes A^{n}$. We may assume that $p\geq 2,$ otherwise the relation is trivially satisfied, as $d_l^k(1\ot a)=0$, for any $a\in A$ and $k>0$. Let $\theta :=\theta _{C,A}$. Coassociativity and the relation  (\ref{ec:almost-koszul}) imply
\begin{align*}
  \left( d_{l}^{n+1}\circ d_{l}^{n}\right) (c\otimes a) & =\sum\limits {c_{(1,{p-1})}}_{(1,{p-2})}\otimes \theta ( {c_{(1,{p-1})}}_{(2,{1})}) \theta (c_{(2,{1})})a
	\\
  & =\sum\limits c_{(1,{p-2})}\otimes \theta ( {c_{(2,{2})}}_{(1,{1})}) \theta( {c_{(2,{2})}}_{(2,{1})}) a=0.
\end{align*}%
  Hence $\K_{l}^{\ast}(A,C)$ is a complex. Let us prove that the maps $d_l^n$ are morphisms of $C$-comodules. For $n=-1$ we have nothing to show. Let $n\geq 0$ and $p>0$. For  $c\otimes a\in C_{p}\otimes A^{n},$ we have
\begin{equation*}
  (\Id_C\otimes d_{l}^{n})(\rho (c\otimes a))=\sum\limits\limits_{r=0}^{p}c_{(1,{r})}\otimes d_{l}^{n}(c_{(2,{p-r})}\otimes a)=\sum\limits\limits_{r=0}^{p-1}\sum\limits c_{(1,{r})}\otimes {c_{(2,p-r)}}_{(1,p-r-1)}\otimes \theta ({c_{(2,{p-r})}}_{(2,1)}) a.
\end{equation*}%
  For the second identity we used that $\Delta_{p,0}(c)=c\ot 1$ and the fact  that $d_l^n$ vanishes on $C_0\ot A^n$. Since the comultiplication of $C$ is coassociative we get
\begin{equation*}
  \rho (d_{l}^{n}(c\otimes a))=\rho ( \sum\limits c_{(1,{p-1})}\otimes \theta (c_{(2,{1})})a)  =\sum\limits_{u=0}^{p-1}\sum {c_{(1,{p-1})}}_{(1,{u})}\otimes {c_{(1,{p-1})}}_{(2,{p-1-u})}\otimes \theta (c_{(2,{1})})a.
\end{equation*}%
  Thus $(\Id_C\otimes d_{l}^{n})(\rho (c\otimes a))=\rho (d_{l}^{n}(c\otimes a))$. We conclude the proof that $d_l^n$ is a morphism of $C$-comodules, by remarking that the above relation trivially holds for $c\ot a\in C_0\ot A^n$, as $d_l^n(c\ot a)=0$.
\end{proof}

\begin{fact}[The complex $\mathrm{K}^{\ast}(A,C).$]
  We are going to construct a cochain complex $(\K^{\ast }(A,C),d^{\ast })$ in the category of $C$-bicomodules. By definition, $\K^{-1}(A,C):=C$ and $\K^{n}(A,C)=C\otimes A^{n}\otimes C$. The differential maps are defined by the relations ${d}^{-1}:=\Delta$ and
\begin{equation}\label{ec:def d}
	d^{n}=d_{l}^{n}\otimes \Id_C+(-1)^{n+1}\Id_C\otimes d_{r}^{n},
\end{equation}%
  for $n\geq 0$. Since $d_{l}^{n}$ is left $C$-colinear and $d_{r}^{n}$ is right $C$-colinear it follows that $d^{n}$ is a morphism of  bicomodules. It is not difficult to show that $(d^0\circ d^{-1})(c)=0$, for any $c\in C_p$. Indeed, for $p=0$ the relation is obvious. Let us assume that $p>0$. Thus, $d^{-1}(c)=\sum_{i=0}^p c_{(1,i)}\ot c_{(2,p-i)}$. Henceforth, 
\[
  (d^0\circ d^{-1})(c)=\sum_{i=1}^p c_{(1,i-1)}\ot\theta_{C,A}(c_{(2,1)})\ot c_{(1,p-i)} -\sum_{i=0}^{p-1} c_{(1,i)}\ot\theta_{C,A}(c_{(2,1)})\ot c_{(1,p-i-1)}=0.
\]
Let $n\geq 0$. Since $\K_l^\ast(A,C)$ and $\K_r^\ast(A,C)$ are complexes, that is $d_{l}^{n+1}\circ d_{l}^{n}=d_{r}^{n+1}\circ d_{r}^{n}=0,$ we get
\begin{equation*}
  d^{n+1}\circ d^{n}=(-1)^{n+1}\left[ \left( d_{l}^{n+1}\otimes \Id_C\right)\circ\left( \Id_C\otimes d_{r}^{n}\right) -\left( \Id_C\otimes d_{r}^{n+1}\right)\circ \left( d_{l}^{n}\otimes \Id_C\right) \right].
\end{equation*}%
  Using the formulae that define $d_{l}^{\ast }$ and $d_{r}^{\ast }$ and the fact that the multiplication in $A$ is associative, it follows that $d^{n+1}\circ d^{n}=0.$ Hence $(\K^{\ast }(A,C),d^{\ast })$ is a complex. 
\end{fact}

\begin{fact}[The subcomplexes ${\mathrm{K}}_{l}^{\ast }(A,C,m)$ and ${\mathrm{K}}_{r}^{\ast }(A,C,m)$.] \label{fa:K(A,C,m)}
  The complex $\K_{l}^{\ast }(A,C)$ decomposes as a direct sum of subcomplexes in the category of $R$-bimodules  $\oplus_{m\in\mathbb{N}}\K_{l}^{\ast }(A,C,m)$. By definition, $\K_{l}^{\ast }(A,C,0)$ is the complex
\begin{equation*}
  0\longrightarrow R\overset{\cong }{\longrightarrow }C_{0}\otimes A^{0}\longrightarrow 0.
\end{equation*}%
  By construction, the elements of $R$ are cochains of degree $-1$. Note that $\K_l^\ast(A,C,0)$ is always exact. If $m>0$, then  $d_{l}^{n}$ maps $C_{m-n}\otimes A^{n}$ to $C_{m-n-1}\otimes A^{n+1}.$ Therefore, $\K_{l}^{n}(A,C,m):=C_{m-n}\otimes A^{n}$ define a subcomplex $\K_{l}^{\ast}(A,C,m)$ of $\K_{l}^{\ast }(A,C)$, which can be displayed as follows
\begin{equation*}
  0\longrightarrow C_{m}\otimes A^{0}\longrightarrow \cdots \longrightarrow C_{m-n}\otimes A^{n}\overset{d_{l}^{n}}{\longrightarrow }C_{m-n-1}\otimes A^{n+1}\longrightarrow \cdots \longrightarrow C_{0}\otimes A^{m}\longrightarrow 0.
\end{equation*}%
  Obviously, $\K^{\ast }_{l}(A,C)=\textstyle\oplus_{m\geq 0}\K_{l}^{\ast }(A,C,m).$ The complex $\K_{r}^{\ast }(A,C)$ admits a similar decomposition as a direct sum of subcomplexes $\oplus_{m\geq 0}\K_{r}^{\ast }(A,C,m)$.
\end{fact}

\begin{fact}[The normalized bar resolution of $R$ (as a right $C$-comodule).]
  We are going to sketch how the preceding constructions and results can be dualized. Let $C$ be an $R$-coring. We assume that $C$ is connected. Then $R$ is a right $C$-comodule with respect to the trivial coaction. Recall that $\overline{C}:=C/C_{0}$ and that $\overline{\Delta}:\overline{C}\rightarrow \overline{C}\otimes \overline{C}$ is the unique map such that $\overline{\Delta}\circ p_{C}=\left(p_{C}\otimes p_{C}\right)\circ \Delta $, where $p_{C}$ is the canonical projection. We also use the notation $\widetilde{\Delta }:=(p_{C}\otimes {\Id}_{C})\circ \Delta.$

  The \emph{right normalized bar resolution} $\beta{}_{r}^{\,\ast}(C)$ of $R$ is the exact sequence of right $C$-comodules
\begin{equation}\label{ec:Bar-C-Norm}
  0\longrightarrow R\overset{\delta{}^{\,-1}}{\longrightarrow } C\overset{\delta{}^{\,0}}{\longrightarrow }\overline{C}\otimes C\longrightarrow \cdots \longrightarrow\overline{C}^{\otimes n}\otimes C\overset{\delta{}^{\,n}}{\longrightarrow }\overline{C}^{\,\otimes n+1}\otimes C\longrightarrow \cdots,
\end{equation}%
  where $\delta^{-1}$ is the canonical inclusion and, for $n\geq 0$,
\[
  \delta{}^{\,n}=\sum_{i=1}^{n}(-1)^{i-1}{\Id}_{\overline{C}^{\otimes i-1}}\otimes \overline{\Delta}\otimes{\Id}_{\overline{C}^{\otimes n-i}\otimes  C}+(-1)^{n}{\Id}_{\overline{C}^{\otimes n}}\otimes\widetilde{\Delta }.
\]
  The normalized resolution $\beta{}_{l}^{\,\ast }(C)$ of $R$ in the category of left $C$-comodules is defined in a similar way.

We use $\beta{}_{r}^{\,\ast }(C)$ to compute $\ext_{C}^{\ast}(R,R).$ Applying the functor $\Hom^{C}(R,-)$ and using the isomorphisms $\Theta _{R,-},$ we obtain the \emph{normalized bar complex} $(\Omega^{\ast }(C),\partial^{\ast })$
\begin{equation}\label{ec:Bar-C-Complex}
  0\longrightarrow \Omega ^{0}(C)\overset{\partial^{0}}{\longrightarrow }\Omega ^{1}(C)\overset{\partial^{1}}{\longrightarrow }\cdots \longrightarrow\Omega ^{n}(C)\overset{\partial^{n}}{\longrightarrow }\Omega^{n+1}(C)\longrightarrow \cdots ,
\end{equation}%
  where $\Omega ^{0}(C)=R$ and  $\Omega ^{n}(C):=\overline{C}^{\, \otimes n}$, for each positive $n$. The differential morphisms are defined by the formulae  $\partial^{0}=0$ and, for $n>0$, 
\[
  \partial^{n}=\sum_{i=1}^{n}(-1)^{i-1}{\Id}_{\overline{C}^{\,	\otimes i-1}}\otimes \overline{\Delta}\otimes {\Id}_{\overline{C}^{\,\otimes n-i}},
\]
  We shall use the notation  $E^{n}(C):=\h^n(\Omega^\ast(C))$ and $E(C):=\textstyle\oplus_{n\in\mathbb{N}}E^{n}(C).$  Thus $E^n(C)\cong\ext_C^n(R,R)$.
\end{fact}

\begin{fact}[The $R$-rings $T_{R}^{a}(V)$ and $\Omega(C)$.]\label{fa:Ta}
  Let $V$  be an $R$-bimodule and let $T_{R}^{a}(V)$ denote the free $R$-ring generated by $V$. Therefore, $T_{R}^{a}(V):=\oplus_{n\in\mathbb{N}}V^{\otimes n}$ and the graded ring structure is defined by the bimodule isomorphisms $m^{p,q}:V^{\otimes p}\otimes V^{\otimes q}\to V^{\otimes p+q}.$ The unit of $R$ is a unit for $T_{R}^{a}(V)$ too. In particular, to every connected coring $C$ we associate a connected $R$-ring $\Omega(C):=T_{R}^{a}(\overline{C})$.  
\end{fact}
\begin{lemma}\label{le:E(C)}
  The multiplication and the unit of $\Omega (C)$ are cochain maps. Thus $E(C):=\ext_{C}^{\ast }(R,R)$ is a connected $R$-ring and $E^{1}(C)=\Ker\overline{\Delta}.$
\end{lemma}

\begin{proof}
  As in the ordinary case of algebras over a commutative ring, one  shows that the multiplication is a morphism of cochain complexes from  $\Omega^\ast(C)\otimes \Omega^\ast(C)$ to $\Omega^\ast(C)$, where the set of $n$-cochains of the former complex is $\textstyle\oplus_{p=0}^{n}\Omega^{p}(C)\otimes \Omega^{n-p}(C)$ and the differential $d^n$ on $\Omega^{p}(C)\otimes \Omega^{n-p}(C)$ is given by
\begin{equation*}
  d^{n}=\partial^{p}\otimes {\Id}_{\Omega^{n-p}(C)}+(-1)^{p}{\Id}_{\Omega^{p}(C)}\otimes \partial^{n-p}.
\end{equation*}%
  Clearly, the unit  of $\Omega (C)$ is a morphism of cochain complexes, where $R$ is regarded as a complex concentrated in degree zero. Since the cohomology of a cochain $R$-ring inherits a canonical graded $R$-ring structure, $\partial^{0}=0$ and $\partial^{1}=\overline{\Delta}$, it follows that $E(C)$ is a connected graded $R$-ring and $E^{1}(C)=\Ker\overline{\Delta}.$
\end{proof}

\begin{proposition}\label{pr:Ext}
  Let $C$ be a strongly graded $R$-coring. Then $(E(C),C)$ is an almost-Koszul pair.
\end{proposition}

\begin{proof}
  We already know that $E(C)$ is a connected $R$-ring. Let $\overline{C}_{n}:=p_C(C_{n}).$ Hence,  $\overline{C}=\textstyle \oplus_{n>0}\overline{C}_{n}$ and $p_C$ is injective on each component of positive degree. In particular, $C_{n}\cong \overline{C}_{n}.$

  Let $\theta :C_{1}\rightarrow \overline{C}_{1}$ denote the restriction of $ p_{C}$ to $C_{1}.$ Since $\Delta _{0,1}(c)=1\otimes c$ and $\Delta_{1,0}(c)=c\otimes 1$ it follows that the image of $\theta $ is included into $\Ker\overline{\Delta }=E^{1}(C),$ so we may regard $\theta $ as a map from $C_1$ to $E^{1}(C).$ We claim that the pair $(E(C),C)$ satisfies the identity (\ref{ec:almost-koszul}) with respect to $\theta_{C,E(C)}:=\theta$. If $B^{2}(C)$ is the group of $2$-coboundaries in the normalized bar complex $\Omega ^{\ast }(C)$ and $c\in C_{2},$ then
\begin{equation*}
  \sum\limits \theta (c_{(1,{1})})\cdot \theta (c_{(2,{1})}^{{}})=\sum\limits p_{C}(c_{(1,{1})})\otimes p_{C}(c_{(2,{1})})+B^{2}(C)=\overline{\Delta }(p_{C}(c))+B^{2}(C)=\partial^{1}(p_{C}(c))+B^{2}(C)=0.
\end{equation*}%
  Note that the first equality is a consequence of the definition of  multiplication in $E(C),$ while for the second identity we used the definition of $\overline{\Delta }$ and the relations $\Delta_{0,2}(c)=1\otimes c$ and $\Delta _{2,0}(c)=c\otimes 1.$ It remains to prove that $\theta $ is bijective. Let $p_{C}(c)$ be an element in $\Ker\overline{\Delta }$ and let $c=\sum_{n=1}^{d}c_{n}$, where $c_{n}\in C_{n}$. We claim that $c_{n}=0$ for $n\geq 2.$ We have
\begin{equation*}
  \sum_{n=2}^{d}\sum\limits_{r=1}^{n-1}(p_{C}\otimes p_{C})(\Delta_{r,n-r}(c_{n}))=\sum\limits_{n=0}^{d}\sum\limits_{r=0}^{n}(p_{C}\otimes p_{C})(\Delta _{r,n-r}(c_{n}))=(p_{C}\otimes p_{C})\left( \Delta \left(c\right) \right) =\overline{\Delta }(p_{C}(c))=0.
\end{equation*}%
  We fix $n\geq 2.$ Since $(p_{C}\otimes p_{C})(\Delta _{r,n-r}(c_{n}))$ belongs to $\overline{C}_{r}\otimes \overline{C}_{n-r},$ in view of the foregoing remarks, we deduce that this element vanishes for every $0<r<n.$ Since $p_{C}\otimes p_{C}$ is an $R$-bimodule isomorphism between $C_{r}\otimes C_{n-r}$ and $\overline{C}_{r}\otimes \overline{C}_{n-r}$ it follows that $\Delta _{r,n-r}(c_{n})=0.$ As $C$ is strongly graded, we get that the kernel of $\overline{\Delta }$ is included into $\overline{C}_{1}.$ The other inclusion is obvious, so $\theta $ is an isomorphism.
\end{proof}

\begin{fact}[The cotensor product.]\label{cot}
	Recall that the cotensor product $N\square _{C}M$ between a right $C$-comodule $(N,\rho ^{N})$ and a left $C$-comodule $(M,\rho ^{M})$ is defined as the kernel of $\rho ^{N}\otimes {\Id}_{M}-{\Id}_{N}\otimes \rho ^{M}$.

	For any right $R$-module $V$ the tensor product $V\otimes C$ is a right $C$-comodule with respect to the coaction given by ${\Id_V} \otimes \Delta$. If  $M$ is a left $C$-comodule, then 
\begin{equation}\label{ec:Izo-Cotensor}
	\zeta :(V\otimes C)\square _{C}M\rightarrow V\otimes M,\quad \zeta(\sum\limits_{i=1}^{n}v_{i}\otimes c_{i}\otimes m_{i})=\sum\limits_{i=1}^{n}v_{i}\otimes \varepsilon (c_{i})m_{i}
\end{equation}%
	is an isomorphism. The inverse of $\zeta$ maps $v\otimes m$ to $\sum\limits v\otimes m_{\left\langle -1\right\rangle }\otimes m_{\left\langle  0\right\rangle }.$ In particular, for $V=R$, we get that $C\square_{C}M$ and $M$ are isomorphic as $R$-bimodules.

	In the case when $C$ is connected, we know that $R$ is a left and a right $C$-comodule with respect to the trivial coactions. For such a coring $C,$ there is a canonical isomorphism $R\,\square_{C}M\cong M^{coC},$ where the set of coinvariant elements $M^{coC}$ contains all $m\in M$ such that $\rho _{M}(m)=1\otimes m.$ Of course, the above discussion applies for right $C$-comodules as well.
\end{fact}

\begin{fact}[The complexes $\K_{\ast }^{l}(A,C),$ $\K_{\ast }^{r}(A,C)$ and $ \K_{\ast }(A,C).$]
	For any almost-Koszul pair $(A,C)$, by duality, we can also construct three chain complexes. We first define a complex $\K_{\ast }^{r}(A,C)$ of graded right $A$-modules. Let $\K_{-1}^{r}(A,C):=R$ and $\K_{n}^{r}(A,C):=C_{n}\otimes A$. The differential $d^r_0:C_0\ot A \to R$ maps $x\ot a$ to $x\pi_A^0(a)$. For $n>0$ one defines $d_{n}^{r}:C_{n}\otimes A\rightarrow C_{n-1}\otimes A$  by
\begin{equation*}
	d_{n}^{r}(c\otimes a)=\sum\limits c_{(1,{n-1})}\otimes \theta_{C,A}(c_{(2,{1})})a.
\end{equation*}%
	To show that $\K_\ast^r(A,C)$ is a complex one proceeds as in \S\ref{fa:K^*}, so we omit the proof.  Applying the previous construction to the opposite almost-Koszul pair $(A^{op},C^{op})$ we obtain a complex $(\K_{\ast }^{l}(A,C),d_{\ast }^{l})$ of graded left $A$-modules. Explicitly, $\K_{-1}^{l}(A,C):=R$ and $\K_{n}^{l}(A,C):=A\otimes C_{n}$. The differential $d_{0}^l$ maps $a\ot x$ to $\pi_A^0(a)x$. For $n > 0$ the maps $d_n^l$ are given by
\begin{equation*}
	d_{n}^{l}(a\otimes c)=\sum\limits a\theta _{C,A}(c_{(1,{1})})\otimes c_{(2,{n-1})}.
\end{equation*}%
	By combining the above two constructions we obtain a new complex $(\K_{\ast}(A,C),{d}_{\ast })$ in  the category  of $A$-bimodules. By definition, $\K_{-1}(A,C):=A$ and $\K_{n}(A,C):=A\otimes C_{n}\otimes A$. The map $d_0$ is induced by the multiplication of $A$ and, for $n>0$, we have
\begin{equation*}
	{d}_{n}=d_{n}^{l}\otimes \Id_A+(-1)^{n}\Id_A\otimes d_{n}^{r}.
\end{equation*}
\end{fact}

\begin{fact}[The subcomplexes $\K_{\ast }^{r}(A,C,m)$ and $\K_{\ast}^{l}(A,C,m)$.]\label{K(A,C,m)}
	The complex $\K_{\ast }^{r}(A,C)$ can be written as a direct sum of subcomplexes $\textstyle\oplus_{m\in\mathbb{N}} \K_{\ast}^{r}(A,C,m)$. In the case when $m=0$ we take $\K_{\ast}^{r}(A,C,0)$ to be the exact cochain complex
\begin{equation*}
	0\longrightarrow R\overset{\cong }{\longrightarrow }C_{0}\otimes A^{0}\longrightarrow 0.
\end{equation*}%
	If $m>0$, then $\K_{n}^{r}(A,C,m):=C_{m-n}\otimes A^{n}.$ Therefore in this case $\K_{\ast }^{r}(A,C,m)$ is the complex
\begin{equation*}
	0\longleftarrow C_{0}\otimes A^{m}\longleftarrow \cdots \longleftarrow C_{n-1}\otimes A^{m-n+1}\overset{d_{n}^{r}}{\longleftarrow }C_{n}\otimes A^{m-n}\longleftarrow \cdots\longleftarrow C_{m}\otimes A^{0}\longleftarrow 0. 
\end{equation*}%
	Analogously, $\K_{\ast }^{l}(A,C)$ decomposes as a direct sum of  subcomplexes $\K_{\ast }^{l}(A,C)\cong\oplus_{m\in\mathbb{N}}\K_{\ast}^{l}(A,C,m)$.
\end{fact}

\begin{proposition}\label{pr:Inv-K(A,C)}
	Let $(A,C)$ be an almost-Koszul pair. Then the following statements are true.
\begin{enumerate}
	\item There are canonical isomorphisms $\K_{\ast }(A,C)\otimes _{A}R\cong\K_{\ast }^{l}(A,C)$ and $R\otimes _{A}\K_{\ast }(A,C)\cong\K_{\ast}^{r}(A,C)$. 

	\item There are canonical isomorphisms $\K^{\ast }(A,C)\,\square _{C}R\cong\K_{l}^{\ast }(A,C)$ and $R\,\square_{C}\K^{\ast }(A,C)\cong\K_{r}^{\ast}(A,C)$. 
\end{enumerate}
\end{proposition}

\begin{proof}
	We shall only prove the isomorphisms for $\K_\ast^l(A,C)$ and $\K^\ast_l(A,C)$. The other two isomorphisms can be proved in a similar way.
	
	The relation ${\xi }_{-1}(r)=1\otimes_{A}r$ defines a left $A$-module bijective map ${\xi }_{-1}:R\to A\otimes _{A}R$. For $n\geq 0$, let $\xi _{n}:\K_{n}^{l}(A,C)\to \K_{n}(A,C)\otimes _{A}R$ be the left $A$-module isomorphism $\xi _{n}(a\otimes c)=\left( a\otimes c\otimes 1\right) \otimes _{A}1. $ One can prove easily that $(d_0\ot_A\Id_R)\circ\xi_0=\xi_{-1}\circ d^l_0$. Furthermore, for $n>0$,
\begin{align*}
	\left[ \left( d_{n}\otimes _{A}\Id_R\right) \circ \xi _{n}\right](a\otimes c)=&\sum\limits \left( a\theta _{C,A}(c_{(1,{1})})\otimes c_{(2,{n-1})}\otimes 1\right) \otimes _{A}1+
			\\
	& +(-1)^{n}\sum\limits \left( a\otimes c_{(1,{n-1})}\otimes \theta_{C,A}(c_{(2,{1})})\right) \otimes _{A}1. 
\end{align*}%
	Since the left $A$-module structure of $R$ is induced by the algebra morphism $\pi _{A}^{0}:A\rightarrow R$ and $\theta _{C,A}(c_{(2,{1})})$ is an element in $A^{1},$ it follows that the  second sum in the right-hand side of the above relation is zero. Hence $\xi _{\ast }$ is a morphism of chain complexes between $\K_{\ast }^{l}(A,C)$ and $\K_{\ast }(A,C)\otimes _{A}R$, as 
\begin{equation*}
	\left[ \left( d_{n}\otimes _{A}\Id_R\right) \circ \xi _{n}\right] (a\otimes c)=\sum\limits \left( a\theta _{C,A}(c_{(1,{1})})\otimes c_{(2,{n-1})}\otimes 1\right) \otimes _{A}1=\left( \xi _{n-1}\circ d_{n}^{l}\right)(a\otimes c).
\end{equation*}
	
	In order to prove the second part of the proposition, we define $\xi^{-1}:R\to C\,\square_C R$ by ${\xi}{}^{\,-1}(r)=1\otimes r$. If $n\geq 0$, then $\xi ^{n}:\K_{l}^{n}(A,C)\rightarrow \K^{n}(A,C)\square _{C}R$ is given by $\xi ^{n}\left( c\otimes a\right) =(c\otimes a\otimes 1)\otimes 1.$ Clearly, $\xi ^{n}$ is a morphism of left $C$-comodules and it is bijective by \S \ref{cot}. To check that $\xi_\ast$ is a morphism of complexes we first notice that
\[
	(\xi^0\circ d^{-1}_l)(1)=(1\ot 1\ot 1)\ot 1=[(d^{-1}\,\square_C \Id_R)\circ\xi^0](1).
\]
	Let $n\geq 0$. Since $d_{r}^{n}(a\otimes 1)=0$, for any $a\in A^n$, we conclude the proof by the computation below
\begin{align*}
	\left[ \left( d^{n}\square _{C}\Id_R\right) \circ \xi ^{n}\right]\left(c\otimes a\right)& =\left[ d_{l}^{n}(c\otimes a)\otimes 1\right] \otimes 1+(-1)^{n+1}\left[ c\otimes d_{r}^{n}(a\otimes 1)\right] \otimes 1.\\
	&=\left[ d_{l}^{n}(c\otimes a)\otimes 1\right] \otimes 1=\left( \xi ^{n+1}\circ d_{l}^{n}\right) (c\otimes a).\qedhere
\end{align*}
\end{proof}
Let us regard the graded $R$-bimodule $A^\ast$ as a cochain complex with trivial differential maps. Let $R\longrightarrow A^\ast $ denote the augmented complex with respect to the identity of $R=A^0$ (we regard $R$ as the component of degree $-1$). The coaugmented complex $R\longleftarrow C_\ast$ is defined similarly.
\begin{proposition}\label{pr:KisoC}
	Let $(A,C)$ be an almost-Koszul pair.

\begin{enumerate}
	\item The complexes $R\otimes _{A}\K_{\ast }^{l}(A,C)$ and $\K_{\ast}^{r}(A,C)\otimes _{A}R$ are isomorphic to $R\longleftarrow C_\ast$.

	\item The complexes $\Hom^{C}(R,\K_{l}^{\ast }(A,C))$ and $\Hom^{C}(R,\K_{r}^{\ast }(A,C))$ are isomorphic to $R\longrightarrow A^\ast $. 
\end{enumerate}
\end{proposition}

\begin{proof}         
	Let $\psi_{-1}:R\to R\ot_A R$ be the canonical isomorphism. For $n\geq 0$ we define the isomorphism $\psi _{n}:C_{n}\rightarrow R\otimes _{A}\K_{n}^{l}(A,C)$ by $\psi _{n}(c)=1\otimes _{A}(1\otimes c).$ Since  $\psi_{-1}=(\Id_R\ot_A d_0^l)\circ\psi_0$ and for $n>0$ we have 
\begin{equation*}
	\left( (\Id_{A}\otimes d_{n}^{l})\circ \psi _{n}\right) (c)=1\otimes_{A}d_{n}^{l}(1\otimes c)=\sum\limits \pi_{A}^{0}\left( \theta _{C,A}(c_{{(1,{1})}})\right) \otimes_{A}(1\otimes c_{{(2,{n-1})}})=0,
\end{equation*}%
	it follows that $\psi _{\ast }$ is a morphism of complexes.
	
	Let $\psi^{-1}:R\to \Hom^C(R,R)$ be the bijective map $\psi^{-1} (r)=f_r$, where $f_r(x)=xr$. For $n\geq 0$ we define $\psi ^{\ast }:A^{n}\rightarrow \Hom^{C}(R,C \otimes A^{n}))$ by $\psi ^{n}(a)=f_{a},$ where $f_{a}$ is the unique $C$-comodule map such that $f_{a}(1)=1\otimes a.$ Obviously $\psi_n$ is bijective. It is easy to see that $\psi^{0}= \Hom^C(R,d^{-1}_l)\circ\psi^{-1}$. By the definition of $d_{l}^{n}$ we have $d_{l}^{n}(1\otimes a )=0$.  Thus $\Hom^{C}(R,d_{l}^{n})\circ \psi ^{n}=0,$ for any $n\geq 0$. In conclusion, $\psi^{\ast}$ is an isomorphisms of complexes.
	
	The other two isomorphisms can be proved in a similar way. 
\end{proof}

Let $(A,C)$ be an almost-Koszul pair. Our next goal is to compare the complexes $\K_{\ast}^{l}(A,C)$ and $\K_{r}^{\ast }(A,C)$ with the bar resolutions $\beta{}_{\ast }^{\,l}(A)$ and $\beta{}_{r}^{\,\ast }(C)$, respectively. We start by introducing some notation. For $n>0$, let $\overline{\pi}_{n}^{\,C}:\overline{C}\to C_{n}$ be the map induced by the projection $\pi _{n}^{C}:C\to C_{n}$ and let $\overline{\theta}:=\theta_{C,A}\circ\overline{\pi}^C_1$.
If $x$ is a class in $\overline{C}$ then we introduce the Sweedler notation $\overline{\Delta}(x)=\sum x_{(1)}\ot x_{(2)}.$  Note that, if $x=p_C(c)$ then $\sum x_{(1)}\ot x_{(2)}=\sum p_C(c_{(1)})\ot  p_C(c_{(2)}).$

To relate $\beta{}_{r}^{\,\ast}(C)$ and $(\K_{r}^{\ast}(A,C), (-1)^\ast d^\ast)$ we construct  the maps $\phi ^{n}:\beta^n_r(C)\to \K^n_r(A,C)$, for any $n\geq -1$. By definition, $\phi_{-1}$  and $\phi_0$  are $\Id_R$ and the canonical isomorphism $C\cong A^{0}\otimes C$, respectively. Furthermore, for $n>0$, we define $\phi^n:\overline{C}^{\,\ot n}\ot C\to A^n\ot C$ by
\begin{equation*}
	\phi ^{n}(x^1\otimes\cdots\otimes x^n\otimes c)=\overline{\theta} (x^1)\overline{\theta} (x^2) \cdots \overline{\theta}(x^n)\otimes c.
\end{equation*}%

For every $n\geq -1$ we now construct a map $\phi_n:\K_n^l(A,C)\to\beta_n^l(A)$, as follows. First, we set $\phi _{-1}:={\Id}_{R}$ and we take $\phi _{0}$ to be the canonical isomorphism $A\otimes C_{0}\cong A.$ Then, for $n>0$, we define  $\phi _{n}$ by
\begin{equation*}
	\phi _{n}(a\otimes c):=\sum\limits a\otimes\theta_{C,A} (c_{(1,{1})})\otimes\theta_{C,A} (c_{(2,{1})})\otimes \cdots \otimes \theta_{C,A}(c_{(n,{1})}),
\end{equation*}%
where in the above equation we used the Sweedler notation $\Delta(n)(c)=\sum c_{(1,{1})}\otimes c_{(2,{1})}\otimes \cdots \otimes c_{(n,{1})}.$

\begin{proposition}\label{pr:Psi}\label{pr:phi}
	If  $(A,C)$ is an almost-Koszul pair, then $\phi ^{\ast} :\beta{}_{r}^{\,\ast }(C)\to(\K_{r}^{\ast }(A,C),(-1)^{\ast }{d}{}_{r}^{\,\ast })$ and $\phi _{\ast }:\K_{\ast}^{l}(A,C)\to \beta{}_{\ast }^{\,l}(A)$
	 are morphisms of complexes.
\end{proposition}

\begin{proof}
	We need the relation $\phi^{n+1}\circ\delta{}^{\,n}=(-1)^{n}{d}{}_{r}^{\,n}\circ \phi ^{n}$, for all $n\geq -1.$ Let $x=p_C(c)$ be some class in $\overline{C}$. Since $\overline{\theta}\circ p_C=\theta\circ \pi_1^C$ and $\Delta \left( c\right) =\sum_{u,v\geq 0}c_{(1,{u})}\otimes c_{(2,{v})}$, we get
\begin{equation}
	\sum \overline{\theta}(x_{(1)}) \overline{\theta}(x_{(2)}) = \sum\limits_{u,v\geq 0}{\theta}\left(\pi_1^C(c_{(1,{u})})\right){\theta}\left(\pi_1^C(c_{(2,{v})})\right) =\sum\limits \theta(c_{(1,{1})})\theta (c_{(2,{1})})=0.  \label{theta'}
\end{equation}%
	We can now prove the claimed relation. Let ${x}^{\,1},\dots ,x^n$ be elements in $\overline{C}$. If $p>0$ and $c\in C_{p}$, then
\begin{align*}
	(\phi ^{n+1}\circ {\delta }^{n})({x}^{\,1}\otimes \cdots \otimes {x}^{\,n}\otimes c)={\sum\limits_{i=1}^{n}} \sum\limits & (-1)^{i-1}\overline{\theta}(x^{1})\cdots \overline{\theta}(x^{i-1})\overline{\theta}(x_{(1)}^{i})\overline{\theta}(x_{(2)}^{i})\overline{\theta}(x^{i+1})\cdots \overline{\theta}(x^{n})\otimes {c}{}\; +
			\\
	+&(-1)^{n}\sum\limits \overline{\theta}(x^{1})\cdots\overline{\theta}(x^{n})\overline{\theta}\left(p_C(c_{(1)})\right)\otimes {c}_{(2)}.
\end{align*}%
	The double sum is zero, by \eqref{theta'}. As $\Delta \left( c\right)=\sum_{u=0}^{p}\sum c_{(1,{u})}\otimes c_{(2,{p-u})}, $ it follows that
\begin{equation*}
	\sum\limits\overline{\theta}\left(p_C(c_{(1)})\right)\otimes {c}_{(2)}=\sum\limits \theta(c_{(1,{1})})\otimes c_{(2,{p-1})}.
\end{equation*}%
	Using this relation we get
\begin{equation*}
	(-1)^{n}({d}{}_{r}^{\,n}\circ \phi ^{n})({x}^{\,1}\otimes \cdots\otimes {x}^{\,n}\otimes c)=(-1)^{n}\sum\limits\overline{\theta}(x^{1})\cdots \overline{\theta}(x^{n})\overline{\theta}\left (p_C(c_{(1)})\right)\otimes c_{(2)}.
\end{equation*}%
	To complete the proof of the fact that $\phi^\ast$ is a morphism we remark that both $(-1)^{n}{d}{}_{r}^{\,n}\circ \phi ^{n}$ and $\phi^{n+1}\circ {\delta }{}^{\,n}$ vanish on $\overline{C}^{\,\ot n}\ot C_0$. 

	Let us now show that $\phi_\ast$ is a morphism of complexes. We have to prove that $\phi _{n-1}\circ {d}{}_{n}^{\,l}=\delta_{n}\circ \phi _{n}.$ For $n=0$ this relation is obvious. Let $a\otimes c$ be a tensor monomial in $A\otimes C_n$. In the case when $n=1$, both sides of the equation map $a\otimes c$ to $a\theta (c).$ Let us  assume that $n>1$. By \eqref{ec:D(n)}, we get 
\begin{equation*}
	(\phi _{n-1}\circ {d}{}_{n}^{\,l})(a\otimes c)=\sum\limits a\theta (c_{(1,{1})})\otimes \theta (c_{(2,{1})})\otimes \cdots \otimes \theta (c_{(n,{1})}).
\end{equation*}%
	On the other hand, by the definition of $\phi _{n}$ and $\delta_{n}$, we have
\begin{align*}
	(\delta_{n}\circ \phi _{n})(a\otimes c)=\sum\limits \;&a\theta (c_{(1,{1})})\otimes \theta (c_{(2,{1})})\otimes \cdots \otimes \theta (c_{(n,{1})})+
			\\
	& +\sum\limits_{i=1}^{n-1}\sum\limits \left( -1\right) ^{i}a\otimes \theta(c_{(1,{1})})\otimes \cdots \otimes \theta (c_{(i,{1})})\theta (c_{(i+1,{1})})\otimes \cdots \otimes \theta (c_{(n,{1})}).
\end{align*}
	By coassociativity and using the condition (\ref{ec:almost-koszul}), it follows that the double sum vanishes.
\end{proof}

\section{Koszul pairs}

	In this section we shall investigate the exactness of the (co)chain complexes that we have associated to an almost-Koszul pair $(A,C)$. Roughly speaking, we shall show that one of these complexes is exact if and only if the other five are so. Furthermore, these complexes are exact precisely when  $A$ is a Koszul $R$-ring.

\begin{lemma}\label{le:cog1}
	Let $(A,C)$ be an almost-Koszul pair.

\begin{enumerate}
	\item If $A$ is strongly graded, then $\h_{0}(\K_{\ast }(A,C))=0.$

	\item If $C$ is strongly graded, then $\h^{0}(\K^{\ast}  (A,C))=0.$
\end{enumerate}
\end{lemma}

\begin{proof}
	Let $\theta :=\theta _{C,A}.$ Since $C$ is connected we have  $A\otimes C_{0}\otimes A\cong A\otimes A.$ Via this isomorphism, one identifies ${d}_{0}$ with the multiplication $m$ of the $R $-ring $A.$ Moreover,
\begin{equation*}
	{d}_{1}(a\otimes c\otimes b)=a\theta (c)\otimes b-a\otimes \theta (c)b,
\end{equation*}%
	for any $a,b\in A$ and $c\in C_{1}.$ Let us assume that $z=\sum_{i=1}^{d}a_{i}\otimes b_{i}$ is an element in the kernel of $m.$ We have to show that $z$ is in the image of ${d}_{1}.$ Since $A$ is strongly graded and $\theta $ is bijective we may assume that $b_{i}=\theta (c_{1}^{i})\cdots \theta (c_{n_{i}}^{i}),$ for some $c_{1}^{i},\dots ,c_{n_{i}}^{i}\in C_{1}.$ It is easy to see that $z={d}_{1}(x)$, where
\begin{equation*}
	x:= -\sum\limits_{i=1}^{d}\sum\limits_{j=1}^{n_{i}}a_{i}\theta (c_{1}^{i})\cdots \theta (c_{j-1}^{i})\otimes c_{j}^{i}\otimes \theta (c_{j+1}^{i})\cdots \theta (c_{n_{i}}^{i}).
\end{equation*}
	Let us now prove the second part of the lemma. Let $\nu :C\otimes C\rightarrow C\otimes C_{1}\otimes C$ be the unique bimodule morphism such that
\begin{equation*}
	\nu |_{C_{p}\otimes C_{q}}=\Delta _{p-1,1}\otimes \Id _{C_{q}}-\Id_{C_{p}}\otimes \Delta_{1,q-1},
\end{equation*}%
	where $\Delta _{u,v}=0$ whenever $u<0$ or $v<0.$ Let us consider  the following diagram
\begin{equation*}
	\xymatrixcolsep{3pc}\xymatrix{ 0 \ar[r]^-{}& C\ar[r]^-{{d}^{-1}} \ar@{=}[d]& C\otimes A^0\otimes C \ar[r]^-{{d}^{0}}& C \otimes A^1\otimes C\\ 0 \ar[r]^-{}& C\ar[r]_-{\Delta} & C\otimes C \ar[r]_-{\nu}\ar[u]|-{\cong}& C\otimes C_1\otimes C\ar[u]|-{\Id_C\otimes\theta_{C,A}\otimes \Id_C} }
\end{equation*}%
	Since the squares are commutative and the vertical arrows are isomorphisms it is enough to prove that $\Ker\nu =\im\Delta .$ We claim that the sequence
\begin{equation*}
	C\overset{\Delta }{\longrightarrow }C\otimes C\overset{\mu}{\longrightarrow}C\otimes C\otimes C
\end{equation*}%
	is exact, where $\mu =\Delta \otimes {\Id}_{C}-{\Id}_{C}\otimes \Delta.$ Indeed, the inclusion $\im\Delta \subseteq \Ker\mu $ is trivial as the comultiplication is coassociative. Let  $s_{-1}:={\Id}_{C}\otimes \varepsilon $ and $s_{0}:={\Id}_{C}\otimes {\Id}_{C}\otimes \varepsilon $. Since
\begin{equation*}
	-s_{0}\mu +\Delta s_{-1}={\Id}_{C\otimes C},
\end{equation*}%
	any element in the kernel of $\mu $ belongs to the image of $\Delta $, so our claim has been proved. Summarizing, we must show that $\Ker\mu =\Ker\nu.$ Let $x$ be an element in $C\otimes C.$ There are finitely many non-zero elements $x_{p,q}\in C_{p}\otimes C_{q}$ such that $x=\sum_{p,q\geq 0}x_{p,q}.$ Note that $C\otimes C\otimes C=\textstyle\bigoplus_{u,v,w\geq 0}C_{u}\otimes C_{v}\otimes C_{w}.$ Computing the component of $\mu (x)$ in $C_{u}\otimes C_{v}\otimes C_{w}$, we deduce that $x$ is in the kernel of $\mu $ if and only if
\begin{equation}\label{ker(mu)}
	(\Delta _{u,v}\otimes \Id_{C_{w}})(x_{u+v,w})=({\Id}_{C_{u}}\otimes \Delta _{v,w})(x_{u,v+w}),
\end{equation}%
	for any $u,v,w\geq 0.$ Proceeding similarly, we deduce that $x$ is in the kernel of $\nu $ if and only if \eqref{ker(mu)} holds for $v=1$ and any $u,w\geq 0.$ Thus the inclusion $\Ker\mu \subseteq \Ker\nu $ is trivial.

	To prove the other inclusion we pick $x\in\Ker \nu$. Since $C$ is connected the equation (\ref{ker(mu)}) holds for $v=0$ and any  $u,w\geq 0.$ As $x$ is in the kernel of $\nu$, it remains to prove \eqref{ker(mu)} for any $v\geq 2$. By induction, let us assume that (\ref{ker(mu)}) holds for some $v\geq 1.$ By coassociativity and the induction hypothesis we have
\begin{align*}
	[({\Id}_{C_{u}}\otimes \Delta _{1,v}\otimes {\Id}_{C_{w}})\circ(\Delta_{u,v+1}\otimes {\Id}_{C_{w}})](x_{u+v+1,w}) & =[(\Delta _{u,1}\otimes {\Id}_{C_{v}}\otimes  {\Id}_{C_{w}})\circ(\Delta _{u+1,v}\otimes {\Id}_{C_{w}})](x_{u+v+1,w})
			\\
	& =[(\Delta _{u,1}\otimes {\Id}_{C_{v}}\otimes {\Id}_{C_{w}})\circ({\Id}_{C_{u+1}}\otimes \Delta _{v,w})](x_{u+1,v+w})
			\\
	& =(\Delta_{u,1}\otimes \Delta _{v,w})(x_{u+1,v+w}).
\end{align*}%
	On the other hand, using coassociativity once again and then the fact that $x\in \Ker\nu $, we get
\begin{align*}
	[({\Id}_{C_{u}}\otimes \Delta _{1,v}\otimes {\Id}_{C_{w}})\circ({\Id}_{C_{u}}\otimes\Delta_{v+1, w})](x_{u,v+1+w}) & =[({\Id}_{C_{u}}\otimes {\Id}_{C_{1}}\otimes \Delta _{v,w})\circ(\Id_{C_{u}}\otimes\Delta_{1,v+w})](x_{u,v+w+1})
			\\
	& =(\Delta_{u,1}\otimes \Delta _{v,w})(x_{u+1,v+w}).
\end{align*}%
	We conclude the proof of the lemma by remarking that ${\Id}_{C_{u}}\otimes \Delta _{1,v}\otimes {\Id}_{C_{w}}$ is  injective, since $C$ is strongly graded and any $R$-bimodule is flat as a left and a right module.
\end{proof}

\begin{proposition}\label{prop:Koszulcomplexes}\label{pr:K=exact}
	Let $(A,C)$ be an almost-Koszul pair.

\begin{enumerate}
	\item The complexes $\K_{\ast }^{l}(A,C)$, $\K_{\ast }^{r}(A,C)$ and  $\K_{\ast }(A,C)$ are exact, provided that one of them is so.

	\item  The complexes $\K_{l}^{\ast }(A,C)$, $\K_{r}^{\ast }(A,C)$ and $\K^{\ast }(A,C)$ are exact, provided that one of them is so.
\end{enumerate}
\end{proposition}

\begin{proof}
	Assuming that $\K_{\ast }(A,C)$ is exact we deduce that it splits  in the category of right $A$-modules, as $\K_{n}(A,C)$ is projective in this category for any $n\geq -1.$ By Proposition \ref{pr:Inv-K(A,C)} (1) we have
\begin{equation*}
	\K_{\ast }^{\,l}(A,C)\cong \K_{\ast }(A,C)\otimes _{A}R,
\end{equation*}%
	so this complex is exact too. Let us now suppose that $\K_{\ast}^l(A,C)$ is exact. We first remark that, by construction, ${d}_{0}^{\,l}$ is zero on $A^{n}\otimes C_{0} $ for any $n>0.$ Moreover, for $c\in C_{1}$ and $a\in A$ we have
\begin{equation*}
	{d}_{1}^{\,l}(a\otimes c)=a\theta_{C,A} (c)\otimes 1.
\end{equation*}%
	Therefore the exactness of $\K_{\ast }^{l}(A,C)$ in degree $0$ and the fact that $\theta_{C,A}$ is bijective imply the relation $A^{n}=A^{n-1}A^{1}, $ for all $n>0.$ Hence $A$ is strongly graded. In view of the preceding lemma we get that $\K_{\ast}(A,C)$ is exact in degree $0.$

	Since ${d}_{0}$ is surjective it remains to show that $\h_{n}(\K_{\ast }(A,C))=0,$ for all $n>0.$  For short, let $\K_{\ast}:=\K_{\ast }(A,C).$ For $n,p\geq 0$ we define $X_{n,p}:=A^{p-n}\otimes C_{n},$ where $A^{p-n}=0$ in the case when $p<n$. Thus $\K_{n}:=\textstyle\oplus_{p\geq 0}X_{n,p}\otimes A$. Obviously,
\begin{equation}\label{inc}
	(d_{n}^{l}\otimes {\Id}_{A})(X_{n,p}\otimes A)\subseteq X_{n-1,p}\otimes A\quad \text{and\quad }({\Id}_{A}\otimes d_{n}^{r})(X_{n,p}\otimes A)\subseteq X_{n-1,p-1}\otimes A.
\end{equation}%
	Taking into account the above inclusions we get that $\K_{\ast }^{i}:=\textstyle\oplus_{p=0}^{i}X_{\ast,p}\otimes A$ is a subcomplex of $\K_{\ast }$. Let $L_{\ast }^{i}:=\K_{\ast}^{i}/\K_{\ast }^{i-1}$. For any $k\geq 0$ we can identify $L_{k}^{i}$ and $A^{i-k}\otimes C_{k}\otimes A$ as $R$-bimodules. By (\ref{inc}) it follows that $L_{\ast }^{i}$ is isomorphic to $\K_{\ast }^{l}(A,C,i)\otimes A,$ see \S\ref{K(A,C,m)} for the definition of the latter complex. We deduce that $L_{\ast }^{i}$ is exact, being isomorphic to a direct summand of the exact complex $\K_{\ast}^{l}(A,C)\otimes A.$ We now fix $n>0.$ Note that $\K_{n}^{0}=0$. In particular, $\K_{\ast}^{0} $ is exact in degree $n.$ Writing the long exact sequence in homology for
\begin{equation*}
	0\longrightarrow \K_{\ast }^{i-1}\longrightarrow \K_{\ast}^{i}\longrightarrow L_{\ast }^{i}\longrightarrow 0
\end{equation*}%
	we deduce by induction that $\K_{\ast }^{i}$ is exact for any $i\geq 0.$ We can now prove that $\K_{\ast }$ is exact in degree $n.$ Indeed, if $\omega $ is an $n$-cycle in $\K_{\ast }$, then there is $i$ such that $\omega \in \K_{n}^{i}$. Since $\omega$ is a cycle in $\K_{\ast }^{i} $ it follows that $\omega $ is a boundary in this complex. Thus, a fortiori, $\omega $ is a boundary in $\K_{\ast }$, so $\K_{\ast}(A,C)$ is exact. 
	
	To show that $\K_{\ast }(A,C)$ is exact if and only if $\K_{\ast}^{r}(A,C) $ is exact we proceed as follows. By definition, the latter sequence is equal to $\K_{\ast }^{l}(A^{op},C^{op})$ which, in turn, is exact if and only if $\K_{\ast }(A^{op},C^{op})$ is so. Let $d_{\ast }^{op}$ denote the differential of $\K_{\ast}(A^{op},C^{op}).$ For every $n\geq 0,$ the map
\begin{equation*}
	\eta _{n}:A^{op}\otimes _{R^{op}}\left(C^{op}\right)_{n}\otimes _{R^{op}}A^{op}\rightarrow A\otimes _{R}C_{n}\otimes _{R}A,\quad \eta_{n}\left( a\otimes _{R^{op}}c\otimes_{R^{op}}b\right) =b\otimes_{R}c\otimes _{R}a
\end{equation*}%
	is bijective and $\eta _{n-1}\circ d_{n}^{op}=(-1)^{n}d_{n}\circ\eta _{n}. $ In conclusion, the homology of $\K_{\ast }(A^{op},C^{op})$ is trivial if and only if the homology groups of $\K_{\ast }(A,C)$ vanish as well. 
	
	We begin the proof of the second part of the proposition by showing that $\K_{l}^{\ast }(A,C)$ is exact, provided that $\K^{\ast}(A,C)$ satisfies this property. The latter complex splits in the category of right $C$-comodules, as it is exact and $\K^n(A,C)$ is injective as a right comodule for any $n$. In view of Proposition \ref{pr:Inv-K(A,C)} (2) we conclude that $\K_{l}^{\ast }(A,C)$ is exact.

	Let us now assume that $\K_{l}^{\ast }(A,C)$ is exact. Since the comultiplication of a coring is always injective, this property holds in degree $-1$. 
	
	Since $A$ is connected we identify $\K_{l}^{0}(A,C)$ with $C$. Via this identification,  $d_{l}^{0}$ is a map from $C$ to $C\ot A^0$ which vanishes on  $C_{0}$. On the other hand, for $p>0$ and $c\in C_p$, we have $d_{l}^{0}(c)=\sum c_{(1,p-1)}\ot \theta_{C,A}(c_{(2,1)})$. Since $\theta_{C,A}$ is an isomorphism it follows
\begin{equation*}
	\h^0(\K_{l}^{\ast }(A,C))\cong\textstyle\bigoplus\limits_{p\geq 1}\Ker\Delta_{p-1,1}.
\end{equation*}
	By assumption we get that $C$ is strongly graded, so $\K^{\ast}(A,C)$ is exact in degree zero, cf. Lemma \ref{le:cog1} (2).

	It remains to show that $\K^{\ast }(A,C)$ is exact in degree $n>0$. For the sake of simplicity, let us denote this complex by $(\K^{\ast },d^{\ast }).$ If $X_{n,p}:=C_{p-n}\otimes A^{n},$ then $\K^{n}=\textstyle \oplus_{p\geq 0}X_{n,p}\otimes C.$ Moreover,
\begin{equation}\label{ec:Prop Xnp}
	(d_{l}^{n}\otimes {\Id}_{C})(X_{n,p}\otimes C)\subseteq X_{n+1,p}\otimes C\quad \text{and\quad }\left( {\Id}_{C}\otimes d_{r}^{n}\right) \left(X_{n,p}\otimes C\right) \subseteq X_{n+1,p+1}\otimes C.
\end{equation}%
	In particular, $\K_{i}^{\ast }$ is a subcomplex of $\K^{\ast },$ where $\K_{i}^{n}:=\oplus_{p\geq i}X_{n,p}\otimes C.$ Let $L_i^\ast:=\K_{i}^{\ast}/\K_{i+1}^{\ast}$.

	Note that $X_{n+1,i+1}\otimes C$ is a subset of $\K_{i+1}^{n+1}$ and the differential of $L_{i}^{\ast }$ maps an $n$-chain $x\otimes c+\K_{i+1}^{n}\ $ to $d_{l}^{n}(x)\otimes c+\K_{i+1}^{n+1}.$ Thus the $R$-bilinear isomorphism 
\begin{equation*}
	L_{i}^{n}\cong \K_{l}^{n}(A,C,i)\otimes C
\end{equation*}%
	allows us to identify the complexes $L_{i}^{\ast }$ and $\K_{l}^{\ast}(A,C,i)\otimes C$. Since $\K_{l}^{\ast }(A,C,i)$ is a direct summand of $\K_{l}^{\ast }(A,C)$ and $C$ is flat as a left $R$-module we conclude that $L_{i}^{\ast }$ is exact.
	Our goal now is to show that all quotients $\K^{\ast}/\K_{i}^{\ast }$ are exact. We proceed by induction. Clearly $\K^{\ast }/\K_{0}^{\ast }=0.$ Let us assume that $\K^{\ast }/\K_{i}^{\ast }$ is exact. Since in the short exact sequence
\begin{equation*}
	0\longrightarrow L_{i}^{\ast }\longrightarrow \K^{\ast }/\K_{i+1}^{\ast}\longrightarrow \K^{\ast }/\K_{i}^{\ast}\longrightarrow 0
\end{equation*}%
	$L_{i}^{\ast }$ is also exact we deduce that the middle term has the same property. We can now prove that $\K^{\ast }$ is exact. Let $\omega $ be an $n $-cocycle in $\K^{\ast }.$ We choose a positive integer $i$ such that $\omega $ belongs to $M_{i}^{n}:=\textstyle\bigoplus_{p\leq i}X_{n,p}\otimes C.$ The projection induces an $R$-bimodule isomorphism $\nu ^{n}:\K^{n}/\K_{i+1}^{n}\rightarrow M_{i}^{n}$. Using $\nu ^{\ast }$ one transports the differential maps of ${\K^{{\ast}}}/\K_{i+1}^{\ast }$ to get a cochain complex $(M_{i}^{\ast },\partial ^{\ast })$ which is isomorphic to ${\K^{{\ast }}}/\K_{i+1}^{\ast }.$ Clearly, $M_{i}^{\ast }$ is  exact and $\omega $ is a cocycle in $M_{i}^{\ast }.$ Hence, there is $\zeta \in M_{i}^{n-1}$ such that $\omega =\partial ^{n}(\zeta )$. On the other hand $\partial ^{n}(\zeta )=d^{n}(\zeta )$ as, by construction, $\partial ^{n}$ and $d^{n}$ are equal on  $M_{i}^{n}.$ Thus the proposition is proved.
\end{proof}

\begin{theorem}\label{thm: Koszul}
If one of the six complexes from the preceding proposition is exact, then all are so.
\end{theorem}

\begin{proof}
	Recall that for the complexes $\K_{l}^{\ast }(A,C)$ and $\K_{\ast}^{r}(A,C)$ we have the decompositions
\begin{equation*}
	\K_{l}^{\ast}(A,C)=\textstyle\bigoplus\limits_{m\in\mathbb{N}} \K_{l}^{\ast}(A,C,m)
			\qquad \text{and}\qquad
	\K_{\ast }^{r}(A,C)=\textstyle\bigoplus\limits_{m\in\mathbb{N}}  \K_{\ast}^{r}(A,C,m).
\end{equation*}%
	Let $m>0$. By construction, $\K_{l}^{p}(A,C,m)=\K_{m-p}^{r}(A,C,m)$ and $d_{l}^{p}=d_{m-p}^{r}$, for any integer $p$. Thus
\begin{equation*}
	\h^{p}(\K_{l}^{\ast }(A,C,m))=\h_{m-p}(\K_{\ast }^{r}(A,C,m)).
\end{equation*}%
	Since, by definition, both $\K_{l}^{\ast }(A,C,0)$ and $\K_{\ast}^{r}(A,C,0)$ are always exact, it follows that the cohomology groups of $\K_{l}^{\ast }(A,C)$ vanish if and only if the homology groups of $\K_{\ast}^{r}(A,C)$ vanish as well. We conclude the proof of the theorem by applying Proposition \ref{pr:K=exact}.
\end{proof}

\begin{definition}\label{de: pair}
	An almost-Koszul pair $(A,C)$ is said to be \emph{Koszul} if and only if the complexes from Proposition \ref{pr:K=exact} are exact.
\end{definition}

\begin{corollary}\label{co:resolutions}
	Let $(A,C)$ be a Koszul pair. Then the following statements hold:

\begin{enumerate}
	\item The complex $\K_{l}^{\ast }(A,C)$ is a resolution of $R$ by injective graded left ${C}$-comodules.

	\item	The complex $\K_{r}^{\ast }(A,C)$ is a resolution of $R$ by injective graded right ${C}$-comodules.

	\item The complex $\K_{\ast }^{l}(A,C)$ is a resolution of $R$ by projective graded left ${A}$-modules.

	\item The complex $\K_{\ast }^{r}(A,C)$ is a resolution of $R$ by projective graded right ${A}$-modules.

	\item If $A$ is an injective $R$-bimodule, then $\K^{\ast} (A,C)$ is a resolution of $C$ by injective graded ${C}$-bicomodules.

	\item If $C$ is a projective $R$-bimodule, then $\K_{\ast}(A,C)$ is a resolution of $A$ by projective graded ${A}$-bimodules.
\end{enumerate}
\end{corollary}

\begin{proof}
	The first four statements are immediate consequences of the preceding theorem. Let us assume that $C$ is projective as an $R$-bimodule. Thus the sub-bimodule $C_n$ is also projective. It follows that $A\otimes C_{n}\otimes A$ is projective as an $A$-bimodule, so $\K_{\ast}(A,C)$ is a resolution of $A$ by projective $A$-bimodules. Analogously, by Lemma \ref{le:injective},  if $A$ is an injective $R$-bimodule then $C\otimes A^n\otimes C$ is an injective bicomodule. Therefore, the last statement holds as well.
\end{proof}

\begin{remark}
	Note that, although $R$ is semisimple, there might exist $R$-bimodules which are neither projective nor injective. However, if $R$ is a separable algebra over a field $\Bbbk$ (\emph{i.e.} $R$ is projective as a bimodule over itself), then every $R$-bimodule is both projective and injective. Henceforth, for a Koszul  pair $(A,C)$, under this additional assumption on $R $, the complex  $\K^{\ast}(A,C)$ is a resolution  of $C$ by injective $C$-bicomodules and $\K_{\ast}(A,C)$ is a resolution by projective $A$-bimodules. 
\end{remark}

\begin{corollary}\label{co:(A,T)}
	Let $(A,C)$ be a Koszul pair. Then the $R$-ring $A$ and the $R$-coring $C$ are strongly graded. In particular, $(A,T(A))$ and $(E(C),C)$ are almost-Koszul pairs.
\end{corollary}

\begin{proof}
	The complex $\K_{\ast }^{l}(A,C)$ is exact, as $(A,C)$ is Koszul. Hence, by the proof of Proposition \ref{pr:K=exact}, it follows that $A$ is strongly graded. By Proposition \ref{pr:Tor}, the pair $(A,T(A))$ is almost-Koszul. The statements concerning the coring $C$ can be proved in a similar way.
\end{proof}

\begin{corollary}
	The pair $(A,C)$ is Koszul if and only if $(A^{op},C^{op})$ is  Koszul too.
\end{corollary}

\begin{proof}
	The homology of $\K_{\ast }(A^{op},C^{op})$ is trivial if and only if the homology groups of $\K_{\ast }(A,C)$ vanishes, by the proof of Proposition \ref{pr:K=exact}.
\end{proof}

\begin{theorem}\label{te:alg-echiv}\label{te:C-Koszul}
	Let $(A,C)$ be a Koszul pair. The graded $R$-coring $T(A)$ is isomorphic to $C$ and $(A,T(A))$ is a Koszul pair. Dually, $E(C)$ and $A$ are isomorphic as graded $R$-rings and $(E(C),C) $ is Koszul.
\end{theorem}

\begin{proof}
	We know that $A$ is strongly graded and that $(A,T(A))$ is almost-Koszul. Let $\phi _{\ast }$ be the morphism of complexes from $\K_{\ast }^{l}(A,C)$ to $\beta{}_{\ast }^{\,l}(A),$ that we constructed in Proposition \ref{pr:phi}. Since $\K_{\ast }^{l}(A,C)$ is a projective resolution of $R$ and $\phi_{-1}=\Id_R$ is invertible, there exists a morphism $\phi _{\ast}^{\prime }:\beta{}_{\ast}^{\,l}(A)\rightarrow \K_{\ast }^{l}(A,C)$ that lifts the identity of $R$. By the Comparison Theorem \cite[Theorem 2.2.6]{We} it follows that $\phi _{\ast }\circ \phi_{\ast }^{\prime}\ $and $\phi _{\ast }^{\prime }\circ \phi_{\ast }$ are homotopic to ${\Id}_{\beta{}_{\ast}^{\,l}(A)}$ and ${\Id}_{\K_{\ast }^{l}(A,C)},$ respectively. Obviously, $\Id_R\otimes_A\phi_{\ast}$ and $\Id_R\otimes_A\phi^{\prime }_{\ast}$ are inverses of each other up to a homotopy. We deduce that $\h_{n}({\Id}_{R}\otimes _{A}\phi _{\ast })$ is an isomorphism between $\h_n(R\otimes_A \K_{\ast}^l(A,C))$ and $\h_n\left(\Omega_{\ast}(A)\right)$, for any $n\geq 0$. On the other hand, by Proposition \ref{pr:KisoC}, the former homology group is isomorphic to $C_n$. In conclusion, the map $\gamma _{n}:C_{n}\rightarrow T_{n}(A) $ defined by
\begin{equation*}
	\gamma _{n}(c):=\Delta (n)(c)+B_{n}(A)
\end{equation*}%
	is an $R$-bimodule isomorphism. Here, $\Delta(0):={\Id}_{R}$ and $B_{n}(A)$ denotes the group of $n$-boundaries in $\Omega _{n}(A).$ Now it is not difficult to see that $\{\gamma _{n}\}_{n\in \mathbb{N}}$ is an isomorphism of graded $R$-corings between $C$ and $T(A).$ In conclusion ${\Id}_{A}\otimes \gamma _{\ast }$ is an isomorphism of chain complexes from $\K_{\ast }^{l}(A,C)$ to $\K_{\ast }^{l}(A,T(A)).$ Thus the complex $\K_{\ast}^{l}(A,T(A))$ is exact, meaning that $(A,T(A))$ is Koszul.

	We have seen that $C$ is strongly graded and that $(E(C),C)$ is an almost-Koszul pair. The morphism $\phi ^{\ast }$ from Proposition \ref{pr:Psi} lifts the identity of $R.$ By hypothesis $\beta{}_{r}^{\,\ast }(C)$ and $(\K_{r}^{\ast}(A,C),(-1)^{\ast }d{}_{r}^{\,\ast })$ are resolutions of $R$ in the category of right $C$-comodules, so $\Hom^{C}(R,\phi ^{\ast })$ is a quasi-isomorphism of complexes. By applying the functor $\Hom^{C}(R,-)$ to ${\beta}_{r}^{\ast }(C)$ and $\K_{r}^{\ast}(A,C)$ we get an isomorphism $\gamma^{\ast}$ of cochain complexes from $(\Omega^{\ast}(C),\partial^{\ast})$ to $(A^{\ast },0)$. 
	
	Of course, $\gamma ^{0}:={\Id}_{R}$.  Let us recall that $\overline{\pi } _{1}^{C}:\overline{C}\to C_1$ denotes the map induced by the projection $\pi^C_1$  and $\overline{\theta}:=\theta _{C,A}\circ \overline{\pi }{}_{1}^{\,C}$. For $n>0$ the map  $\gamma^n:\overline{C}^{\,\ot n}\to A^n$ is explicitely given by the formula
\begin{equation*}
	\gamma ^{n}({x}^{1}\otimes \cdots \otimes{x}^{n})=\overline{\theta }({x}^{1})\cdots \overline{\theta }({x}^{n}).
\end{equation*}%
	As $\Omega ^{\ast }(C)$ is the free $R$-ring generated by $\overline{C}$, it follows that $\gamma ^{\ast}$ is a morphism of cochain $R$-rings. Thus, it induces an isomorphism of  graded $R$-rings $E(C)\cong A$, which in turn can be used to identify $\K_{r}^{\ast }(E(C),C)$ and $\K_{r}^{\ast}(A,C).$
\end{proof}

\begin{corollary}\label{thm:CharacterizationKoszul}
	Let $(A,C^{\prime })$ and $(A,C^{\prime\prime })$ be Koszul pairs. Then $C^{\prime }\cong C^{\prime \prime }$ as graded $R$-corings. Dually, if $(A^{\prime },C)$ and $(A^{\prime \prime },C)$ are Koszul, then $A^{\prime }\cong A^{\prime \prime }$ as graded $R$-rings$.$
\end{corollary}

\begin{proof}
	By the preceding theorem, there are canonical coring isomorphisms $C^{\prime }\cong T(A)$ and $C^{\prime \prime }\cong T(A).$ Analogously, we have $A^{\prime }\cong E(C)$ and $A^{\prime \prime }\cong E(C)$.
\end{proof}

\begin{corollary}
	If $(A,C)$ is a Koszul pair then $E(T(A))\cong A\ $and $T(E(C))\cong C.$
\end{corollary}

\begin{proof}
	By Theorem \ref{te:C-Koszul}, both $(A,T(A))$  and $(E(T(A)),T(A))$ are Koszul. Using the preceding corollary we conclude that $A\cong E(T(A))$. The proof of the second isomorphism is similar.
\end{proof}

\begin{fact}[Koszul rings.]
	Let $A$ be a connected $R$-ring. Following \cite{BGS} we shall say that $A$ is a left Koszul ring if $R$ has a resolution $P_{\ast }\rightarrow R$ by projective graded left $A$-modules such that every $P_{n}$ is generated by its homogeneous elements of degree $n.$
\end{fact}

\begin{theorem}  \label{te:caracterizare_Koszul}
	Let $A$ be a connected $R$-ring. Then the following assertions are equivalent:
	\begin{enumerate}
	 \item $A$ is Koszul.
	 
	 \item $A$ is strongly graded and $(A,T(A))$ is Koszul.
	 
	 \item There exists a graded $R$-coring $C$ such that $(A,C)$ is Koszul.
	\end{enumerate}
\end{theorem}

\begin{proof}
	Let us assume that  $A$ is strongly graded and that $(A,T(A))$ is Koszul. Then $\K_{\ast }^{l}(A,T(A))$ is a resolution of $R$ by graded projective $A$-modules. Clearly, by  definition, every $\K_{n}^{l}(A,T(A))$ is generated as an $A$-module by $A^{0}\otimes T_{n}(A),$  its homogeneous component of degree $n$. Hence $A$ is Koszul.

	Let $A$  be a Koszul ring and let $V:=A^1$. By \cite{BGS}, any Koszul ring is quadratic and generated by $V$, so there is $W\subseteq V\otimes V$  such that $A$ is isomorphic as a graded $R$-ring with the quotient of $T_{R}^{a}(V)$ by the two-sided ideal generated by $W.$  Furthermore, we define $C_{0}:=R,$ $C_{1}:=V$ and
\begin{equation*}
	C_{n}=\bigcap_{p=0}^{n-2}V^{\otimes p}\otimes W\otimes V^{\otimes n-p-2},
\end{equation*}%
	for every $n\geq 2$. We claim that the $R$-bimodule $C:=\textstyle\bigoplus_{n\in \mathbb{N}}C_{n}$ is a graded subcoring of $T_{R}^{c}(V). $ It is enough to prove that $\Delta _{p,q}(c)\in C_{p}\otimes C_{q} $, for any $p,q\geq 0$ and $c\in C_{p+q}$. In the case when $p=0$ or $q=0,$ we have nothing to prove. Let us assume that $p>0$ and $q>0.$ Obviously, $\Delta _{p,q}(c)\in C_{p}\otimes T^{q},$ if $p=1.$ In the case when $p\geq 2$ we have $\Delta _{p,q}(c)=c,$ where in the right-hand side of this identity $c$ is regarded as an element in $V^{\otimes p}\otimes V^{\otimes q}$. Thus
\begin{equation*}
	\Delta _{p,q}(c)\in \bigcap_{i=0}^{p+q-2}V^{\otimes i}\otimes W\otimes V^{\otimes p+q-i-2}\subseteq\bigcap_{i=0}^{p-2}V^{\otimes i}\otimes W\otimes V^{\otimes p+q-i-2}=C_{p}\otimes V^{\otimes q}.
\end{equation*}%
	Similarly, $\Delta _{p,q}(c)\in V^{\otimes p}\otimes C_{q}.$ Thus $\Delta_{p,q}(c)\in C_{p}\otimes C_{q},$ that is our claim is proved.

	By construction $C$ is connected and $A^{1}=C_{1}=V.$ Let $c\in C_{2}=W$. Since $\Delta _{1,1}(c)=c$ and the multiplication of $T_{R}^{a}(V)$ is defined by the canonical isomorphisms $V^{\otimes p}\otimes V^{\otimes q}\cong V^{\otimes p+q}$, we have $m^{1,1}_A(\Delta _{1,1}(c))=c+W=0.$ Thus $(A,C)$ is almost-Koszul with respect to $\theta_{C,A}=\Id_V$. Taking into account the definition of $C$ it follows that $\K_{\ast }^{l}(A,C)$ coincides up to a degree shifting with the Koszul complex \cite[p. 483]{BGS}, which is exact by assumption. In conclusion if $A$ is Koszul, then there is  a Koszul pair $(A,C)$. Finally, by Theorem \ref{te:alg-echiv} we deduce that $A$ is strongly graded and $(A,T(A))$ is Koszul, provided that there exists a Koszul pair $(A,C)$.
\end{proof}

\begin{corollary}
	If $A$ is left Koszul, then $A$ is right Koszul, and conversely.
\end{corollary}

\begin{proof}
	Let $A$ be a left Koszul ring. Then, by the preceding theorem, $A$ is strongly graded and $(A,T(A))$ is Koszul. Since $A^{op}$ is strongly graded and $(A^{op},T(A)^{op})$ is Koszul, it follows that $A^{op}$ is a left Koszul ring. Of course, this is equivalent to the fact that $A$ is right Koszul.
\end{proof}

\begin{remark}
	The results of this section suggest the following definition. A  connected coring $C$ is called Koszul if and only if $C$ is strongly graded and $(E(C),C)$ is a Koszul pair. Koszul corings and some of their applications will be investigated in a sequel of this paper.
\end{remark} 

\section{Hochschild (co)homology of Koszul rings}

	In this section $R$ denotes a separable algebra over a field  $\Bbbk $. Therefore, for any  Koszul pair $(A,C)$ the complex $\K_{\ast }(A,C)$ is a projective resolution of $A$ in the category of right $A$-modules, cf.  Corollary \ref{co:resolutions}. We shall use this resolution to investigate the Hochschild (co)homology of $A$  with coefficients in an $A$-bimodule $M$.

\begin{fact}[The cyclic tensor product.]
	The tensor product algebra $R^{e}:=R\otimes _{\Bbbk }R^{op},$ between $R$ and its opposite algebra $R^{op},$ is called the \emph{enveloping algebra} of $R.$

	Since $R$ is an algebra over the field $\Bbbk$, we have to adapt the definition of $R$-bimodules to reflect this extra structure. By an  $R$-bimodule we mean a left (or  a right) $R^e$-module $V$. Equivalently, $V$ is a left and a right $R$ module which satisfies the usual associativity relation $(rv)s=r(vs)$ and, in addition, the condition $xv=vx$ holds for all  $x\in \Bbbk$ and $v\in V$. 
	
	For an $R$-bimodule $V$ let $[R,V]$ be the linear space spanned by all commutators $[r,v]:=rv-vr,$ with $r\in R$ and $v\in V$. Let $V_R:=V/[R,V]$.

	 Recall that $\otimes=\ot_R$. For any $R$-bimodules $V_{1},\dots ,V_{n}$ the tensor product $V_{1}\otimes \cdots \otimes V_{n}$ is a bimodule, so we may define \emph{the cyclic tensor product} of $V_{1},\dots ,V_{n}$  by the relation
	\[
	 V_{1}{\,\widehat{\otimes }\,\cdots \,\widehat{\otimes }\,V}_{n}:=\left(V_{1}\otimes \cdots \otimes V_{n}\right) _{R}.
	\]
	For the equivalence class of $v_{1}\otimes \cdots \otimes v_{n}$ in the cyclic tensor product we use the notation $v_{1}{\,\widehat{\otimes}\,\cdots \,\widehat{\otimes }\,v}_{n}.$
	If $V$ and $W$ are $R$-bimodules, then the map $v{\,\widehat{\otimes }\,w\mapsto v}\otimes _{R^{e}}w$ is an isomorphism $V{\,\widehat{\otimes }\,W}\cong V\otimes _{R^{e}}W$, so
\begin{equation*}
	V_{1}{\,\widehat{\otimes }\,\cdots \,\widehat{\otimes }\,V}_{n}\cong \left(V_{1}{\,\otimes \cdots \,\otimes V}_{i}\right) \otimes _{R^{e}}\left( V_{i+1}{\,\otimes \cdots \,\otimes V}_{n}\right) .
\end{equation*}%
	The cyclic tensor products $V{\,\widehat{\otimes }\,W}$ and $W{\widehat{\otimes }\,V}$ are isomorphic via the linear map $v{\,\widehat{\otimes }\,w}\mapsto w{\,\widehat{\otimes }\,v.}$ Thus
\begin{equation*}
	V_{1}{\,\widehat{\otimes }\,\cdots \,\widehat{\otimes }\,V}_{n}\cong V_{2}{\,\widehat{\otimes }\,V}_{3}{\,\widehat{\otimes }\,\cdots \,\widehat{\otimes }\,V}_{n}{\,\widehat{\otimes }\,V}_{1}\cong {\,\cdots \cong V}_{n}{\,\widehat{\otimes }\,{V}_{1}{\,\widehat{\otimes }\,}\cdots \,\widehat{\otimes }\,V}_{n-1}.
\end{equation*}
\end{fact}

\begin{fact}[The complex $\K_{\ast }(A,M)$.]
	Let $A$ be a Koszul ring over a  separable $\Bbbk $-algebra $R$. The $R$-ring structure of $A$ induces a canonical ring morphism from $\Bbbk $ to $A$. Since we are interested in the Hochschild (co)homology of $A$, we assume that the image of this map is central in $A$, that is $A$ is a $\Bbbk$-algebra. In this setting we also define an $A$-bimodule as a left module over the enveloping algebra $A^e:=A\ot_\Bbbk A^{op}$ of $A$. We fix a connected $R$-coring $C$ such that $(A,C)$ is Koszul. Recall that  $C\cong T(A)$ as $R$-corings.  By the above assumption, $T(A)$ is an $R^e$-module, so we may assume that $C$ has the same property.

	The Hochschild homology of $A$ with coefficients in an $A$-bimodule $M $ is defined by the relation $\hh_{\ast}(A,M):=\tor_{\ast }^{A^{e}}(A,M)$. Let $\K_{\ast }'(A,C)$ denote the complex obtained from $\K_{\ast }(A,C)$ by dropping the part of degree $-1$. Thus the Hochschild homology of $A$ with coefficients in $M$ may be computed as the homology of the complex $\K_\ast'(A,C)\otimes _{A^{e}}M$.  We identify the vector spaces $\K_{n}(A,C)\otimes _{A^{e}}M$ and $M{\,\widehat{\otimes }\,}C_{n}\cong{\,C_{n}\widehat{\otimes }\,}M $ via the map $\varphi _{n}$ defined by
\begin{equation*}
	\varphi _{n}((x\otimes c\otimes y)\otimes _{A^{e}}m)=(ymx){\widehat{\otimes }\,}c.
\end{equation*}%
	Its inverse is given by $\varphi  _{n}^{-1}(m{\,\widehat{\otimes }\,}c)=(1\otimes   c\otimes 1)\otimes _{A^{e}}m$. Let  $\partial _{n}:M{\,\widehat{\otimes }\,}C_{n}\rightarrow M{\ \widehat{\otimes}\,}C_{n-1}$ be the map
\begin{equation*}
	\partial _{n}:= \varphi_{n-1}\circ (d_{n}\otimes_{A^{e}}\Id_M)\circ \varphi_{n}^{-1},
\end{equation*}%
	where $d_{n}$ denotes the differential of $\K_{\ast }(A,C)$.  It is easy to see  that $\pd_n$ satisfies the following equation 
\begin{equation}\label{eq:pdn}
	\partial _{n}(m{\,\widehat{\otimes }\,}c)=\sum\limits m\theta _{C,A}(c_{(1,1)}){\,\widehat{\otimes }\,\,c_{(2,{n-1})}}+(-1)^{n}\sum\limits \theta _{C,A}(c_{(2,1)})m{\,\,\widehat{\otimes }\,\,c_{(1,{n-1})}}.
\end{equation}
Clearly, $\varphi _{\ast }:M\otimes _{A^{e}}\K_{\ast}'(A,C)\to (M{\,\widehat{\otimes }\,}C_{\ast },\partial _{\ast })$ is an isomorphism. Hence we have the following result.
\end{fact}

\begin{theorem}\label{thm:HochHom}
	Let $(A,C)$ be a Koszul pair over a separable $\Bbbk$-algebra $R$. The Hochschild homology of $A$ with coefficients in $M$ is the homology of the chain complex $\K_{\ast}(A,M)=M{\widehat{\otimes }}C_{\ast }.$ For $m\in M$ and $c\in C_{n},$ the differential $\partial _{n}$ of this complex is given by equation \eqref{eq:pdn}.
\end{theorem}

	Recall that the Hochschild cohomology of $A$ with coefficients in an $A$-bimodule $M$ is defined by $\hh^{\ast }(A,M):=\ext_{A^{e}}^{\ast }(A,M)$. For a Koszul pair $(A,C)$, proceeding as in the proof of the preceding theorem, we obtain a complex to compute Hochschild cohomology of $A$. 

\begin{theorem}\label{thm:HochCoh}
	Let $(A,C)$ be a Koszul pair over a separable $\Bbbk $-algebra $R$. The Hochschild cohomology of $A$ with coefficients in $M$ is the cohomology of the cochain complex $\K^{\ast}(A,M)=\Hom_{R^{e}}(C_{\ast },M)$. For $c\in C_{n+1}$ and $f\in \Hom_{R^{e}}(C_{n},M)$ the differential $\partial ^{n}$ of this complex is given by
\begin{equation}\label{HCoh}
	\partial ^{n}(f)(c)=\sum\limits \theta _{C,A}(c_{(1,1)}) f(c_{(2,n)}) +(-1)^{n+1} \sum \limits f(c_{(1,n)}) \theta _{C,A}(c_{(2,1)}).
\end{equation}
\end{theorem}

  As an application of our previous results, we compute $\Hd A$, the Hochschild dimension  of a Koszul ring $A$. By definition, $\Hd A$ is the projective dimension of $A$ as a left (or right) $A^{e}$-module.  Hence $\Hd A=n$ if and only if $\hh^{n+1}(A,M)=0$ for any bimodule $M$, but there is at least one bimodule $M_0$ such  that $\hh^{n}(A,M_0)\neq 0$. Of course, if such an $n$ does not exist, then we say that the Hochschild dimension of $A$ is infinite. The projective dimension of $R$ as a left $A$-module will be denoted by $\Pd({}_{A}R)$. For the projective dimension of the right $A$-module $R$ we shall use the notation $\Pd(R_{A})$.

\begin{theorem}\label{thm:Hdim}
  If $(A,C)$ is a Koszul pair over a separable $\Bbbk$-algebra $R$, then
\begin{equation*}
  \Hd A=\Pd({}_{A}R)=\Pd(R_{A})=\sup \{n\mid C_{n}\neq 0\}.
\end{equation*}
\end{theorem}

\begin{proof}
  Obviously, $\Pd({}_{A}R)=\Pd(R_{A})$. Let us assume that $C_{d+1} = 0$, for some $d$. Then $\K_*^l(A,C)$ provides a projective resolution of $R$ in the category of left $A$-modules of length at most $d$ . Thus, for $n>d$, we have $C_n\cong\tor_n^A(R,R)$=0. It follows that $S:=\sup \{n\mid C_{n}\neq 0\}\leq d$. In particular, if $\Pd(R_{A})\leq d$, then $S\leq d$. On the other hand, assuming that the latter inequality holds, we get $C_{d+1}=0$. Thus, in view of the foregoing remarks, $\Pd(_RA)\leq d$. In conclusion, $S=\Pd(_RA)$.

  It remains to prove that $\Hd A\leq d$ if and only if $\Pd(_RA)\leq d$. Let us suppose that the Hochschild dimension of $A$ is less than or equal to $d$. Therefore, the $d$-syzygy  $M$ of the resolution $\K_{\ast }(A,C)$ is a projective $A$-bimodule such that the sequence
\begin{equation*}
  0\longleftarrow A\longleftarrow \K_{0}(A,C){\longleftarrow }\dotsb {\longleftarrow }\K_{d-2}(A,C)\longleftarrow \K_{d-1}(A,C){\longleftarrow }M{\longleftarrow }0.
\end{equation*}%
  is exact. Proceeding as in the proof of Proposition \ref{prop:Koszulcomplexes}, one shows that this sequence splits in the category of right $A$-modules. Hence, by applying the functor $(-)\ot_A R$  we get a resolution of $R$ by projective {left} $A$-modules.  Since the length of this resolution is at most $d$ we conclude that $\Pd(_RA)\leq d$. Conversely, if the projective dimension of $R$ as a left $A$-module is less than or equal to $d$, then $C_{d+1}\cong T_d(A)=0$. Thus $\K_\ast(A,C)$ is a resolution of $A$ of length less than or equal to $d$, that is $\Hd A\leq d$.
\end{proof}

We conclude this section by giving a first  example of Koszul pair, which we shall use later for the computation of Hochschild (co)homology of generalized Ore extensions. For every $R$-bimodule $V$ let $T:=T^a_R(V)$. Let  $C=R\oplus V$ denote the connected coring with the comultiplication $\Delta$ uniquely defined such that $\Delta (v)=v\otimes 1+1\otimes v,$ for any $v\in V.$ We take $\theta _{T,C}$ to be the identity map.

\begin{proposition}\label{lemma:td} 
	The pair $(T,C)$ is Koszul.
\end{proposition}

\begin{proof}
  The identity from the definition of almost-Koszul pairs is automatically verified, as $C_{2}=0.$ Thus $(T,C)$ is such a pair. Furthermore, $\K_{\ast }^{l}(T,C)$ is the complex
\begin{equation*}
  0 \longleftarrow R\longleftarrow T\longleftarrow T\otimes V \longleftarrow 0, 
\end{equation*}
  whose non-zero arrows are the projection $\pi _{T}^{0}$ of $T$ on $T^{0}$ and $d_{1}^{l}.$ Since the multiplication in $T$ is given by concatenation of tensor monomials, and $\Delta _{1,0}(v)=v\otimes 1$ we deduce that $d_{1}^{l}(x\otimes v)=x\otimes v.$ Therefore, $d_{1}^{l}$ is the identity map of $T\otimes V=\overline{T}$. Hence $\K_{\ast }^{l}(T,C)$ is exact.
\end{proof}
 
\begin{remark}
  By Theorem \ref{thm:Hdim} it follows that $\Hd T_R^a(V)=1$, for any separable algebra $R$ and any $R$-bimodule $V$. This property of tensor algebras was proved in \cite{CQ}, where the algebras of Hochschild dimension one are called quasi-free and they  represent the key ingredient in the definition of nonsingularity in Noncommutative Geometry. The Koszulity of tensor algebras is also proved in the last section of the paper, as a consequence of the fact that they are braided bialgebras. 
\end{remark}

\section{Almost-Koszul pairs associated to twisted tensor products}

  We keep the notation and the assumptions from the first section. In this section we consider two connected strongly graded $R$-rings $A$ and $B$ together with an invertible  graded twisting map $\sigma :B\otimes  A\rightarrow A\otimes B$. Our first aim is to show that $\sigma $ induces a canonical graded twisting map of $R$-corings 
\begin{equation*}
  {\tau }:T(A)\otimes T(B)\rightarrow T(B)\otimes T(A). 
\end{equation*}%
  This construction is performed such that $\left(A\otimes _{\sigma }B,T(A)\otimes _{{\tau }}T(B)\right)$ is an almost-Koszul pair, where $A\otimes _{\sigma }B$ and $T(A)\otimes _{{\tau }}T(B)$ denote the twisted tensor product $R$-ring and the twisted tensor product $R$-coring with respect $\sigma $ and ${\tau }$, respectively.  Furthermore, if $A$ and $B$ are Koszul and $\sigma$ is invertible, then we shall show that $A\otimes _{\sigma }B$ is a Koszul $R$-ring.

\begin{fact}[Twisting maps of $R$-rings.]
  Let $A$ and $B$ be $R$-rings. A \emph{twisting map} between $A$ and $B$ is given by an $R$-bilinear map $\sigma :B\otimes A\rightarrow A\otimes B$ which is compatible with the multiplication of $A$ and $B$, \emph{i.e.}
\begin{align}
  \sigma \circ (\Id_{B}\otimes m_{A})& =(m_{A}\otimes \Id_{B})\circ (\Id_{A}\otimes \sigma )\circ (\sigma \otimes \Id_{A}),  \label{twist1}
	\\
  \sigma \circ (m_{B}\otimes \Id_{A})& =(\Id_{A}\otimes m_{B})\circ (\sigma \otimes \Id_{B})\circ (\Id_{B}\otimes \sigma ).  \label{twist2}
\end{align}%
  By definition, $\sigma $ must be compatible with the units of $A$ and $B$ as well. Therefore, $\sigma (1_{B}\otimes a)=a\otimes 1_{B}$ and $\sigma (b\otimes 1_{A})=1_{A}\otimes b,$ for all $a\in A$ and $b\in B$. In computations we shall write $\sigma(b\otimes a)\in A\otimes B$ as a formal sum $\sum a_{\sigma }\otimes b_{\sigma }.$ Thus, for instance
\begin{equation*}
  \lbrack (m_{A}\otimes \Id_{B})\circ (\Id_{A}\otimes \sigma )\circ (\sigma \otimes \Id_{A})](b\otimes a^{\prime }\otimes a^{\prime \prime })=\sum a_{\sigma }^{\prime }a_{\sigma ^{\prime }}^{\prime \prime }\otimes \left(b_{\sigma }\right) _{\sigma ^{\prime }}.
\end{equation*}%
  The occurrence of $\sigma $ and $\sigma ^{\prime }$ in the above identity indicates that the twisting map  is applied twice.

  If $\sigma $ is a twisting map then $A\otimes B$ has a canonical $R$-ring structure with respect to the multiplication
\begin{equation*}
  (a^{\prime }\otimes b^{\prime })(a^{\prime \prime }\otimes  b^{\prime \prime })=\sum a^{\prime }a_{\sigma }^{\prime \prime  }\otimes b_{\sigma }^{\prime }b^{\prime \prime }
\end{equation*}%
  and the unit $1_{A}\otimes 1_{B}.$ The twisted tensor product will be denoted by $A\otimes _{\sigma }B$. In the case when $R$ is commutative and $A $ and $B$ are $R$-algebras, the twisted tensor product $A\otimes _{\sigma}B$ may be seen as a deformation of the usual tensor product algebra.

  Let us now assume that the $R$-rings $A$ and $B$ are both graded. A twisting map $\sigma $ between $A$ and $B$ is called \emph{graded} if $\sigma (B^{p}\otimes A^{q})\subseteq A^{q}\otimes B^{p}$. The restrictions of $\sigma $ to $B^{p}\otimes A^{q}$ will be denoted by $\sigma ^{p,q}.$ For such a  $\sigma $ the $R$-ring $A\otimes _{\sigma }B$ is graded and its homogeneous  component of degree $n$ is the direct sum of all $R$-bimodules $A^{p}\otimes B^{q}$ with $p+q=n.$
\end{fact}

\begin{fact}[Twisting maps of cochain $R$-rings.]
  We now assume that $(\Omega ^{\ast },d_{\Omega }^{\ast })$ and $(\Gamma ^{\ast },d_{\Gamma }^{\ast })$ are cochain $R$-rings. A graded twisting map $\sigma ^{\ast ,\ast }:\Gamma ^{\ast }\otimes \Omega ^{\ast }\rightarrow \Omega ^{\ast }\otimes \Gamma ^{\ast }$ is called a\emph{\ twisting map of cochain }$R$-\emph{rings} if $\sigma ^{\ast ,\ast }$ is compatible with the differential maps of $\Omega ^{\ast }$ and $\Gamma ^{\ast },$ in the sense that $\sigma ^{\ast ,q}$ is a map of complexes from $(\Gamma ^{\ast }\otimes \Omega ^{q},d_{\Gamma }^{\ast }\otimes \Id_{\Omega ^{q}})$ to $(\Omega ^{q}\otimes \Gamma ^{\ast },\Id_{\Omega ^{q}}\otimes d_{\Gamma }^{\ast }),$ for all $q.$ Symmetrically, $\sigma ^{p,\ast }$ must be a morphism of complexes from $(\Gamma ^{p}\otimes \Omega ^{\ast },\Id_{\Gamma ^{p}}\otimes d_{\Omega }^{\ast })$ to $(\Omega ^{\ast }\otimes\Gamma ^{p},d_{\Omega }^{\ast }\otimes \Id_{\Gamma ^{p}}),$ for all $p.$
\end{fact}

\begin{proposition}\label{pr:twist}
  Let $(\Omega ^{\ast },d_{\Omega }^{\ast })$ and $(\Gamma ^{\ast},d_{\Gamma }^{\ast })$ be cochain $R$-rings. Suppose that $V$ is an $R$-bimodule.

\begin{enumerate}
  \item If $\varphi ^{\ast }:(V\otimes \Omega ^{\ast },{\Id}_{V}\otimes d_{\Omega }^{\ast })\rightarrow (\Omega ^{\ast }\otimes V,$ $d_{\Omega }^{\ast}\otimes {\Id}_{V})$ is a morphism of complexes which is compatible with the multiplication and the unit of $\Omega ^{\ast },$ then $\overline{\varphi }{}^{\,\ast}:=\h^{\ast }(\varphi ^{\ast })$ is compatible with the multiplication and the unit of the graded $R$-ring $\h^{\ast}(\Omega ^{\ast }).$

  \item If $\chi ^{\ast }:(\Gamma ^{\ast }\otimes V,d_{\Gamma }^{\ast}\otimes{\Id}_{V})\rightarrow (V\otimes \Gamma ^{\ast },{\Id}_{V}$ $\otimes d_{\Gamma}^{\ast })$ is a morphism of complexes which is compatible with the multiplication and the unit of $\Gamma ^{\ast },$ then $\overline{\chi }^{\,\ast }:=\h ^{\ast}(\chi ^{\ast })$ is compatible with the multiplication and the unit of the graded $R$-ring $\h^{\ast }(\Gamma ^{\ast }).$

  \item Every twisting map of cochain $R$-rings $\sigma ^{\ast ,\ast }:\Gamma ^{\ast }\otimes \Omega ^{\ast }\rightarrow \Omega ^{\ast}\otimes \Gamma ^{\ast}$ induces a twisting map of graded $R$-rings $\overline{\sigma } ^{\,\ast ,\ast}:\h^{\ast }(\Gamma ^{\ast })\otimes \h^{\ast }(\Omega ^{\ast })\rightarrow \h^{\ast}(\Omega ^{\ast})\otimes \h^{\ast}(\Gamma ^{\ast}).$
\end{enumerate}
\end{proposition}

\begin{proof}
  For every $p>0$ the morphism $\varphi ^{\ast }$ induces a map
\begin{equation*}
  \overline{\varphi }^{\,p}:\h^{p}(V\otimes \Omega ^{\ast})\rightarrow\h^{p}(\Omega ^{\ast }\otimes V).
\end{equation*}%
  By assumption every left or right $R$-module is flat. Hence, $\overline{ \varphi}^{\,p}$ can be seen as a map from $V\otimes\h^{p}(\Omega ^{\ast })$ to $\h^{p}(\Omega ^{\ast})\otimes V.$ For $x\in \mathrm{Z}^{p}(\Omega^{\ast })$ let $\left[ x\right] $ denote its cohomology class. Note that $\varphi ^{\,p}(v\otimes x)$ is an element in $\mathrm{Z} ^{p}(\Omega ^{\ast})\otimes V$, as $\varphi ^{\ast }$ is a morphism of complexes. So $\varphi ^{p}(v\otimes x)=\sum x_{\varphi }\otimes v_{\varphi }$, for some $x_{\varphi }\in \mathrm{Z}^{p}(\Omega ^{\ast })  $. Hence
\begin{equation}
  \overline{\varphi }^{\,p}(v\otimes \left[ x\right] )=\sum [x_{\varphi }]\otimes v_{\varphi }.  \label{sigma-bar}
\end{equation}%
  Since $\varphi ^{\ast }$ is compatible with the multiplication of $\Omega ^{\ast}$, we get
\begin{equation*}
  \sum [\left( xy\right) _{\varphi }]\otimes {v_{\varphi }}=\sum \left[ x_{\varphi}y_{\varphi ^{\prime }}\right] \otimes {\left(v_{\varphi}\right) _{\varphi ^{\prime }}}.
\end{equation*}%
  Thus, by the definition of the multiplication in $\h^{\ast}(\Omega^{\ast}) $, it follows
\begin{equation*}
  \sum \left( \left[ x\right] \left[ y\right] \right)_{\overline{\varphi }}\otimes v_{\overline{\varphi}}=\sum \left[ x\right] _{\overline{\varphi }}\left[ y\right] _{\overline{\varphi }^{\prime }}\otimes \left( v_{\overline{\varphi }}\right) _{\overline{\varphi}^{\prime }}.
\end{equation*}%
  In conclusion, $\overline{\varphi }^{\,\ast }$ is also compatible with the multiplication of $\h^{\ast }(\Omega ^{\ast }).$ By the definition of $\overline{\varphi }^{\,\ast }$ one can easily see that this family of $R$-bilinear maps is compatible with the unit of $\h^{\ast }(\Omega ^{\ast })$. 
  
  We omit the proof of the second statement, being similar to the above one. 
  
  Let $\sigma $ be a twisting map of cochain $R$-rings. If $p\geq 0$ then $ \sigma ^{p,\ast }$ is a morphism of complexes from $\Gamma ^{p}\otimes \Omega ^{\ast }$ to $\Omega ^{\ast }\otimes \Gamma ^{p}$ which is compatible with the multiplication and the unit of $\Omega ^{\ast}.$ By the first part of the proposition it follows that $\h^{\ast }(\sigma ^{p,\ast }):\Gamma ^{p}\otimes  \h^{\ast}(\Omega )\rightarrow \h^{\ast }(\Omega ^{\ast })\otimes \Gamma ^{p}$ is  compatible with the multiplication and the unit of                     $\h^{\ast }(\Omega ^{\ast}).$ For a given $q\geq 0$, the family of $R$-bilinear maps  $\left\{\h^{q}(\sigma ^{p,\ast })\right\} _{p\geq 0}$ is a morphism  of complexes from $\Gamma ^{\ast}\otimes \h^{q}(\Omega )$ to $\h ^{q}(\Omega ^{\ast})\otimes \Gamma ^{\ast }$ which is compatible with the multiplication and the unit of $\Gamma ^{\ast }$. Applying the second part of the proposition, for every $p$ and $q$, we get a map
\begin{equation*}
  \overline{\sigma }^{\,p,q}:\h^{p}(\Gamma ^{\ast })\otimes\h^{q}  (\Omega^{\ast})\rightarrow \h^{q}(\Omega ^{\ast })\otimes\h^{p} \Gamma ^{\ast }),\quad\overline{\sigma }^{\,p,q}([x]\otimes\lbrack y])=\sum\lbrack y_{\sigma }\rbrack\otimes \lbrack x_{\sigma}\rbrack,
\end{equation*}%
  such that $\overline{\sigma }^{\,\ast ,\ast }$ is a graded twisting map of graded $R$-rings.
\end{proof}

\begin{fact}[Twisting maps of $R$-corings.]
  Let $C$ and $D$ be $R$-corings. A \emph{twisting map}  \emph{between} $C$ \emph{and} $D$ is an $R$-bilinear map $\tau :C\otimes D\rightarrow D\otimes C$ compatible with the comultiplication of $C$ and $D, $ that is
\begin{align}
  (\Delta _{D}\otimes \Id_{C})\circ \tau & =(\Id_{D}\otimes\tau) \circ (\tau\otimes \Id_{D})\circ (\Id_{C}\otimes \Delta _{D}), \label{cotwist1}
			\\
  (\Id_{D}\otimes \Delta _{C})\circ \tau & =(\tau \otimes \Id_{C})\circ (\Id_{C}\otimes \tau )\circ (\Delta _{C}\otimes \Id_{D}).  \label{cotwist2}
\end{align}%
  By definition, $\tau $ must be compatible with the counits of $C$ and $D$ as well. Thus, $\left( \Id_{D}\otimes \varepsilon _{C}\right) \circ \tau =\varepsilon _{C}\otimes {\Id}_{D}$ and $\left( \varepsilon _{D}\otimes \Id _{C}\right) \circ \tau =\Id_{C}\otimes \varepsilon _{D}.$

  For a twisting map of corings we use the notation $\tau (c\otimes d)=\sum d_{\tau }\otimes c_{\tau },$ for all $c\in C$ and $d\in  D.$ The tensor product $C\otimes D$ has a canonical $R$-coring structure, that will be denoted by $C\otimes _{\tau }D$. The counit of this $R$-coring is $\varepsilon _{C}\otimes \varepsilon _{D}$, and its comultiplication $\Delta $ is defined by the formula
\begin{equation*}
  \Delta :=(\Id_C\otimes \tau \otimes \Id_D)\circ(\Delta _{C}\otimes \Delta _{D}).
\end{equation*}%
  Let us now assume that $C$ and $D$ are graded corings. A twisting map of corings $\tau :C\otimes D\rightarrow D\otimes C$ is called  \emph{graded }if $\tau (C_{p}\otimes D_{q})\subseteq D_{q}\otimes C_{p}.$ The restriction of $\tau $ to $C_{p}\otimes D_{q}$ will be denoted by $\tau _{p,q}.$ Clearly, in this case $C\otimes _{\tau }D$ is a graded $R$-coring, whose homogeneous component of degree $n$ is the direct sum of all bimodules $C_{p}\otimes D_{q}$ with $p+q=n.$
\end{fact}

\begin{fact}[Twisting maps of chain corings.]
  We now assume that $(\Omega _{\ast },d_{\ast }^{\Omega })$ and $(\Gamma_{\ast },d_{\ast }^{\Gamma })$ are chain $R$-corings. A  graded twisting map $\tau _{\ast ,\ast }:\Omega _{\ast }\otimes \Gamma _{\ast }\rightarrow \Gamma_{\ast}\otimes \Omega _{\ast }$ is called a\emph{\ twisting map of chain} $R $-\emph{corings} if $\tau _{\ast ,q}$ is a map of complexes from $(\Omega_{\ast}\otimes \Gamma _{q},d_{\ast }^{\Omega }\otimes \Id_{\Gamma _{q}})$ to $(\Gamma_{q}\otimes \Omega _{\ast },\Id_{\Gamma _{q}}\otimes d_{\ast}^{\Omega }),$ for any $q.$ In addition, $\tau _{p,\ast }$ is a morphism of complexes from $(\Omega _{p}\otimes \Gamma _{\ast},\Id_{\Omega _{p}}\otimes d_{\ast }^{\Gamma })$ to $(\Gamma  _{\ast }\otimes \Omega _{p},d_{\ast }^{\Gamma}\otimes \Id_{\Omega_{p}}),$ for any $p.$

  Reasoning as in the proof of Proposition \ref{pr:twist}, one can  show that the following result holds true.
\end{fact}

\begin{proposition}\label{pr:cotwist}
  Let $(\Omega _{\ast },d_{\ast }^{\Omega })$ and $(\Gamma_{\ast},d_{\ast }^{\Gamma })$ be chain corings. Suppose that $V$ is an $R$-bimodule.

\begin{enumerate}
  \item If $\varphi _{\ast }:(V\otimes \Gamma _{\ast },{\Id}_{V}\otimes d_{\ast }^{\Gamma })\rightarrow (\Gamma _{\ast }\otimes V,$ $d_{\ast}^{\Gamma }\otimes {\Id}_{V})$ is a morphism of complexes that is compatible with the comultiplication and the counit of $\Gamma _{\ast }$ then $\overline{\varphi }_{\ast }:=\h_{\ast}(\varphi _{\ast })$ is also compatible with the comultiplication and the counit of the graded $R$-coring $\h_{\ast }(\Gamma_{\ast}).$

  \item If $\chi _{\ast }:(\Omega _{\ast }\otimes V,d_{\ast }^{\Omega }\otimes{\Id}_{V})\rightarrow (V\otimes \Omega _{\ast },{\Id}_{V}\otimes d_{\ast}^{\Omega })$ is a morphism of complexes that is compatible with the comultiplication and the counit of $\Omega _{\ast }$ then $\overline{\chi}_{\ast }:=\h_{\ast }(\chi_{\ast })$ is also compatible with the comultiplication and the counit of the graded $R$-coring $\h_{\ast }(\Omega _{\ast }).$

  \item If $\tau _{\ast ,\ast }:\Omega _{\ast }\otimes \Gamma_{\ast}\rightarrow \Gamma_{\ast}\otimes \Omega_{\ast }$ is a twisting map of chain corings, then $\tau _{\ast ,\ast }$  induces a twisting map of graded $R$-corings
\begin{equation*}
  \overline{\tau }_{\ast ,\ast }:\h_{\ast }(\Omega _{\ast })\otimes \h_{\ast }(\Gamma _{\ast })\rightarrow \h_{\ast }(\Gamma _{\ast })\otimes \h_{\ast}(\Omega _{\ast }).
\end{equation*}
\end{enumerate}
\end{proposition}

\begin{fact}[Entwining maps.]
  Let $A$ be an $R$-ring, and let $C$ be an $R$-coring. We say that a bimodule morphism $\lambda :C\otimes A\rightarrow A\otimes C$ is an\emph{\ entwining map} if $\lambda (c\otimes 1_{A})=1_{A}\otimes c$ and $\left( \Id_{A}\otimes \varepsilon _{C}\right) \circ \lambda =\varepsilon _{C}\otimes \Id_{A},\ $ and the following relations  hold
\begin{eqnarray}
  \lambda \circ (\Id_C\otimes m_{A}) &=&(m_{A}\otimes \Id_C)\circ (\Id_A\otimes  \lambda )\circ (\lambda \otimes \Id_A),  \label{entw1}
	\\
  (\Id_A\otimes \Delta _{C})\circ \lambda &=&(\lambda \otimes \Id_C)\circ (\Id_C\otimes \lambda)\circ (\Delta _{C}\otimes \Id_A).  \label{entw2}
\end{eqnarray}%
  Similarly one can define an entwining structure $\nu:A\ot C\to C\ot A$.
  
  Let us now assume that $A$ and $C$ are both graded. An entwining map $\lambda :C\otimes A\rightarrow A\otimes C$ is called \emph{graded} if $\lambda (C_{p}\otimes A^{q})\subseteq A^{q} \otimes C_{p}.$ The restriction of $\lambda $ to $C_{p}\otimes A^{q}$ will be denoted by $\lambda _{p}^{q}.$

  Let $(\Omega ^{\ast },d_{\Omega }^{\ast })$ and $(\Gamma _{\ast },d_{\ast}^{\Gamma })$ be a cochain $R$-ring and a chain $R$-coring,  respectively. A graded entwining map $\lambda _{\ast }^{\ast }:\Gamma _{\ast }\otimes \Omega ^{\ast }\rightarrow \Omega ^{\ast }\otimes \Gamma _{\ast }$ is called a \emph{differential entwining map} if $\lambda _{p}^{\ast }: \Gamma _{p}\otimes \Omega ^{\ast }\to \Omega ^{\ast }\otimes \Gamma _{p}$  and $\lambda _{\ast }^{q}: \Gamma _{\ast }\otimes \Omega ^{q}\to  \Omega ^{q}\otimes \Gamma _{\ast }$ are morphisms of complexes, for any $p$  and $q$. We state for future reference, without proof, the following proposition.
\end{fact}

\begin{proposition}\label{pr:entw}
  Let $(\Omega ^{\ast },d_{\Omega }^{\ast })$ and $(\Gamma_{\ast},d_{\ast }^{\Gamma})$ be a cochain $R$-ring and a chain $R$-coring, respectively. Any differential entwining map $\lambda _{\ast }^{\ast}:\Gamma _{\ast }\otimes \Omega ^{\ast}\rightarrow \Omega ^{\ast }\otimes \Gamma _{\ast }$ induces a graded entwining map
\begin{equation*}
  \overline{\lambda }_{\ast }^{\,\ast }:\h_{\ast }(\Gamma _{\ast})\otimes \h^{\ast }(\Omega ^{\ast })\rightarrow \h^{\ast}(\Omega ^{\ast })\otimes \h _{\ast }(\Gamma _{\ast }).
\end{equation*}
\end{proposition}

  It is well known that any $\Bbbk $-linear map $\sigma ^{1,1}:N\otimes _{\Bbbk }M\rightarrow M\otimes _{\Bbbk }N$ can be extended in a unique way to a graded twisting map $\sigma :T(N)\otimes_\Bbbk T(M)\rightarrow T(M)\otimes_\Bbbk T(N)$ between the free algebras generated by $M$ and $N.$ We shall adapt the method from \cite{BM} in order to produce examples of twisting  maps of chain coalgebras and differential entwining maps. Some of them will be used later on in  the  paper to show that the twisting tensor product of two Koszul $R$-rings is Koszul.

  Recall that, for every $R$-bimodule $V,$ the graded $R$-ring $T_{R}^{a}(V)$ and the graded $R$-coring $T_{R}^{c}(V)$ have the same homogeneous component of degree $n$, namely $V^{\otimes n}$. Their multiplication and comultiplication are defined by the canonical isomorphism $V^{\otimes p}\otimes V^{\otimes q}\overset{\cong }{\longrightarrow }V^{\otimes p+q}$ and its inverse, respectively.

\begin{proposition}\label{pr:phi(V,W)}
  Let $\varphi ^{V,W}:W\otimes V\rightarrow V\otimes W$ be an $R$-bimodule map.

\begin{enumerate}
  \item There exists a unique bimodule map $\varphi _{\ast }^{V,W}:T_{R}^{c}(W)\otimes V\rightarrow V\otimes T_{R}^{c}(W)$ verifying the relation $\varphi _{1}^{V,W}=\varphi ^{V,W}$ and which is compatible with the graded coring structure of $T_{R}^{c}(W)$. 

  \item If $A$ is an $R$-ring and $\varphi^{A,W}$ is compatible with the multiplication and the unit of $A$, then $\varphi _{\ast }^{A,W}$ is an entwining map. Moreover, $\varphi _{\ast }^{A,W}$ is graded, provided that $A$ is graded and $\varphi ^{A,W}$ maps $W\otimes A^{q}$ to $A^{q}\otimes W$ for all $q$.

  \item If, in addition, $B$ is a connected $R$-ring and $\varphi^{A,\overline{B}}$ is compatible with the multiplication and the unit of $\overline{B}$, then $\varphi _{\ast}^{A,\overline{B}}$ is an entwining map from $(\Omega _{\ast }(B)\otimes A,\partial _{\ast }\otimes {\Id}_{A})$ to $(A\otimes \Omega _{\ast }(B),{\Id}_{A}\otimes \partial _{\ast })$ which commutes with the differentials.

  \item If $(A,d^{\ast })$ is a cochain $R$-ring and $\varphi ^{A, \overline{B}} $ is a morphism of complexes as in (3), then $\varphi _{\ast }^{A,\overline{B}}$ is a differential entwining map.
\end{enumerate}
\end{proposition}

\begin{proof}
  In the case when $R$ is a field, the first part of the lemma is proved in \cite{BM}. The same proof works in our setting as well. Let us assume that we have already constructed $\varphi _{\ast }^{V,W}.$ By the definition of the coring $T_R^c(W)$, the compatibility of $\varphi _{\ast }^{V,W}$ with $\Delta _{p,q}$ is equivalent to the  relation
\begin{equation} \label{ec:phi_ast}
  \varphi _{p+q}^{V,W}=(\varphi _{p}^{V,W}\otimes {\Id}_{W^{\otimes q}})\circ ({\Id}_{W^{\otimes p}}\otimes \varphi _{q}^{V,W})
\end{equation}%
  where $p$ and $q$ are arbitrary nonnegative integers. In particular, if the map $\varphi _{\ast }^{V,W}$ exists it is  uniquely defined by the condition $ \varphi _{1}^{V,W}=\varphi ^{V,W}.$ On the other hand, to prove the existence of $\varphi _{\ast }^{V,W}$ we can proceed as follows. The map $\varphi _{0}^{V,W}$ must be the canonical identification $R\otimes V\cong V\otimes R$, as $ \varphi _{\ast}^{V,W}$ is compatible with the counit of $T_{R}^{c}(W)$.   We set $\varphi_{1}^{V,W}=\varphi ^{V,W}$ and, for $p>1$, we define $\varphi _{p}^{V,W}$ by
\begin{equation*}
  \varphi _{p}^{V,W}:=\left( \varphi ^{V,W}\otimes {\Id}_{W^{\otimes p-1}}\right) \circ \left( {\Id}_{W}\otimes \varphi ^{V,W}\otimes {\Id}_{W^{\otimes p-2}}\right) \circ \cdots \circ \left( {\Id}_{W^{\otimes p-2}}\otimes \varphi ^{V,W}\otimes {\Id}_{W}\right) \circ \left( {\Id}_{W^{\otimes p-1}}\otimes \varphi ^{V,W}\right) .
\end{equation*}%
  It follows easily by induction on $q$ that the relation \eqref{ec:phi_ast} is true for any $p$ and $q$.

  For the second part of the proposition we have to prove that  $\varphi _{\ast}^{A,W}$ is compatible with the unit and the multiplication of $A.$ Both compatibility conditions follow by  induction on $p$, using on the one hand the relation (\ref{ec:phi_ast}) written for $q=1$ and the fact that $\varphi _{1}^{A,W}=\varphi ^{A,W}$ is compatible with the ring structure of $A$, on the other hand.  Clearly, if $A$ is graded, then $\varphi _{\ast }^{A,W}$ maps $W^{\otimes p}\otimes A^{q}$ to $A^{q}\otimes W^{\otimes p}.$ In the graded  case we shall denote the restriction of $\varphi _{p}^{A,W}$ to $W^{\otimes p}\otimes A^{q}$ by $\varphi _{p,q}^{A,W}.$

  Let us prove the third part of the proposition. Recall that $\Omega _{\ast}(B)=T_{R}^{c}(\overline{B})$ as $R$-corings. Using the relation  (\ref{ec:phi_ast}) we get
\begin{align*}
  \varphi _{p}^{A,\overline{B}}& =(\varphi _{i}^{A,\overline{B}}\otimes {\Id}_{\overline{B}^{\,\otimes p-i}})\circ ({\Id}_{\overline{B}^{\,\otimes i}}\otimes \varphi _{1}^{A,\overline{B}}\otimes{\Id}_{\overline{B}^{\,\otimes p-i-1}})\circ({\Id}_{\overline{B}^{\,\otimes i+1}}\otimes \varphi  _{1}^{A,\overline{B}}\otimes {\Id}_{\overline{B}^{\,\otimes p-i-2}})\circ ({\Id}_{\overline{B}^{\,\otimes i+2}}\otimes \varphi _{p-i-2}^{A,\overline{B}}),
	\\
  \varphi _{p-1}^{A,\overline{B}}& =(\varphi _{i}^{A,\overline{B}}\otimes {\Id}_{\overline{B}^{\,\otimes p-i-1}})\circ ({\Id}_{\overline{B}^{\,\otimes i}}\otimes \varphi _{1}^{A,\overline{B}}\otimes {\Id}_{\overline{B}^{\,\otimes p-i-2}})\circ ({\Id}_{\overline{B}^{\,\otimes i+1}}\otimes \varphi_{p-i-2}^{A,\overline{B}}).
\end{align*}%
  Since, by assumption, $\varphi _{1}^{A,\overline{B}}=\varphi ^{A,\overline{B}}$ is compatible with the multiplication of $B,$ we deduce that
\begin{equation*}
  \varphi _{p-1}^{A,\overline{B}}\circ ({\Id}_{\overline{B}^{\,\otimes i}}\otimes m_{B}\otimes {\Id}_{\overline{B}^{\,\otimes p-i-2}}\otimes {\Id}_{A})=({\Id}_{A}\otimes {\Id}_{\overline{B}^{\,\otimes i}}\otimes m_{B}\otimes {\Id}_{\overline{B}^{\,\otimes p-i-2}})\circ \varphi _{p}^{A,\overline{B}},
\end{equation*}%
  for any $i\in \{0,\dots ,p-2\}.$ Taking into account the  definition of $\partial _{\ast }$ (see   \S\ref{fa:bar_resolution}), the above identity implies that $\varphi _{\ast }^{A,\overline{B}}$ is a morphism of complexes.

  Let us show that $\varphi _{\ast }^{A,\overline{B}}$ is a differential entwining map, provided that $(A,d^{\ast })$ is a cochain $R$-ring and $\varphi ^{A,\overline{B}}$ is a morphism of complexes. Hence, it remains to prove that
\begin{equation}\label{ec:phi_differential_entwining_map}
  \varphi _{p,q+1}^{A,\overline{B}}\circ ({\Id}_{\Omega _{p}(B)}\otimes d^{q})=(d^{q}\otimes {\Id}_{\Omega _{p}(B)})\circ \varphi _{p,q}^{A,\overline{B}},
\end{equation}%
  for any $p,q\geq 0.$ If $p=0$ then this relation is trivially true for any $q,$ as $\varphi _{0,q}^{A,\overline{B}}$ is the canonical identification $R\otimes A^{q}\cong A^{q}\otimes R.$ For $p=1$ the equation holds as well, since $\varphi _{1,\ast }^{A,\overline{B}}:\overline{B}\otimes A^{\ast}\rightarrow A^{\ast }\otimes \overline{B}$ is a morphism  of complexes  by assumption. Let us suppose that (\ref{ec:phi_differential_entwining_map}) is true for a given $p$ and any $q\geq 0$. By using the recurrence relation that defines $\varphi _{\ast }^{A,\overline{B}}$ and the fact that $\varphi _{p,\ast }^{A,\overline{B}}$ and $\varphi _{1,\ast }^{A,\overline{B}}$ are morphisms of complexes we deduce that (\ref{ec:phi_differential_entwining_map}) holds for $p+1$ and any $q.$ 
\end{proof}

  Proceeding in a similar way one proves the proposition below. Starting with a bimodule map, we now produce examples of twisting maps of (graded or chain) $R$-corings.

\begin{proposition}\label{pr:psi(V,W)}
  Let $\psi ^{V,W}:W\otimes V\rightarrow V\otimes W$ be an $R$-bimodule map.
\begin{enumerate}
  \item There exists a unique graded $R$-bimodule map $\psi _{\ast}^{V,W}:W\otimes T_{R}^{c}(V)\rightarrow T_{R}^{c}(V)\otimes W$ which verifies the relation $\psi _{1}^{V,W}=\psi ^{V,W} $ and is compatible with the coring structure of $T_{R}^{c}(V)$. 

  \item If $C$ is a $R$-coring and $\psi^{V,C}$ is compatible with the comultiplication and the counit of $C$, then $\psi _{\ast }^{V,C}$ is a twisting map of corings. Moreover, $\psi _{\ast }^{V,C}$ is graded, provided that $C$ is graded and $\psi ^{V,C}$ maps $C_{p}\otimes V$ to $V\otimes C_{p}$ for all $p$. In this case we shall use the notation $\psi ^{V,C}_{p,q}:=\psi^{V,C}_{p}\vert_{C_{p}\otimes V^{\ot q}}$.

  \item If, in addition $A$ is a connected $R$-ring and $\psi _{\ast }^{\overline{A},C}$ is compatible with the comultiplication and the counit of $C$, then the twisting map $\psi _{\ast}^{\overline{A},C}$ is a morphism of chain complexes from $(C\otimes \Omega _{\ast }(A),{\Id}_{C}\otimes \partial _{\ast })$ to $(\Omega _{\ast}(A)\otimes C,\partial_{\ast }\otimes {\Id}_{C}).$

  \item If $(C,d_{\ast })$ is a chain coring and $\psi^{\overline{A},C}$ is a morphism of complexes as in (3), then $\psi_{\ast }^{\overline{A},C}$ is a twisting map of chain corings.
\end{enumerate}
\end{proposition}

\begin{remark}\label{re:left-right_symmetry}
  If we apply  Proposition \ref{pr:phi(V,W)} and Proposition  \ref{pr:psi(V,W)} to the same map $\varphi ^{V,W}=\psi ^{V,W},$  then we get two different morphisms $\varphi _{\ast }^{V,W}$ and $\psi _{\ast }^{V,W}$, that can be seen as left-right symmetric version of each other. By analogy, we can also define the morphisms
\begin{equation*}
  \widetilde{\varphi }{}_{\ast }^{V,B}:B\otimes T_{R}^{c}(V)\rightarrow T_{R}^{c}(V)\otimes B\quad \text{and\quad}\widetilde{\varphi }{}_{\ast}^{\overline{A},B}:B\otimes \Omega _{\ast}(A)\rightarrow \Omega _{\ast}(A)\otimes B,
\end{equation*}%
  the symmetric versions of $\varphi ^{A,W}$ and $\varphi ^{\overline{A},B},$ respectively. Again by symmetry, there are morphisms
\begin{equation*}
  \widetilde{\psi {}}{}_{\ast }^{D,W}:T_{R}^{c}(W)\otimes D\rightarrow D\otimes T_{R}^{c}(W) \text{\quad and\quad} \widetilde{\psi}{}_{\ast }^{\,D,\overline{B}}:\Omega _{\ast}(B)\otimes D\rightarrow D\otimes \Omega_{\ast }(B),
\end{equation*}%
  for any $R$-coring $D.$ For the sake of completeness, let us mention that the results from Proposition \ref{pr:phi(V,W)} and Proposition \ref{pr:psi(V,W)} can be easily dualized. In this way we obtain twisted tensor products of (graded or chain) $R$-rings in which one of the factors is either $T_{R}^{a}(V)$ or $\Omega  ^{\ast }(C)$. We do not state the dual results in detail, as we shall not use them in this paper.
\end{remark}

\begin{fact}[The entwining maps $\protect\lambda $ and $\protect\nu .$]\label{fa:Lambda_Nu}
  Let  $A$ and $B$ be two connected strongly graded $R$-rings. Then it makes sense to consider the almost-Koszul pairs $(A,T(A))$ and $(B,T(B)).$  For ease of notation,  we shall write $C$ and $D$ instead of $T(A)$ and $T(B)$, respectively. Recall that, by definition, $(A,C)$ and $(B,D)$ are endowed with two $R$-bimodule isomorphisms  $\theta _{C,A} : C_{1}\rightarrow A^{1}$ and $\theta _{D,B} : D_{1}\rightarrow B^{1}$.

  We now assume, in addition, that $\sigma: B\ot A\to A\ot B$ is a given invertible graded twisting map. Obviously, the inverse  $\sigma^{-1}$ of $\sigma$ is also a twisting map or rings. Note that $(\sigma^{-1})^{p,q}=(\sigma^{q,p})^{-1}$. We claim that, under these assumptions, there is a graded entwining map $\lambda: C\ot B\to B\ot C$ that extends in a  certain sense the inverse of $\sigma$. Indeed, by taking $\varphi^{B,\overline{A}} := \sigma^{-1}\vert_{\overline{A}\ot B}$  in Proposition  \ref{pr:phi(V,W)}, we get an entwining map $\varphi^{B,\overline{A}}_{\ast}$ between $\Omega_{\ast}(A)$ and $B$, which is a morphism of complexes. Since any right $R$-module is flat we have
\[
  \Ho_{p}(\Omega_{\ast}(A)\ot B)\cong  \Ho_{p}(\Omega_{\ast}(A))\ot B= C_p\ot B \quad \text{and} \quad \Ho_{p}(B\ot \Omega_{\ast}(A))\cong B  \ot  \Ho_{p}(\Omega_{\ast}(A))=B\ot C_p,
\]
  Therefore, by Proposition \ref{pr:cotwist} (2), the induced morphism $\lambda: C\ot B\to B\ot C$ is a graded entwining map.

  By symmetry (see Remark \ref{re:left-right_symmetry}), if we take $\varphi^{\overline{B},A}$ to be the restriction of $\sigma^{-1}$ to $A\ot \overline{B}$, then $\widetilde{\varphi}{}^{\,\overline{B},A}_{\ast}$ induces another graded entwining map $\nu:A\ot D\to D\ot A$.
\end{fact}

\begin{fact}[The twisting map $\protect\tau' $.]\label{fa:Tau}
  Under the same  assumptions as above, we can now construct a  twisting map $\tau$ between  $C$ and $D$. We apply Proposition \ref{pr:psi(V,W)} for $\psi^{\overline{B},C}:=\lambda\vert_{C\ot \overline{B}}$ to get a twisting map $\psi^{\overline{B},C}_{\ast}:C\ot\Omega_{\ast} (B)\to\Omega_{\ast}(B)\ot C$   of graded corings. In fact, if we regard  $C$ as a chain coring with trivial differential maps, then     $\psi^{\overline{B},C}_{\ast}$ is a twisting map of chain corings, so it induces a graded twisting map $\tau':C\ot D\to D\ot C$, cf. Proposition \ref{pr:cotwist} (3).

  Some useful properties of $\lambda$, $\tau'$ and $\nu$ are collected in the theorem below.
\end{fact}

\begin{theorem}\label{th:Tau_Lambda_Nu}
  Let   $\sigma :B\otimes A\rightarrow A\otimes B$  be an invertible graded twisting map. The twisting map $\tau' $ and the entwining maps $\lambda$ and $\nu$ constructed in  \S\ref{fa:Tau} and \S\ref{fa:Lambda_Nu} verify the following relations:
\begin{align}
  ({\Id}_{B^{p}}\otimes \theta _{C,A})\circ \lambda _{1}^{p} &  =\left(\sigma ^{p,1}\right) ^{-1}\circ (\theta  _{C,A}\otimes {\Id}_{B^{p}}),    \label{Cond1} 
	\\
  \lambda _{p}^{1}\circ ({\Id}_{C_{p}}\otimes \theta _{D,B}) & =(\theta_{D,B}\otimes {\Id}_{C_{p}})\circ \tau_{p,1}',  \label{Cond2}
	\\
  \nu _{p}^{1}\circ (\theta _{C,A}\otimes {\Id}_{D_{p}})& =({\Id} _{D_{p}}\otimes \theta _{C,A})\circ \tau_{1,p}'.  \label{Cond3}
	\\
  (\theta_{D,B}\ot\Id_{A^p})\circ\nu^p_1& =\left(\sigma^{1,p}\right)^{-1} \circ(\Id_{A^p}\ot\theta_{D,B}).\label{Cond4}
\end{align}
\end{theorem}

\begin{proof}
  We know that $C=T(A)$ is the homology of $(\Omega_{\ast}(A),\partial_{\ast})$, so  $C_1=\overline{A}/{\overline{A}}^2$. Let us denote the class of $a\in\overline{A}$ by $[a]$. By definition, $\varphi^{B,\overline{A}}_1=\sigma^{-1}\vert_{\overline{A}\ot B}$. Therefore, for any $a\in\overline{A}$ and $b\in B^q$, we have 
  \[
    \lambda_1 ([a]\ot b)=\sum b_{\sigma^{-1}}\ot [a_{\sigma^{-1}}].
  \]
  On the other hand, $\theta_{C,A}$ maps $[a]$ to the homogeneous component of degree $1$ of $a$. The equation \eqref{Cond1} now follows by a simple computation.

  To prove the second identity we first note that  $D_1=\overline{B}/{\overline{B}}^2$ and $\psi^{\overline{B},C}_1=\psi^{ \overline{B},C} =\lambda\vert_{C\ot \overline{B}}$. Since $\tau'$ is the morphism induced by $\psi^{ \overline{B},C}_{\ast}$ and $\theta_{D,B}$ maps the class of $b\in \overline{B}$ modulo $\overline{B}^2$ to its homogeneous component of degree $1$, we conclude the proof as in the case of the previous relation. To show that \eqref{Cond3} holds one proceeds in a similar way.
\end{proof}

\begin{fact}[Notation and assumptions.]\label{notation}
  Our goal is to show that the pair $(A\otimes _{\sigma}B,T(A) \otimes _{{\tau} }T(B))$ is almost-Koszul, see Remark \ref{re:teuh} for the definition of ${\tau}$. Then we shall show that this pair is Koszul, provided that $A$ and $B$ are Koszul $R$-rings. In fact we are able to prove these results for any almost-Koszul pairs $(A,C)$ and $(B,D)$ which are equipped with the following extra data:
\begin{enumerate}
  \item An invertible graded twisting map $\sigma :B\otimes A \rightarrow A\otimes B$. 

  \item An invertible graded twisting map $\tau' :C\otimes D\rightarrow D\otimes C$.

  \item An invertible entwining map $\lambda :C\otimes B\rightarrow B\otimes C$.

  \item An invertible entwining map $\nu :A\otimes D\rightarrow D\otimes A.$
\end{enumerate}
  \noindent We assume that the conditions \eqref{Cond1}-\eqref{Cond4} are satisfied, where $\theta _{C,A}:C_{1}\rightarrow A^{1}$ and $\theta_{D,B}:D_{1}\rightarrow B^{1}$ are the isomorphisms corresponding to $(A,C)$ and $(B,D),$ respectively. We have already seen that, for any invertible twisting map $\sigma :B\otimes A\rightarrow A\otimes B,$ the pairs $(A,T(A))$ and $(B,T(B))$ fulfill the conditions \eqref{Cond1}-\eqref{Cond4}, where $\tau'$, $\lambda$ and $\nu$ are as in Theorem \ref{th:Tau_Lambda_Nu}.   

  In the case when $p=q=1$ the above identities imply the following equation:
\begin{equation*}
	\sigma ^{1,1}\circ \left( \theta _{D,B}\otimes \theta _{C,A}\right) \circ \tau _{1,1}'=\theta _{C,A}\otimes \theta _{D,B}.
\end{equation*}%
  Equivalently, for $c\in C_{1}$ and $d\in D_{1},$ we have
\begin{equation}\label{ec:cond1}
	\sum \theta _{C,A}(c_{\tau' })_{\sigma }\otimes \theta _{D,B}(d_{\tau'})_{\sigma }=\theta _{C,A}(c)\otimes \theta _{D,B}(d).
\end{equation}
\end{fact}

\begin{remark}\label{re:teuh}
  If $\tau' :C\otimes D\rightarrow D\otimes C$ is a graded twisting map of graded $R$-corings, then the map ${\tau}$ defined by ${\tau}_{p,q}:=(-1)^{pq}\tau _{p,q}'$ is also a graded twisting map between $C$ and $D$.
\end{remark}

\begin{proposition}\label{prekoszul}
  With the notation and assumptions from \S \emph{\ref{notation}} and the preceding remark, the pair $(A\otimes _{\sigma }B,C\otimes _{{\tau}}D)$ is almost-Koszul.
\end{proposition}

\begin{proof}
  It is obvious that $A\otimes _{\sigma }B$ and $C\otimes_{{\tau}}D$ are connected. By definition we have 
\begin{equation*}
  (A\otimes _{\sigma }B)^{1}=(R\otimes B^{1})\oplus (A^{1}\otimes R)\quad \text{and}\quad (C\otimes _{{\tau}}D)_{1}=(R\otimes D_{1})\oplus (C_{1}\otimes R).
\end{equation*}%
  We define $\theta :(C\otimes _{{\tau}}D)_{1}\rightarrow (A\otimes_{\sigma }B)^{1}$ such that it coincides with $(\theta_{C,A}\otimes {\Id}_{R})$ and $({\Id}_{R}\otimes \theta_{D,B})$ on $C_{1}\otimes R$ and $R\otimes D_{1},$ respectively. We claim that $\theta $ satisfies the relation \eqref{ec:almost-koszul}. Indeed, if $\cdot _{\sigma }$ denotes the multiplication on $A\otimes _{\sigma }B$, then we have to show that
\begin{equation}\label{koszulcondition}
	\sum\limits \theta (x_{(1,1)})\cdot _{\sigma }\theta (x_{(2,1)})=0,
\end{equation}%
  for any $\,x$ in $(C\otimes _{{\tau}}D)_{2}=(C_{2}\otimes R)\oplus (C_{1}\otimes D_{1})\oplus (R\otimes D_{2}).$ Hence for  proving \eqref{koszulcondition} we may assume that $x$ belongs to one of the three direct summands. Let us consider the case $x\in C_{2}\otimes R$, so $x=c\otimes 1$ for some $c\in C_{2}$. By definition of the comultiplication on $C\otimes _{{\tau}}D$, we have
\begin{equation*}
	\Delta _{1,1}(c\otimes 1)=\sum\limits (c_{(1,1)}\otimes 1)\otimes (c_{(2,1)}\otimes 1).
\end{equation*}%
  Henceforth, in this case we have
\begin{equation*}
	\sum\limits \theta (x_{(1,1)})\cdot _{\sigma }\theta (x_{(2,1)})=\sum\limits (\theta _{C,A}(c_{(1,1)})\otimes 1)\cdot _{\sigma }(\theta _{C,A}(c_{(2,1)})\otimes 1)=\sum\limits \theta _{C,A}(c_{(1,1)})\theta_{C,A}(c_{(2,1)})\otimes 1=0,
\end{equation*}%
  since $(A,C)$ is an almost-Koszul pair. If $x\in C_{0}\otimes D_{2}$, the computations are done in a similar way.

  Let us finally assume that $x=c\otimes d$ with $c\in C_{1}$, $d\in D_{1}$. Since $c\in C_{1}$ we have $\Delta {c}=1\otimes c+c\otimes 1,$ and ${\tau}_{1,1}=-\tau _{1,1}'$. Thus 
\begin{equation*}
	\Delta (c\otimes d)=(c\otimes d)\otimes (1\otimes 1)+(c\otimes 1)\otimes (1\otimes d)+\ (1\otimes 1)\otimes (c\otimes d)-1\otimes \tau_{1,1}'(c\otimes d)\otimes 1.
\end{equation*}%
	The component of the latest expression belonging to $(C\otimes_{{\tau}}D)_{1}\otimes (C\otimes_{{\tau}}D)_{1}$ is precisely
\begin{equation*}
  \Delta _{1,1}(c\otimes d)=(c\otimes 1)\otimes (1\otimes d)-\sum\limits (1\otimes d_{\tau' })\otimes (c_{\tau' }\otimes 1).
\end{equation*}%
  Henceforth, applying first $\theta \otimes \theta $ and then the product in $A\otimes _{\sigma }B$, we get
\begin{equation*}
  \sum\limits \theta (x_{(1,1)})\cdot _{\sigma }\theta (x_{(2,1)})=\theta_{C,A}(c)\otimes \theta _{D,B}(d)-\sum\limits \theta _{C,A}(c_{\tau'})_{\sigma }\otimes \theta _{D,B}(d_{\tau'})_{\sigma }.
\end{equation*}%
  In view of the relation \eqref{ec:cond1} it follows that the equation \eqref{ec:almost-koszul} holds in this case as well.
\end{proof}

\begin{theorem}\label{thm:koszul}
  We keep the notation and the assumptions from  \S\emph{\ref{notation}}. If $(A,C)$ and $(B,D)$ are Koszul, then $(A\otimes _{\sigma}B,C\otimes _{{\tau}}D)$ is Koszul too.
\end{theorem}

\begin{proof}
  We have already proved that $(A\otimes_{\sigma}B,C\otimes_{{\tau}}D)$ is an almost-Koszul pair. Let $\K_{\ast}$ be the complex that is  obtained from $\K_{\ast }^{l}(A\otimes_{\sigma}B,C \otimes_{{\tau}}D)$ dropping the part in degree $-1$. We define $(\K_{\ast }', d_{\ast }')$ and $(\K_{\ast }'',d_{\ast }'')$ in a similar way from $\K_{\ast }^{l}(A,C)$ and $\K_{\ast }^{l}(B,D)$, respectively. We claim that $\Id_A\ot\lambda\ot\Id_D:\K_\ast'\ot\K_\ast''\to \K_\ast$ is an isomorphism of complexes. Let
\begin{equation*}
  \partial _{n}:(A\otimes _{\sigma }B)\otimes (C\otimes_{{\tau}}D)_{n}\rightarrow (A\otimes _{\sigma}B)\otimes (C\otimes _{{\tau}}D)_{n-1}
\end{equation*}
  denote the differential map in $\K_{\ast }$. We fix $p$ and $q$ such that $p+q=n$. For $c\in C_{p}$ and $d\in D_{q}$ we have
\begin{equation*}
  \Delta (c\otimes d)=\sum\limits_{u=0}^p\sum \limits_{v=0}^q(-1)^{(p-u)v}c_{(1,u)}\otimes {d_{(1,v)}}_{\tau'}\otimes {c_{(2,{p-u})}}_{\tau'}\otimes d_{(2,q-v)}.
\end{equation*}%
  Thus the component of $\Delta (c\otimes d)$ in $(C\otimes_{{\tau}}D)_{1}\otimes (C\otimes_{{\tau}}D)_{p+q-1}$ is obtained from the above equality by dropping all summands but the ones with either $u=1$ and $v=0,$
  or $u=0$ and $v=1$. Therefore,
\begin{equation*}
    \Delta _{1,p+q-1}(c\otimes d)=\sum\limits c_{(1,{1})}\otimes 1\otimes c_{(2,{p-1})}\otimes d+(-1)^{p}\sum\limits 1\otimes {d_{(1,{1})}}_{\tau' }\otimes c_{\tau' }\otimes d_{(2,{q-1})}.
\end{equation*}%
  Hence, for any $\zeta=a\otimes b\otimes c\otimes d$ in $A\otimes B\otimes C_p\otimes D_q$, with $p+q=n$, we get
\[
  \partial _{n}\left(\zeta\right) =\sum\limits a\theta _{C,A}(c_{(1,{1})})_{\sigma}\otimes b_{\sigma}\otimes c_{({2},{p-1})}\otimes d+\sum\limits (-1)^{p}a\otimes b\theta_{D,B}({d_{(1,{1})}}_{\tau' })\otimes c_{\tau' }\otimes d_{({2},{q-1})}.
\]
  To make computations with morphisms in the category of $R$-bimodules we use string representation of morphisms in a tensor category, which is explained for example in \cite[Chapter XIV.1]{Kassel95a}. Each morphism will be represented downwards, as a black bead. Sometimes, to avoid confusion, we shall write the name of the morphism near the corresponding bead. For the identity of a bimodule we shall draw only the string. The tensor  product and the composition of two morphisms will be represented by horizontal and vertical juxtaposition, respectively. In conclusion, every string diagram may be interpreted as the representation of a composition $f_{1}\circ \cdots \circ f_{n},$ where each $f_{i}$ is a tensor product $f_{i}={\Id}_{X_{i}}\otimes g_{i}\otimes {\Id}_{Y_{i}}.$ The corresponding diagrams will be drawn one under the other, starting with $f_{n}$ on the top.

  As usual, the multiplication of an $R$-ring is drawn by joining two strings. For the components $\Delta_{p,q}$ of the comultiplication of a coring $C$ we shall use the `dual' representation, in which the string representing $C_{p+q}$ is split in two strings that corresponds to $C_p$ and $C_q$, respectively.

  As an example, let us have a look at the picture below, which represents $\partial _{n}$. Here the beads symbolize the morphism $\theta_{D,B}$ and $\theta_{C,A}$, respectively. Note the notation of $\sigma$ as a crossing. For $\tau'$ and $\lambda $ we shall use the inverse crossing representation, to  put stress on the fact that they were obtained using $\sigma ^{-1}$.
 \begin{equation*}
   \begin{array}{c}
     \includegraphics{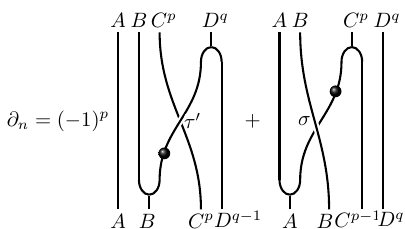}  
   \end{array}
   \label{f1}
\end{equation*} 
  Let $\delta _{n}:=({\Id}_{A}\otimes \lambda^{-1} \otimes {\Id}_{D})\circ \partial _{n}\circ ({\Id}_{A}\otimes \lambda \otimes {\Id}_{D})$. Then $\delta _{n}=\delta _{n}^{\prime }+\delta _{n}^{\prime \prime }, $ where $\delta _{n}^{\prime }$ is the first term of $\partial _{n}$ composed to the left by ${\Id}_{A}\otimes \lambda^{-1} \otimes {\Id}_{D}$ and to the right by ${\Id}_{A}\otimes \lambda \otimes {\Id}_{D}.$ The map $\delta _{n}^{\prime \prime }$ is obtained in a similar way from the second term of $\partial _{n}.$ 
  
  The computation of $\delta _{n}^{\prime }$ is performed in the diagram \eqref{f2}. For the first equality we are using (\ref{Cond2}). The second one means that $\lambda $ is compatible with the multiplication of $B, $ while the third one is obvious, as $\lambda $ and $\lambda ^{-1}$ are inverses each other.
\begin{equation}
  \begin{array}{c}
    \includegraphics[scale=1]{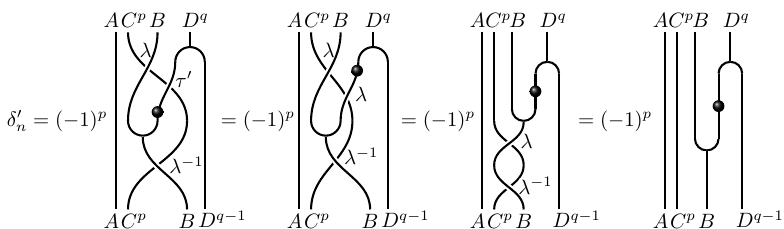}
  \end{array}\label{f2}
\end{equation}
  The morphism $\delta _{n}^{\prime \prime }$ is computed in diagram \eqref{f3} using the same method. To deduce the first identity we use (\ref{Cond1}). The second equality follows by the fact that the coring structure of $C$ and the entwining map $\lambda$ are compatible. The third one is obvious, as $\lambda ^{-1}$ is the inverse of $\lambda.$
\begin{equation}
  \begin{array}{c}
    \includegraphics{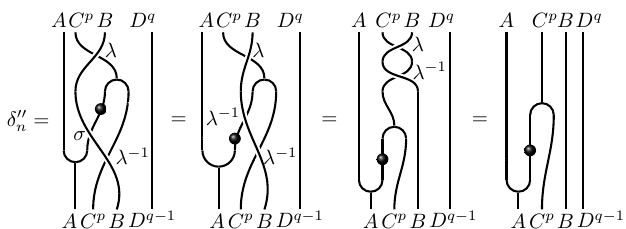}
  \end{array}\label{f3}
\end{equation}
  The computations in \eqref{f2} and \eqref{f3} shows that $\delta _{n}^{\prime}=(-1)^{p }{\Id}_{\K_{p }^{\prime}}\otimes d^{\prime \prime }_q$ and $\delta _{n}^{\prime \prime}=d^{\prime }_p\otimes {\Id}_{\K_{q}^{\prime \prime}}$ for any $p$ and $q$ such that $p+q=n$. Therefore, $\K_{\ast}$ is the tensor product of the  complexes $\K_{\ast }'$ and $\K_{\ast}''.$

  Let us now assume that $(A,C)$ and $(B,D)$ are Koszul. By definition, then the complexes $\K_{\ast }'$ and $\K_{\ast}''$ are acyclic and their homology groups in degree zero are isomorphic to $R.$ If $\K_{\ast}$ has the same properties, then $\K_{\ast }^{l}(A\otimes _{\sigma}B,C\otimes_{{\tau}}D)$ will be exact. Using K\"{u}nneth formula \cite[Theorem 3.6.3]{We} and the fact that $R$ is semisimple, we get
\begin{equation*}
  \mathrm{H}_{n}(\K_{\ast }^{l}(A\otimes _{\sigma}B,C\otimes_{{\tau}}D))\cong\textstyle\bigoplus\limits\limits_{p+q=n} \mathrm{H}_{p}(\K_{\ast}^{l}(A,C))\otimes \mathrm{H}_{q}(\K_{\ast }^{l}(B,D)).
\end{equation*}%
  Thus $\K_{n}^{l}(A\otimes _{\sigma }B,C\otimes_{{\tau}}D)$ clearly is acyclic. On the other hand,
\begin{equation*}
  \mathrm{H}_{0}(\K_{\ast }^{l}(A\otimes _{\sigma }B,C\otimes_{{\tau}}D))\cong \mathrm{H}_{0}(\K_{\ast}^{l}(A,C))\otimes \mathrm{H}_{0}(\K_{\ast}^{l}(B,D))\cong R\otimes R\cong R.\qedhere
\end{equation*}%

\end{proof}

\begin{corollary}
  Let $\sigma :B\otimes A\rightarrow A\otimes B$ be an invertible graded twisting map between two Koszul $R$-rings. Then $A\otimes _{\sigma }B$ is Koszul and $T(A\otimes _{\sigma}B)\cong T(A)\otimes _{{\tau}}T(B).$
\end{corollary}

\begin{proof}
  Let $\tau'$, $\lambda$ and $\nu$ be the maps that we constructed in \S\ref{fa:Lambda_Nu} and \ref{fa:Tau}. By the preceding theorem and Theorem \ref{te:caracterizare_Koszul}, the pairs  $\left(A\ot_\sigma B,T(A)\ot_{{\tau}}T(B)\right)$ and $\left(A\ot_\sigma B,T(A\ot_{\sigma}B)\right)$ are Koszul. To end the proof we apply Corollary \ref{thm:CharacterizationKoszul}.
\end{proof}

\begin{corollary}\label{cor:hdim}
  Let $A$ and $B$ be Koszul rings over a separable $\Bbbk$-algebra $R$. Let $\sigma$ be an invertible graded twisting map between $A$ and $B$. If $\Hd A = n$ and $\Hd B =m$, then $\Hd A\otimes_\sigma B\leq n+m$. Moreover, $\Hd(A\otimes_\sigma B)= m+n$ if and only if $T_n(A)\otimes T_m(B)\neq 0$. 
\end{corollary}

\begin{proof}
  By Theorem \ref{thm:Hdim}, we have $T_n(A)\neq 0$, $T_m(B)\neq 0$ and $T_p(A) = 0 = T_q(B)$, for all $p>n$ and $q>m$. It follows that the homogeneous component of degree $m+n$ of $T(A)\otimes_\tau T(B)$ is $T_n(A) \otimes T_m(B)$. Since the homogeneous component of degree $r>m+n$ of the twisted tensor product coring is zero and $A\ot_{\sigma}B$ is a Koszul ring, we conclude the proof using  Theorem \ref{thm:Hdim} once again. 
\end{proof} 

\section{The Hochschild (co)homology of a twisted tensor product}

  As usual, when we speak about the Hochschild (co)homology of an $R$-ring we assume that $R$ is a separable algebra over a field $\Bbbk$ and that any $R$-bimodule is symmetric as a $\Bbbk$-bimodule (with respect to the bimodule structure induced by restriction of scalars). Let $\sigma:B\ot A \to A\ot B$ be an invertible graded twisting map between two Koszul $R$-rings. By Theorem \ref{th:Tau_Lambda_Nu} and Theorem \ref{thm:koszul} there is a twisting map $\tau$ between the corings $T(A)$ and $T(B))$ such that $\left(A\ot_{\sigma} B,T(A)\ot_{{\tau}} T(B)\right)$ is a Koszul pair. In the first part of this  section, for more flexibility,  instead of using this pair as a tool we prefer to place ourselves in the setting of \S\ref{notation}, adding the assumption that $(A,C)$ and $(B,D)$ are Koszul.

  Our aim is to show that the Hochschild homology of
  $A\ot_{\sigma} B$ can be computed as the homology of the total complex associated to a suitable double complex. A similar result will be obtained for Hochschild cohomology. Then, as an application, we introduce generalized Ore extensions of an $R$-ring $A$ and we show that such an extension is Koszul, provided that $A$ is so. Furthermore, using the results from this section, we investigate the  Hochschild cohomology of generalized Ore extensions.

\begin{fact}[The Hochschild homology of twisted tensor products of Koszul rings.] \label{fact:dcomplex}
  By Theorem \ref{thm:koszul} the pair $(A\otimes _{\sigma}B,C\otimes _{{\tau }}D)$ is Koszul, so we can use Theorem \ref{thm:HochHom} to compute the Hochschild homology of $A\otimes_{\sigma }B$. The decomposition 
\begin{equation*}
  (C\otimes_{{\tau}}D)_{n}=\textstyle\bigoplus\limits_{p+q=n}C_{p}\otimes D_{q}
\end{equation*}%
	suggests that $\K_{\ast }(A\otimes _{\sigma }B,M)$ is the total complex of a double complex $\left( \K_{\ast \ast }(A\otimes_{\sigma }B,M),\partial_{\ast\ast }^{h},\partial_{\ast \ast }^{v}\right) $. Before proving this, let us  write the latter chain complex in an equivalent way. Recall that 
\begin{equation*}
  \K_{n}(A\otimes _{\sigma }B,M) = M{\,\widehat{\otimes }\,}(C\otimes_{{\tau}}D)_{n} \cong\textstyle\bigoplus\limits_{p+q=n}(M{\,\widehat{\otimes}\,} C_{p}\wot D_{q})\cong \textstyle \bigoplus\limits_{p+q=n}D_{q}{\widehat{\otimes}\,}M{\,\widehat{\otimes}\,}C_{p}=\K_{n}^{\prime }(A\otimes_{\sigma}B,M).
\end{equation*}%
  Through these identifications, to the differential map $\partial _{n}$ of $\K_{\ast }(A\otimes _{\sigma }B,M)$ corresponds a morphism $\partial_{n}^{\prime }:\K_{n}^{\prime }(A\otimes _{\sigma }B,M)\rightarrow \K_{n-1}^{\prime }(A\otimes_{\sigma }B,M)$. In view of Theorem \ref{thm:HochHom}, to compute $\partial'_n$ we need $\Delta _{1,n-1}^{C\otimes_{{\tau}}D}(c\otimes d)$ and $\Delta_{n-1,1}^{C\otimes_{{\tau}}D}(c\otimes d)$ for any $c\otimes d\in C_{p}\otimes D_{q}$ with $p+q=n$. The first element is given by the formula
\begin{equation*}
  \Delta _{1,n-1}^{C\otimes _{{\tau}}D}(c\otimes d)=\sum c_{(1,1)}\otimes 1\otimes c_{(2,{p-1})}\otimes d+(-1)^{p}\sum 1\otimes {d_{(1,1)}}_{\tau'}\otimes c_{\tau' }\otimes d_{(2,{q-1})},
\end{equation*}
  cf. the proof of Theorem \ref{thm:koszul}. We can compute the second element in a similar way, obtaining
\begin{equation*}
  \Delta _{n-1,1}^{C\otimes _{{\tau}}D}(c\otimes d)=\sum c\otimes d_{(1,{q-1})}\otimes 1\otimes d_{(2,1)}+(-1)^{q}\sum  c_{(1,{p-1})} \otimes d_{\tau' }\otimes {c_{(2,1)}}_{\tau'}\otimes 1.
\end{equation*}%
  Let $m\in M.$ Taking into account the relation (\ref{eq:pdn}) it follows that $\partial _{n}^{\prime }(d{\,\widehat{\otimes}\,m\,\widehat{\otimes } \,c)}$ can be written as a sum of two elements. The first one, belonging to $D_{q}{\,\widehat{\otimes }\,}M{\,\widehat{\otimes }\,} C_{p-1},$ has the following form
\begin{align}
  \partial _{p,q}^{h}(d{\,\widehat{\otimes }\,}m{\,\widehat{\otimes }\,}c)& =\sum \;d{\,\widehat{\otimes}\,}m\left( \theta _{C,A}(c_{(1,1)}) \otimes 1\right) {\,\widehat{\otimes }\,}c_{(2,{p-1})}+  \label{dh}
			\\
  & \hspace*{0.75cm}+(-1)^{p}\sum d_{\tau' }{\,\widehat{\otimes}\,}\left({\theta }_{C,A}({c_{(2,1)}}_{\tau'})\otimes 1\right) m{\,\widehat{\otimes}\,}c_{(1,{p-1})}.  \notag
\end{align}%
  The other one, which is an element of $D_{q-1}{\,\widehat{\otimes}\,}M{\,\widehat{\otimes }\,}C_{p},$ can be written as $(-1)^{p}\partial_{p,q}^{v}(d{\,\widehat{\otimes}\,}m{\,\widehat{\otimes
}\,}c),$ where
\begin{align}
  \partial _{p,q}^{v}(d{\,\widehat{\otimes }\,}m{\,\widehat{\otimes }\,}c) & =\sum \;d_{(2,{q-1})}{\,\widehat{\otimes }\,}m\left(1\otimes \theta_{D,B}({d_{(1,1)}}_{\tau' })\right) {\,\widehat{\otimes}\,} c_{\tau'} 
	 \label{dv}
				\\
  & \hspace*{0.75cm}+(-1)^{q}\sum d_{(1,{q-1})}{\,\widehat{\otimes}\,}\left(1\otimes \theta _{D,B}(d_{(2,1)})\right) m{\,\widehat {\otimes }\,} c. \notag
\end{align}%
 Thus the complexes $\left(\K_{\ast}(A\otimes_{\sigma}B,M),\partial _{\ast }\right) $ and $\left(\K_{\ast }^{\prime }(A\otimes_{\sigma }B,M),\partial_{\ast }^{\prime }\right) $ are isomorphic. We have also proved that the restriction of the differential map $\partial _{n}^{\prime }$ to $D_{q}{\,\widehat{\otimes }\,}M{\,\widehat{\otimes }\,}C_{p}$ satisfies the relation $\partial_{n}'=\partial _{p,q}^{h}+(-1)^{p}\partial _{p,q}^{v}$, where     $\partial_{p,q}^{h}$ and $\partial_{p,q}^{v}$ are defined as above. Let $\K_{p,q}(A\otimes _{\sigma }B,M):=D_{q}{\widehat{\otimes}\,}M{\,\widehat{\otimes }}C_{p}.$ By a straightforward but tedious computation, based on the relations \eqref{Cond1}-\eqref{Cond4}, one shows that $\left( \K_{\ast \ast}(A\otimes _{\sigma }B,M),\partial_{\ast \ast}^{h},\partial _{\ast \ast}^{v}\right) $ is a double complex, that is the diagram
\begin{equation}
  \begin{array}{c}
   \xymatrix{
		\K_{p,q}(A\otimes _{\sigma }B,M) \ar[d]_{\pd^v_{p,q}} 				\ar[rr]^{\pd^h_{p,q}} && \K_{p-1,q}	(A\otimes_{\sigma}B,M) 		\ar[d]^{\pd^v_{p-1,q}}
				\\
    \K_{p,q-1}(A\otimes _{\sigma }B,M)\ar[rr]_{\pd^h_{p,q-1}} && 		\K_{p-1,q-1}(A\otimes _{\sigma }B,M)
    }
  \end{array}\label{eq:dcomplex}
\end{equation}%
  is commutative for all nonzero $p$ and $q$. Obviously, its  total complex is $\left(\K_{\ast}^{\prime}(A\otimes _{\sigma}B,M),\partial_{\ast}^{\prime }\right) $.

  We reinterpret the double complex \eqref{eq:dcomplex} to relate the homology groups of the rows and columns with the Hochschild homology groups of $A$ and $B$, respectively. First, let us notice that $D_{q}\otimes M$ is an $A$-bimodule with respect to the actions
\begin{equation*}
 a (d\otimes m)a':=\sum d_{\nu }\otimes (a_{\nu }\otimes 1)m(a'\otimes 1),
\end{equation*}%
  where for the left module structure we used the entwining map $\nu :A\otimes D\rightarrow D\otimes A$.
  Similarly, we can endow $M\otimes C_{p}$ with a $B$-bimodule structure by
\begin{equation*}
  b(m\otimes c)b':=\sum (1\otimes b)m(1\otimes b_{\lambda }')\otimes c_{\lambda}.
\end{equation*}%
  Using the relations \eqref{Cond2} and \eqref{Cond3}, it follows that the complexes $\K_{\ast }(B,M\otimes C_{p})$ and $\K_{\ast}(A,D_{q}\otimes M)$ are isomorphic to $\K_{p\ast }(A\otimes_{\sigma}B,M)$ and $\K_{\ast q}(A\otimes_{\sigma}B,M)$, respectively.

  There are two filtrations on the total complex of $\K_{\ast\ast}(A\otimes _{\sigma }B,M)$, cf. \cite[Section 5.6]{We}. They give rise to two spectral sequences, both converging to the Hochschild homology of $A\otimes _{\sigma }B$ with coefficients in $M.$ For future reference, we summarize the above results in the theorem below.
\end{fact}

\begin{theorem}\label{thm:DoubleComplexHom}
  We keep the notation and the assumptions of \S \ref{notation}. If $(A,C)$ and $(B,D)$ are Koszul pairs over a separable $R$-algebra and $M$ is an $A\otimes_{\sigma }B$-bimodule, then the Hochschild homology of $A\otimes _{\sigma }B$ with coefficients in $M$ is the homology of the total complex of $\left(\K_{\ast\ast }(A\otimes _{\sigma }B,M),\partial _{\ast\ast}^{h},\partial_{\ast \ast }^{v}\right) $. The pages $^{I}E_{\ast \ast }^{1}$ and $^{II}E_{\ast \ast }^{1}$ of the spectral sequences that correspond to the column-wise and row-wise filtrations are 
\begin{equation}\label{eq:spectralseq}
	^{I}E_{pq}^{1}=\hh_{q}\left( B,M\otimes C_{p}\right) \quad \text{and\quad}^{II}E_{pq}^{1}=\hh_{q}\left( A,D_{p}\otimes M\right).
\end{equation}%
  Both spectral sequences converge to the Hochschild homology of  $A\otimes_{\sigma }B$ with coefficients in $M.$
\end{theorem}

\begin{fact}[The Hochschild cohomology of twisted tensor products of Koszul rings.]
  For the computation of Hochschild cohomology of $A\otimes_{\sigma }B$ with coefficients in a bimodule $M$ one may also use a similar double complex. With the notation from \S\ref{fact:dcomplex}, let us remark that
\begin{equation}\label{dec_Hoch_coh}
  \K^{n}(A\otimes _{\sigma }B,M)=\Hom_{R^{e}}(\left(C\otimes_{{\tau}}D\right) _{n},M)\cong \textstyle \bigoplus \limits_{p+q=n}\Hom_{R^{e}}\left(C_{p}\otimes D_{q},M\right).
 \end{equation}%
  For every $R$-bilinear morphism $f:C_{p}\otimes D_{q}\rightarrow M$ we define
\begin{equation*}
  \partial _{h}^{p,q}(f)(c\otimes d)=\sum (\theta _{C,A}(c_{(1,1)}) \otimes 1)f(c_{(2,p)}\otimes d)+(-1)^{p+1}\sum f(c_{(1,p)}\otimes  d_{\tau' })(\theta_{C,A}({c_{(2,1)}}_{\tau'})\otimes 1).
\end{equation*}%
  Note that $\partial _{h}^{p,q}(f)$ is a morphism of $R$-bimodules from $C_{p+1}\otimes D_{q}$ to $M.$ Similarly, for any $f$ as above we define the map $\partial_{v}^{p,q}(f):C_{p}\otimes D_{q+1}\rightarrow M$ by
\begin{equation*}
	\partial _{v}^{p,q}(f)(c\otimes d)=\sum \left(1\otimes\theta_{D,B}({d_{(1,1)}}_{\tau })\right) f(c_{\tau}\otimes d_{(2,q)})+(-1)^{q+1}\sum f\left( c\otimes d_{(1,q)}\right) (1\otimes \theta _{D,B}(d_{(2,1)})).
\end{equation*}%
  Taking into account the identification (\ref{dec_Hoch_coh}), by direct computation, we see that the differential maps of the complex $\K^{\ast}(A\otimes _{\sigma }B,M)$ satisfy the relations $\partial^{n}=\partial_{h}^{p,q}+(-1)^{p}\partial_{v}^{p,q},$ and that the diagram 
\begin{equation}
  \begin{array}{c}
    \xymatrix{\Hom_{R^e}(C_p\ot D_{q+1}, M) \ar[rr]^{\pd_h^{p,q+1}} && \Hom_{R^e}(C_{p+1}\ot D_{q+1}, M)
				\\
    \Hom_{R^e}(C_p\ot D_q, M) \ar[rr]_{\pd_h^{p,q}} \ar[u]^{\pd_v^{p,q}} && \Hom_{R^e}(C_{p+1}\ot D_q,M)\ar[u]_{\pd_v^{p+1,q}} }
  \end{array}\label{dc_co}
\end{equation}%
  is commutative. The resulting double complex will be denoted by $\left( \K^{\ast\ast }(A\otimes_{\sigma }B,M),\partial _{h}^{\ast \ast },\partial_{v}^{\ast\ast }\right).$ 
  
  We have seen that the homology groups of the rows and columns of the double complex in Theorem \ref{thm: Koszul} computes the Hochschild of $A$ and $B$ with respect to appropriate $A$ and $B$ bimodules, respectively. For the Hochschild cohomology of $A\otimes _{\sigma }B$ a similar result does not hold in general. Nevertheless, supposing that $R$ is a separable  commutative $\Bbbk $-algebra and that all $R$-bimodules that we work with are symmetric, in view of the adjunction formula, we can rewrite $\K^{pq}(A\otimes _{\sigma}B,M)$ as follows
\begin{equation*}
  \Hom_{R^{e}}(C_{p}\otimes D_{q},M)=\Hom_{R}(C_{p}\otimes D_{q},M)\cong\Hom_{R}\left( C_{p},\Hom_{R}(D_{q},M)\right).
\end{equation*}%
  Through this isomorphism, for a given $q,$ the row $\K^{\ast q}(A\otimes_{\sigma }B,M)$ can be identified with the complex $\K^{\ast }\left( A,\Hom_{R}(D_{q},M)\right) $, where the linear space of $R$-module morphisms from $D_{q}$ to $M$ is regarded as an $A$-bimodule via the actions
\begin{equation*}
    (afa')(d)=\sum(a\otimes 1) f(d_{\nu })(a_{\nu }' \otimes 1).
\end{equation*}%
  Therefore, the page $_{II}E_{1}^{pq}$ of the spectral sequence associated to the row-wise filtration of $\K^{\ast \ast}(A\otimes_{\sigma }B,M)$ has in bidegree $(p,q)$ the Hochschild cohomology group $\hh^p(A,\Hom_{R}(D_{q},M)).$

  To give an analogous description of the columns of $\K^{\ast\ast }(A\otimes_{\sigma }B,M)$ we first endow  $\Hom_{R}(C_{p},M)$ with a $B$-bimodule structure, using the following left and right actions
\begin{equation*}
  (bgb')(c)=\sum (1\otimes b_{\lambda })g(c_{\lambda })(b\otimes 1).
\end{equation*}%
  Then, by the adjunction formula, $\Hom_{R^{e}}(C_{p}\otimes D_{q},M)\cong\Hom_{R}\left( D_{q},\Hom_{R}(C_{p},M)\right)$. Thus,  the column $\K^{p\ast}(A\otimes _{\sigma }B,M)$ is isomorphic to the complex $\K^{\ast}\left(B,\Hom_{R}(C_{p},M)\right) ,$ for any $p.$ In particular, the page $_{I}E_{1}^{pq}$ of the spectral sequence associated to the column-wise filtration has in the spot $(p,q)$ the Hochschild cohomology group of $\hh^q(B,\Hom_{R}(C_{p},M)).$ Summarizing, we sketched the proof of  the following.
\end{fact}

\begin{theorem}\label{thm:DoubleComplexCohom}
  We keep the notation and assumptions of \S\ref{notation}. If $(A,C)$ and $(B,D)$ are Koszul pairs over a separable $R$-algebra and $M$ is an $A\otimes _{\sigma}B$-bimodule, then the Hochschild cohomology of $A\otimes _{\sigma }B$ with coefficients in $M$ is the cohomology of the total complex of $\left( \K^{\ast\ast}(A\otimes _{\sigma }B,M),\partial_{h}^{\ast \ast },\partial _{v}^{\ast\ast}\right) $. Assuming that $R$ is commutative and that the $R$-bimodules $A$, $B$ and $M$ are symmetric, then the pages $_{I}E_{1}^{\ast\ast }$ and $_{II}E_{1}^{\ast \ast }$ of the spectral sequences that correspond to the column-wise and row-wise filtrations are given by
\begin{equation*}
  {}_{I}E_{1}^{pq}=\hh^{q}\left( B,\Hom_{R}(C_{p},M)\right) \text{\quad and\quad}_{II}E_{1}^{pq}=\hh^{q}\left( A,\Hom_{R}(D_{p},M)\right).
\end{equation*}%
  Both spectral sequences converge to the Hochschild cohomology of $A\otimes_{\sigma }B$ with coefficients in $M.$
\end{theorem}

\begin{fact}[Generalized Ore extensions of $R$-rings.]
  Let $A$ be an $R$-ring, where $R$ is a semisimple ring. If $\sigma :A\rightarrow M_{n}(A)$ is a morphism of $R$-rings then, for every couple $(i,j)$ of positive integers which are less than or equal to $n,$ there exists an $R$-bimodule endomorphism $\sigma _{ij}$ of $A$ such that $\sigma _{ij}(a)$ is the $(i,j)$-element of the matrix $\sigma(a), $ and
\begin{equation}\label{sigma1} 
  \sigma _{ij}(ab)=\sum\limits_{p=1}^{n}\sigma _{ip}(a)\sigma_{pj}(b)\quad\text{and\quad }\sigma _{ij}(1)=\delta_{i,j}1.
\end{equation}%
  Let $\{e_{1},\dots ,e_{n}\}$ denote the canonical basis (both as a left and as a right $R$-module) on $V:=R^{n}.$ For any morphism $\sigma $ as above there exists a unique twisting map $\widetilde{\sigma }:T_{R}^{a}(V)\otimes A\rightarrow A\otimes T_{R}^{a}(V)$ such that $\widetilde{\sigma }(1\otimes a)=a\otimes 1$ and
\begin{equation}\label{sigma3}
  \widetilde{\sigma }\left( e_{i_{1}}\otimes \cdots\otimes e_{i_{p}} \otimes a\right) =\sum_{j_{1},\dots ,j_{p}=1} ^{n}\left(\sigma_{i_{1}j_{1}}\circ \cdots\circ \sigma _{i_{p}j_{p}}\right)(a)\otimes e_{j_{1}}\otimes \cdots \otimes e_{j_{p}} .
\end{equation}%
  Since $\widetilde{\sigma }$ is a twisting map of $R$-rings, the twisted tensor product $A\otimes _{\widetilde{\sigma}}T_{R}^{a}(V)$ makes sense. The set
\begin{equation*}
  \mathcal{B}:=\{e_{i_{1}}\otimes \dots \otimes e_{i_{m}}\mid m\geq 0\text{ and }1\leq i_{1},\dots ,i_{m}\leq n\}
\end{equation*}%
  is a basis of the left $A$-module $A\otimes _{\widetilde{\sigma}}T_{R}^{a}(V).$ Therefore, by identifying  $e_{i_{1}}\otimes \dots \otimes e_{i_{m}}$ with $X_{i_{1}}\cdots X_{i_{m}},$ any element in $A\otimes_{\widetilde{\sigma}}T_{R}^{a}(V)$ can be written in a unique way as a left linear combination of noncommutative monomials in the indeterminates $X_{1},\dots ,X_{n}$ with coefficients in $A$. Via this identification, the multiplication in $A\otimes_{\widetilde{\sigma }}T_{R}^{a}(V)$ is determined by the relation
\begin{equation*}
  X_{i}a=\sum_{j=1}^{n}\sigma _{ij}(a)X_{j}.
  \end{equation*}%
  If $n=1$ then $\sigma $ identifies with an algebra automorphism $\sigma_{11} $ of $A$, and $A\otimes _{\widetilde{\sigma}}T_{R}^{a}(V)$ is the usual Ore extension $A_{\sigma _{11}}[X]$. For this reason, in the case when $\sigma $ is an arbitrary $R$-ring morphism from $A$ to $M_{n}(A),$ the corresponding twisted tensor product $A\otimes_{\widetilde{\sigma }}T_{R}^{a}(V)$ will be called the \emph{generalized Ore extension} of $A$ (with respect to $\sigma $) and it will be denoted by $A_{\sigma}\left\langle X_{1},\dots,X_{n}\right\rangle .$

  The twisting map $\widetilde{\sigma },$ associated to an algebra morphism $\sigma :A\rightarrow M_{n}(A),$ is invertible if and only if there is a matrix $(\sigma _{ij}^{\prime })_{i,j}$ whose elements are $R$-bimodule endomorphisms of $A$ satisfying the equations
\begin{equation*}
  \sum\limits_{p=1}^{n}\sigma _{pi}\circ \sigma _{jp}^{\prime} = \sum \limits_{p=1}^{n}\sigma _{pi}^{\prime }\circ \sigma_{jp}=\delta_{i,j}{\Id}_{A}.
\end{equation*}%
  The matrix $(\sigma _{ij}^{\prime })_{i,j}$ determines the inverse of $\widetilde{\sigma }$ by the formula
\begin{equation}\label{sigma4}
  \widetilde{\sigma }^{-1}\left(  a\otimes e_{i_{1}} \otimes \cdots\otimes e_{i_{p}} \right) =\sum_{j_{1},\dots,j_{p}=1} ^{n}e_{j_{1}}\otimes\cdots \otimes e_{j_{p}} \otimes\sigma_{i_{p}j_{p}}^{\prime}\cdots \sigma_{i_{1}j_{1}}^{\prime }(a).
\end{equation}%
  Let us now assume that $A:=\bigoplus_{p\in \mathbb{N}}A^{p}$ is a connected graded $R$-ring. Obviously, $\widetilde{\sigma }$ is a graded twisting map if and only if every $\sigma _{ij}$ is a morphism of graded $R$-bimodules. From now on we assume that $A$ is a connected $R$-ring and that $\sigma :A\rightarrow M_{n}(A)$ is a morphism of $R$-rings such that the corresponding twisting map $\widetilde{\sigma }$ is graded and invertible.
\end{fact}

\begin{theorem}\label{thm:GenOre}
  Let $R$ and $A$ denote a semisimple ring and a Koszul $R$-ring, respectively. If $\sigma $ is a morphism of $R$-rings from $A$ to $M_{n}(A)$ such that the corresponding twisting map $\widetilde{\sigma }$ is graded and invertible, then the generalized Ore extension $A_{\sigma}\left\langle X_{1},\dots ,X_{n}\right\rangle $ is a Koszul $R$-ring.
\end{theorem}

\begin{proof}
  We know that $A^{1}$ generates $A$ and $(A,T(A))$ is a Koszul pair, see Theorem \ref{thm:CharacterizationKoszul}. Let $V$ denote the $R$-bimodule $R^{n}.$ By Proposition \ref{lemma:td}, the pair $\left(T_{R}^{a}(V),C\right) $ is Koszul, where $C:=C_{0}\oplus C_{1}$ and $C_{1}=V. $ Applying the Proposition \ref{th:Tau_Lambda_Nu} for $\widetilde{\sigma }:T_{R}^{a}(V)\otimes A\rightarrow A\otimes T_{R}^{a}(V),$ we get a twisting map of corings $\tau' :T(A)\otimes C\rightarrow C\otimes T(A)$ and the entwining maps $\lambda :T(A)\otimes T_{R}^{a}(V)\rightarrow T_{R}^{a}(V)\otimes T(A)$ and $\nu :A\otimes C\rightarrow C\otimes A,$ such that the relations (\ref{Cond1})-(\ref{Cond4}) hold with $\sigma $ replaced by  $\widetilde{\sigma }.$ Hence by Theorem \ref{thm:koszul} the pair $(A\otimes _{\widetilde{\sigma}}T_{R}^{a}(V),T(A)\otimes _{{\tau }}C)$ is Koszul. In particular, $A_{\sigma }\left\langle X_{1},\dots,X_{n}\right\rangle $ is a Koszul $R$-ring.
\end{proof}

\begin{remark}
  The matrix $(\sigma _{ij}^{\prime })_{i,j}$ that gives the inverse of $\widetilde{\sigma}$ from relation (\ref{sigma4}) can be used to construct explicitly the maps $\tau' ,$ $\lambda $ and $\nu $ from the proof of the preceding theorem. 

  To construct $\tau' $ we define the $R$-bimodule map $\tau_{p}^{\prime }:\Omega_{p}(A)\otimes V\rightarrow V\otimes \Omega _{p}(A)$ as  follows. If $p=0$ then we take $\tau_{p}^{\prime }$ to be the isomorphism $R\otimes V\cong V\otimes R.$ On the other hand, if $p$ is positive, then we set
\begin{equation*}
  \tau _{p}^{\prime }(a_{1}\otimes \cdots \otimes a_{p}\otimes e_{i}):=\sum\limits_{j_{1},\dotsc,j_{p}=1}^{n}e_{j_{1}} \otimes \sigma_{j_{2}j_{1}}^{\prime }(a_{1})\otimes \sigma_{j_{3}j_{2}}^{\prime}(a_{2})\otimes \dotsb \otimes \sigma _{j_{p}j_{p-1}}^{\prime }(a_{p-1})\otimes\sigma_{ij_{p}}^{\prime }(a_{p}),
\end{equation*}%
  for any $a_{1},\dots ,a_{p}\in A.$ Since $\sum\nolimits_{j_{k}=1}^{n}\sigma_{j_{k}j_{k-1}}^{\prime }(a_{k-1})\sigma_{j_{k+1}j_{k}}^{\prime}(a_{k})=\sigma_{j_{k+1}j_{k-1}}^{\prime }(a_{k-1}a_{k})$ it follows easily that $\tau  _{\ast }^{\prime }$ is a morphism of chain complexes from $\left(\Omega _{\ast }(A)\otimes V,\partial _{\ast}\otimes{\Id}_{V}\right) $ to $\left( V\otimes\Omega _{\ast}(A),{\Id}_{V}\otimes \partial _{\ast}\right).$ To show that $\tau _{\ast}^{\prime }$ is compatible with the comultiplication of $\Omega _{\ast }(A)$ one uses the fact that this coring structure is given by the canonical identifications $\Omega _{p+q}(A)\cong \Omega _{p}(A)\otimes \Omega _{q}(A)$. Thus $\tau _{\ast }^{\prime }$  satisfies the conditions from Proposition \ref{pr:twist}, so for every $p$ there is a bimodule map $\tau_{p,1}':T_{p}(A)\otimes V\rightarrow V\otimes T_{p}(A)$ such that $\{\tau_{p,1}'\}_{p\in \mathbb{N}}$ is compatible with the coring structure of $T(A).$ Let $\tau _{p,0}':T_{p}(A)\otimes R\rightarrow R\otimes T_{p}(A)$ be the canonical $R$-bilinear isomorphisms. The family $\{\tau _{p,0}'\}_{p\in \mathbb{N}}$ together with $\{\tau _{p,1}'\}_{p\in \mathbb{N}}$ define the required twisting map $\tau' $ between $T(A)$ and $C.$

  We can now define the entwining structure $\lambda .$ For any $p$ we take $\lambda _{p}^{0}$ to be the canonical isomorphism $T_{p}(A)\otimes R\cong R\otimes T_{p}(A),$ and for $q>0$ we define $\lambda _{p}^{q}:T_{p}(A)\otimes V^{\otimes q}\rightarrow V^{\otimes q}\otimes T_{p}(A)$ by
\begin{equation*}
  \lambda _{p}^{q}:=\left( {\Id}_{V^{\otimes q-1}}\otimes \tau_{p,1}'\right)\circ \left( {\Id}_{V^{\otimes q-2}}\otimes \tau _{p,1}'\otimes V\right)\circ \cdots \circ \left( \tau_{p,1}'\otimes {\Id}_{V^{\otimes q-1}}\right) .
\end{equation*}%
  Finally, the entwining map $\nu :C\otimes A\rightarrow A\otimes  C$ coincides with the isomorphism $R\otimes A\cong A\otimes R$ on $C_{0}\otimes A$ and with $\widetilde{\sigma }^{-1}$ on $C_{1}\otimes A.$ Plainly, by construction, $\widetilde{\sigma },$ $\tau' ,$ $\lambda $ and $\nu $ satisfies the conditions from \S \ref{notation}.
\end{remark}

\begin{fact}[Hochschild (co)homology of generalized Ore extensions.]\label{Ore}
  We keep the notation and the assumptions from Theorem \ref{thm:GenOre}, and we suppose that $M$ is an $A_{\sigma}\left\langle X_{1},\dots ,X_{n}\right\rangle$-bimodule. By the proof of the aforementioned result $(A_{\sigma }\left\langle X_{1},\dots ,X_{n}\right\rangle ,T(A)\otimes _{{\tau}}C)$ is Koszul. Thus to compute the Hochschild homology of $A_{\sigma }\left\langle X_{1},\dots ,X_{n}\right\rangle $ we may use the double complex constructed in Theorem \ref{thm:HochHom}. In this particular case, the double complex has only two non-trivial rows as  $C_q=0$ for any $q>1$. Therefore, for the generalized Ore extension $A_{\sigma}\left\langle X_{1},\dots ,X_{n}\right\rangle $ the double complex becomes 
\begin{equation*}
  \begin{array}{c}
    \xymatrix{0 &V \htp M\htp T_0(A) \ar[l] \ar[d]^-{\pd^v_{0,1}} & V\htp M\htp T_1(A)\ar[l]_-{\pd_{1,1}^h} \ar[d]^-{\pd^v_{1,1}} & V\htp M\htp T_2(A)\ar[l]_-{\ d_{2,1}^h} \ar[d]^-{\pd^v_{2,1}} & \dotsb \ar[l]_-{\pd_{3,1}^h}
				\\
    0 & M\htp T_0(A) \ar[l] & M\htp T_1(A) \ar[l]^-{\pd_{1,0}^h} & M\htp T_2(A)\ar[l]^-{\pd_{2,0}^h} & \dotsb \ar[l]^-{\pd_{3,0}^h}       }
  \end{array}\label{Ore_double_complex_Hom}
\end{equation*}%
The arrows of this double complex are described in the next proposition.  Recall that the comultiplication on $C$ is taken such that for all elements in $V$ we have $\Delta _{1,0}(v)=v\otimes 1$ and $\Delta_{0,1}(v)=1\otimes v$. 
\end{fact}

\begin{proposition}
  The Hochschild homology of  $A_{\sigma }\left\langle X_{1},\dots
,X_{n}\right\rangle $ is isomorphic to the homology of the total complex
of the above double complex. For $c\in T_p(A)$, $m\in M$ and $v\in V$ the arrows are given by
\begin{align*}
  \partial _{p,0}^{h}({\,}m{\,\widehat{\otimes }\,}c)
  & =\sum m\left( \theta (c_{(1,1)})\otimes 1\right) {\,\widehat{\otimes }\,}c_{(2,{p-1})}+(-1)^{p}\sum \left({\theta}({c_{(2,1)}})\otimes 1\right) m{\,\widehat{\otimes}\,}c_{(1,{p-1})}.
				\\
  \partial _{p,1}^{h}(v{\,\widehat{\otimes }\,}m{\,\widehat{\otimes }\,}c) & =\sum {v\,\widehat{\otimes}}m\left(\theta (c_{(1,1)})\otimes 1\right) {\,\widehat{\otimes }\,}c_{(2,{p-1})}+(-1)^{p}\sum {v}_{\tau'}{\,\widehat{\otimes }\,}\left( {\theta }({c_{(2,1)_{\tau'}}})\otimes 1\right)m{\,\widehat{\otimes}\,}c_{(1,{p-1})}.
				\\
  \partial _{p,1}^{v}(v{\widehat{\otimes }}m{\widehat{\otimes}}c) & =\sum m(1\otimes v_{\tau'})\,\widehat{\otimes}\,c_{\tau'}-(1\otimes v)m\,\widehat{\otimes}\,c.
\end{align*}
\end{proposition}

Similarly, the double complex from Theorem \ref{thm:HochCoh} computing the Hochschild cohomology of the generalized Ore extension $A_{\sigma}\left\langle X_{1},\dots ,X_{n}\right\rangle $ has only two nonzero rows
\begin{equation*}
	\begin{array}{c}
   \xymatrix{0 \ar[r]  
        	& \Hom_{R^e}(V,M) \ar[r]^-{\pd_h^{0,1}}
        	& \Hom_{R^e}(T_1(A)\ot V,M) \ar[r]^-{\pd_h^{1,1}}
        	& \Hom_{R^e}(T_2(A)\ot V, M) \ar[r] & \dotsb
   \\
					0 \ar[r] & \Hom_{R^e}(R,M) \ar[r]_-{\pd_h^{0,0}}\ar[u]^{\pd_v^{0,0}}
					& \Hom_{R^e}(T_1(A), M)  \ar[r]_-{\pd_h^{1,0}} 	\ar[u]^-{\pd_v^{1,0}}
        	& \Hom_{R^e}(T_2(A),M) \ar[u]^-{\pd_v^{2,0}} \ar[r]
        	& \dotsb }
	\end{array}\label{Ore_double_complex_Coh}
\end{equation*}%

  The arrows of this double complex are described in the next proposition.

\begin{proposition}
  The Hochschild cohomology of the generalized Ore extension $A_{\sigma }\left\langle X_{1},\dots ,X_{n}\right\rangle $ is isomorphic to the cohomology of the total complex of the above double complex. For $c\in T_p(A)$, $m\in M$ and $v\in V$ the arrows are given by 
\begin{align*}
    \partial _{v}^{p,0}(f)(c\otimes v)
    	& =\sum (1\otimes v_{\tau'})f(c_{\tau'}\otimes 1)-f(c\otimes 1)(1\otimes v),
\\
    \partial _{h}^{p,0}(f)(c)
    	& =\sum (c_{(1,1)}\otimes 1) f(c_{(2,p)})+(-1)^{p}\sum f(c_{(1,p)})(c_{(2,1)}\otimes 1),
\\
    \partial _{h}^{p,1}(f)(c\otimes v)
    	& =\sum (c_{(1,1)}\otimes 1)f(c_{(2,p)}\otimes v)+(-1)^{p}\sum f(c_{(1,p)}\otimes v_{\tau'})({c_{(2,1)}}_{\tau '}\otimes 1).
\end{align*}
\end{proposition} 

\begin{remark}
  The $p$-degree component of $T(A)\ot_\tau C$ is $(T_p(A)\ot R)\oplus (T_{p-1}(A)\ot V)$. We deduce that $\Hd A\leq\Hd A_{\sigma }\left\langle X_{1},\dots ,X_{n}\right\rangle\leq \Hd A+1$, see Theorem \ref{thm:Hdim}. Therefore, if the Hochschild dimension of $A$ is $n$, then the Hochschild dimension of the generalized Ore extensions is $n+1$. Moreover, the Hochschild dimension of $A$ is infinite if and only if the Hochschild dimension of the ore extension $A_{\sigma }\left\langle X_{1},\dots ,X_{n}\right\rangle$ is so.
\end{remark}

\section{Koszul braided \texorpdfstring{$R$}{R}-bialgebras}

  In this section we are going to give some examples of Koszul rings and indicate some applications of our previous results. We  start by showing that, under certain mild assumptions, any almost-Koszul pair that consists of  connected braided commutative bialgebras is always Koszul. Some particular instances of such pairs, including trivial extensions, quantum affine spaces and the incidence algebra of the power poset $(\mathcal{P}(X),\subseteq)$ of a finite set $X$, will be discussed. 

  \begin{fact}[Braided $R$-bialgebras.]  \label{braided}
  A pair $(V,\mathfrak{c}{})$ is called \emph{braided $R$-bimodule} if $V$ is an $R$-bimodule and $\mathfrak{c}{}:V\otimes V\rightarrow V\otimes V$ is an $R$-bimodule map which satisfies the braid equation
\begin{equation*}
  \fc_1\circ\fc_2\circ\fc_1=\fc_2\circ\fc_1\circ{}\fc_2,
\end{equation*}
  where $\fc_1:=\mathfrak{c}\ot\Id_V$ and $\fc_2:=\Id_V\ot\;\fc$.
  A morphism of braided bimodules from $(V, \mathfrak{c}{}_{V})$ to $(W,\mathfrak{c}{}_{W})$ is a bimodule map $f:V\rightarrow W$ such that $\mathfrak{c}{}_{W}\circ (f\otimes f)=(f\otimes f)\circ \mathfrak{c}{}_{V}.$

  The definition of braided algebras from \cite{Ba} can be adapted for $R$-rings without difficulty. See also \cite{AMS} for  more details about braided algebras, coalgebras and bialgebras. The quadruple $(A,m,u,\mathfrak{c})$ is called a \emph{braided $R$-ring} if $(A,\mathfrak{c})$ is a braided $R$-bimodule and $(A,m,u)$ is an $R$-ring such that $\mathfrak{c}$ is a twisting map of $R$-rings. A \emph{morphism} of  braided $R$-rings is, by definition, a morphism of ordinary $R$-rings which, in addition, is a morphism of braided bimodules. The braided ring $(A,m,u,\mathfrak{c})$ is called \emph{braided commutative}, or  $\mathfrak{c}$-\emph{commutative} if $m\circ\mathfrak{c}=m$.

  We shall say that $(A,m,u,\mathfrak{c})$ is a \emph{graded} braided $R$-ring if and only if $A$ is graded and $\mathfrak{c}$ is a graded twisting map of graded $R$-rings. In this case we shall denote the restriction  of $\mathfrak{c}$ to $A^p\ot A^q$ by $\mathfrak{c}^{p,q}$. Note that, for a braided $R$-ring $(A,\mathfrak{c})$, we can consider the twisted tensor product $A\ot_{\mathfrak{c}}A$ with respect to the twisting map $\mathfrak{c}$.

  Braided $R$-corings, and their graded version, are defined in a dual manner. Braided $R$-bialgebras generalize braided bialgebras, introduced by Takeuchi in \cite{Ta}. By definition, a sextuple $(A,m ,u,\Delta,\varepsilon,\mathfrak{c})$ is a \emph{braided $R$-bialgebra} if $(A,m ,u,\mathfrak{c} {})$ is a braided $R$-ring and $(A,\Delta ,\varepsilon ,\mathfrak{c} {})$ is a braided $R$-coring such that $\Delta$ and $\varepsilon$ are morphisms of $R$-rings (on the  $R$-bimodule $A\otimes A$ we take the ring structure $A\otimes_\mathfrak{c} A$). A braided $R$-bialgebra $A$ is \emph{graded} if the underlying ring, coring and braiding structures are so (with respect to the same decomposition $A=\oplus_{n\in\mathbb{N}}A^n$).

  The free $R$-ring $T_R^a(V)$ and the symmetric $R$-ring $S_R(V,\mathfrak{c})$ of a braided $R$-bimodule $(V,\mathfrak{c})$ are the main examples of braided $R$-bialgebras. To define them one follows the same steps as in \cite{AMS}, where the case of braided bialgebras over a base field is considered.   First, one shows that there exists a unique $R$-bimodule map
  \[
    \mathfrak{c}_T:T_R^a(V)\ot T_R^a(V)\rightarrow T_R^a(V)\ot T_R^a(V),
  \]
  which extends $\mathfrak{c}$ and is a solution of the braid equation  that respects the grading on $T_R^a(V)\ot T_R^a(V)$.
  Since $\mathfrak{c}$ is a solution of the braid equation, $\mathfrak{c}_T$ is a twisting map of $R$-rings. Thus $(T_R^a(V),\mathfrak{c}_T)$ is a graded braided $R$-ring.
  Using the universal property of $T_R^a(V)$, one constructs a unique comultiplication  $\Delta:T_R^a(V)\to T_R^a(V)\ot_{\mathfrak{c}_T}T_R^a(V)$ such that $\Delta(v)=v\ot 1+1\ot v$, for any $v\in V$. We also define $\varepsilon:T_R^a(V)\to R$ to be the unique $R$-ring morphism that coincides with $\Id_R$ on the $0$-degree component of $T_R^a(V)$ and vanishes on $V$. 
  As in \cite{AMS} one shows that $(T_R^a(V),\Delta,\varepsilon,\mathfrak{c}_T)$ is a braided $R$-bialgebra.

  Let $(V,\mathfrak{c})$ be a symmetric braided bimodule, that is  $\mathfrak{c}$ satisfies  the additional condition $\mathfrak{c}^2=\Id_{V\ot V}$. Since  $W:=\im(\Id_{V\ot V}-\mathfrak{c})$ contains only primitive elements, the two-sided ideal $I$ generated by $W$ is a coideal in $T_R^a(V)$. Let $S_R(V,\mathfrak{c}):=T_R^a(V)/I$.   Since $\mathfrak{c}_T$ maps $I\otimes T_R^a(V)$ and $T_R^a(V)\otimes I$ to $T_R^a(V)\otimes I$ and $I\otimes T_R^a(V)$, respectively, we conclude that $S_R(V,\mathfrak{c})$ inherits a canonical structure of braided $R$-bialgebra. We denote the braiding of $S_R(V,\mathfrak{c})$ by $\mathfrak{c}_S$. We shall say that $S_R(V,\mathfrak{c})$ is the braided symmetric $R$-ring of $(V,\mathfrak{c})$. This $R$ ring is $\mathfrak{c}_S$-commutative by construction. \end{fact}

\begin{theorem}\label{te: Koszul Pair}
  Let $R$ be a separable  algebra over a field $\Bbbk$ of characteristic zero. Let $(A,\mathfrak{c}_A)$ and $(H,\mathfrak{c}_H)$ denote two connected  braided $R$-bialgebras such that $A^1=H^1$ and $\mathfrak{c}_A^{1,1}=-\mathfrak{c}_H^{1,1}$. If $A$ and $H$ are strongly graded and braided commutative (as $R$-rings), then $(A,H)$ and $(H,A)$ are Koszul pairs. In particular, $(S_R(V,\mathfrak{c}),S_R(V,\mathfrak{-c}))$ is a Koszul pair.
\end{theorem}

\begin{proof} Let $\{m_A^{p,q}\}_{p,q\in\mathbb{N}}$ and
	$\{\Delta_{p,q}^A\}_{p,q\in\mathbb{N}}$ denote the components of the multiplication and of the comultiplication of $A$. We claim that
\begin{equation}\label{eq:m_circ_Delta^A}
    m_{A}^{n,1}\circ \Delta^{A}_{n,1}=(n+1){\Id}_{A^{n+1}}
    \quad \text{and} \quad
    m_{A}^{1,n}\Delta^{A}_{1,n}=(n+1){\Id}_{A^{n+1}},
\end{equation}
  for any $n\geq 0$. We shall prove only the first identity, the other one can be obtained in a similar way. Let us note that the proof of \cite[eq. (21)]{AMS} works for arbitrary graded braided $R$-bialgebras, not only for usual bialgebras. Hence for any $x\in A^n$ and any $a\in A^1$ we have
\begin{equation}
  \Delta ^A_{n,1}(xa)=x\ot a+\sum x_{(1,n-1)}a_{\mathfrak{c}_A}\ot {x_{(2,1)}}_{\mathfrak{c}_A},  \label{ec:Produs2}
\end{equation}
  where in the above relation we used the notation $\mathfrak{c}_A(a'\ot a'')=\sum a''_{\mathfrak{c}_A}\ot a'_{\mathfrak{c}_A}$. On the other hand, since $A$ is braided commutative by hypothesis, we get $\sum a''_{\mathfrak{c}_A}a'_{\mathfrak{c}_A}=a'a''.$   This relation together with \eqref{ec:Produs2} yield
  \[
    (m_A^{n,1}\circ \Delta^A_{n,1})(xa)=xa+\sum x_{(1,n-1)}a_{\mathfrak{c}_A} {x_{(2,1)}}_{\mathfrak{c}_A}=xa+\sum x_{(1,n-1)}{x_{(2,1)}}a=xa+(m_A^{n-1,1}\circ \Delta^A_{n-1,1})(x)a.
  \]
  In conclusion, the first equality in \eqref{eq:m_circ_Delta^A} follows by induction, using the fact that the products $xa$ with $x\in A^n$ and $a\in A^1$ generates $A^{n+1}$ as a linear space, since $A$ is a strongly graded $R$-ring. 
  
  As $H$ is braided commutative, we also have
\begin{equation}\label{eq:m_circ_Delta^H}
  m_{H}^{n,1}\circ \Delta^{H}_{n,1}=(n+1){\Id}_{H^{n+1}}\quad\text{and} \quad m_{H}^{1,n}\circ\Delta^{H}_{1,n}=(n+1){\Id}_{H^{n+1}}.
\end{equation}
  Note that the relation \eqref{ec:Produs2} holds for the braided bialgebra $H$ as well. In particular, taking into account that $\mathfrak{c}_H^{1,1}=-\mathfrak{c}_A^{1,1}$, we get $\Delta_{1,1}^H(xy)=x\ot y+\mathfrak{c}_H^{1,1}(x\ot y)=x\ot y-\mathfrak{c}_A^{1,1}(x\ot y)$, for all $x,y\in H^1$. Since $H$ is strongly graded and $A$ is braided commutative we deduce that $m^{1,1}_A\Delta^H_{1,1}=0$, that is $(A,H)$ is almost-Koszul.
  
  Now we can prove that $(A,H)$ is Koszul, showing that $\K^{\ast}_l(A,H)$ is exact. We know that this complex is the direct sum of its subcomplexes $\K^{\ast}_l(A,H,p)$, where $p$ is an arbitrary nonnegative integer. Since $\K^{\ast}_l(A,H,0)$ is always exact, it is enough to show that $\Id_{\K^{\ast}_l(A,H,p)}$ is null homotopic, for any $p>0$. We need a sequence of maps $s^{\ast}:\K_{l}^{\ast+1}(A,H ,p)\rightarrow \K_{l}^{\ast}(A,H ,p)$ such that
\begin{equation} \label{eq:homotopy} 
  d_{l}^{n-1}\circ s^{n-1}+s^{n}\circ d_{l}^{n}=\Id_{\K^{n}(A,H,p)}.
\end{equation}
  We take $s^{-1}=0=s^{p}$, and for $0\leq n\leq p-1$ we set
\begin{equation*}
    s^{n}(h\ot a)=p^{-1}\sum ha_{(1,1)}\ot a_{(2,n)}.
\end{equation*}%
  Since $\Delta_{1,0}^A(a)=a\otimes 1$ and $\Delta_{0,1}^H(h)=1\otimes h$, it is easy to prove (\ref{eq:homotopy}) in the case when either $n=0$ or $n=p-1$ . We now suppose that $0<n<p-1$ and we pick $h\in H^{p-n}$ and $a\in A^{n} $. Then
\begin{align*}
  (d_{l}^{n-1}\circ s^{n-1})(h\otimes a)
  & =p^{-1}\sum\left(\Id_{H^{m-n}}\otimes m_{A}^{1,n-1}\right)\left(\Delta^{H}_{p-n,1}(ha_{(1,1)})\otimes a_{(2,n-1)}\right)
			\\
  & =p^{-1}\sum h\otimes a_{(1,1)}a_{(2,n-1)}+p^{-1}\sum {h_{(1,p-n-1)}a_{(1,1)}}_{\mathfrak{c}_H}\ot {h_{(2,1)}}_{\mathfrak{c}_H}a_{(2,n-1)}
			\\
  & =np^{-1}h\otimes a +p^{-1} \sum {h_{(1,p-n-1)}a_{(1,1)}}_{\mathfrak{c}_H}\ot {h_{(2,1)}}_{\mathfrak{c}_H}a_{(2,n-1)},
\end{align*}%
  where for the second equality we used (\ref{ec:Produs2}) written for $H$. A similar computation shows us that
\begin{align*}
  (s^{n}\circ d_{l}^{n})(h\otimes a)
  & =p^{-1}\sum(m_{H}^{p-n-1,1}\otimes \Id_{A^{n}})\left(h_{(1,p-n-1)}\otimes\Delta ^{A}_{1,n}(h_{(2,1)}a)\right)
			\\
  & =p^{-1} \sum h_{(1,p-n-1)}h_{(2,1)} \otimes a+ p^{-1}\sum {h_{(1,p-n-1)}a_{(1,1)}}_{\mathfrak{c}_{A}}\otimes    {h_{(2,1)}}_{\mathfrak{c}_A}a_{(2,n-1)}
			\\
  & =(p-n)p^{-1} h\otimes a+p^{-1}\sum  {h_{(1,p-n-1)}a_{(1,1)}}_{\mathfrak{c}_{A}}\otimes {h_{(2,1)}}_{\mathfrak{c}_A}a_{(2,n-1)}.
\end{align*}%
  We conclude the proof by remarking that $\mathfrak{c}_{H}^{1,1}=-\mathfrak{c}_{A}^{1,1}$, so \eqref{eq:homotopy} is true.
\end{proof}

\begin{fact}[Trivial extensions.]
  As a first application of Theorem \ref{te: Koszul Pair} we shall investigate some homological properties of trivial extensions.	Let $V$ be a non-zero $R$-bimodule, where $R$ is an algebra over a field of characteristic zero. Obviously, the identity map of $V\otimes V$ is an involutive solution of the braid equation. Hence we can apply the previous theorem for $\mathfrak{c}:=\Id_{V\ot V}$. In this case we get that $S_R(V,\mathfrak{c})=T_R^a(V)$. On the other hand, since the characteristic of $\Bbbk$ is $0$, the braided bialgebra $A_R(V):=S_R(V,-\mathfrak{c})$ coincides with the trivial algebra extension of $R$ with kernel $V$. Thus, $A_R(V)=R\oplus V$, and the product of two elements in $V$ is zero. The comultiplication of $A_R(V)$ is uniquely defined such that any element in $v\in V$ is primitive, that is $\Delta (v)=v\otimes 1+1\otimes v$. Note that the braiding in $A_R(V)$ restricted to $V\ot V$ is equal to $-\Id_{V\ot V}$.
\end{fact}

\begin{corollary}\label{co:td}\label{cor:hdim_triv_ext}
  If $R$ is a separable algebra over a field $\Bbbk$ of characteristic zero, then $\left(T_R^a(V),A_R(V)\right)$ and $(A_R(V),T_R^a(V))$ are Koszul. Moreover, $\Hd T_R^a(V)=1$ and $\Hd A_R(V)=\sup\{n\mid V^{\otimes n}\neq 0\}$.
\end{corollary}

  We can now use Corollary \ref{co:td} to compute the Hochschild (co)homology of the trivial extension $A_R(V)$ with coefficients in an  $A_R(V)$-bimodule $M$. For, we apply Theorem \ref{thm:HochHom}. Since the homogenous component of $T_R^a(V)$ of degree $n$ is $V^{\ot n}$, we  have $\K_n(A_R(V),M)=M\wot V^{\wot n}$. On the other hand, the component   $\Delta_{p,n-p}$ of the comultiplication of $T_R^a(V)$ is given by the formula 
\[
  \Delta_{p,n-p}(v_1\cdots v_n)=\textstyle \binom{n}{p} v_1\cdots v_p\ot v_{p+1}\cdots v_n.
\]
  where $v_1\cdots v_n$ is a shorthand notation for the tensor monomial $v_1\ot\cdots\ot v_n$, and $\binom{n}{p}=\frac{n!}{p!(n-p)!}$. Hence the differential of $\K_\ast(A_R(V),M)$   satisfies the relation $\partial_n:=n\partial_n'$, where \begin{equation*}\label{eq:differential_trivial_extension}     \partial _{n}'(m\wot v_1\wot\cdots\wot v_n)= (m\cdot v_1)\wot v_2\wot\cdots\wot v_n+(-1)^{n} (v_n\cdot m)\wot v_1\wot\cdots\wot v_{n-1}.
\end{equation*}
  Since the characteristic of $\Bbbk$ is zero, we conclude that the Hochschild homology of $A_R(V)$ with coefficients in $M$ is the homology of the complex $(M\wot V^{\wot\ast},\partial_\ast')$.

  In the more particular case $M:=A_R(V)$, we can go further on the computation of Hochschild homology. Indeed, in this setting we can identify $A_R(V)\wot V^{\wot n}$ with $V^{\wot n}\oplus V^{\wot n+1}$ via the linear map
\begin{equation*}
  (r,v_0)\wot v_1\wot\cdots\wot v_n\longmapsto(rv_1\wot\cdots \wot v_n,v_0\wot v_1\wot\cdots\wot v_n).
\end{equation*}
  Note that the endomorphism $\lambda_n$ of $V^{\wot n}$, that maps $v_1\wot\cdots\wot v_n$ to $(-1)^{n+1}v_n\wot v_1\wot\cdots\wot v_{n-1}$, defines an action of the cyclic group $\mathbb{Z}_n$ on $V^{\wot n}$. With respect to this action and the above identification, $\partial_n':V^{\wot n}\oplus V^{\wot n+1}\rightarrow V^{\wot n-1}\oplus V^{\wot n}$ is given by the formula
\begin{equation*}
  \partial_n'(x,y)=(0,x-\lambda_n(x)).
\end{equation*}
  Hence, for any $n$, we have
\begin{equation*}
  \hh_n(A_R(V),A_R(V))=\textstyle\Ker (\Id_{V^{\wot n}}-\lambda_n)\bigoplus\coker(\Id_{V^{\wot n+1}}-\lambda_{n+1})=\left(V^{\wot n}\right)^{\mathbb{Z}_n}\bigoplus\left(V^{\wot n+1}\right)_{\mathbb{Z}_{n+1}}.
\end{equation*}
  Note that the summands in the above equation are the spaces of invariant and coinvariant elements with respect to the actions induced by $\lambda_n$ and $\lambda_{n+1}$, respectively.

\begin{remark}
  The complex from Theorem \ref{thm:HochCoh} coincides  with the one introduced by Cibils in \cite{Ci3} using a different method, based on the work  on the rigidity of certain algebras \cite{Ci1, Ci2}. 
\end{remark}
  As a more particular case, we now consider the trivial extension associated to a quiver $\boldsymbol{\Gamma}$ with a finite set $\bG^0$ of vertices, but with an arbitrary set $\bG^1$ of arrows. The source and the target maps of $\bG$ will be denoted by $s,t:\bG^1\to \bG^0$.

  By a path of length $n$ (or, equivalently, an $n$-path) in $\bG$ we mean a sequence of $n$ arrows $\gamma:=a_1\cdots a_n$ such that $t(a_i)=s(a_{i+1})$, for any $i=1,\dots,n-1$. The vertices $s(\gamma):=s(a_1)$ and $t(\gamma):=t(a_n)$ will be called the source and the target of $\gamma$, respectively. A vertex $v\in \bG^0$ will be regarded as a path of length $0$, with the same source and target $v$. Clearly the paths of length $1$ coincide with the arrows of $\bG$. The set of $n$-paths will be denoted by $\bG^n$. Note that this notation is consistent with the one that we use for the set of vertices and arrows in $\bG$.

  Let $\Bbbk$ be a field and let $\Bbbk\boldsymbol{\Gamma}$-denote the vector space admitting as a basis the set $\bigcup_{n\in\mathbb{N}}\boldsymbol{\Gamma}^n$ of all paths in $\bG$. Recall that  with respect to the multiplication
\[
  (a_1\cdots a_n)\cdot (a_1'\cdots a_m')=
	\left\{
		\begin{array}{ll}
			a_1\cdots a_na_1'\cdots a_m',& \text{if }t(a_n)=s(a_1');
		\\
      	0, &\text{otherwise};
      \end{array}
   \right.
\]
  $\bK\bG$ is an associative $\bK$-algebra, which is called the \emph{path algebra} of $\bG$. Since $\bG^0$ is finite, the sum of all vertices is a unit of $\Bbbk\bG$. Note that any path can be regarded as the product of its arrows.

  There is a standard grading on $\Bbbk\bG$, given by the decomposition $\Bbbk\bG=\oplus_{n\in\mathbb{N}}\Bbbk\bG^n$, where $\Bbbk\bG^n$ is the linear subspace spanned by $\bG^n$. In particular, $\Bbbk\bG$ is a connected $\kG^0$-ring. Note that $\kG^0$ is isomorphic as a $\Bbbk$-algebra with the direct product of ${\#\bG^0}$ copies of $\Bbbk$, since the vertices of $\bG$ are orthogonal central idempotents in $\bK\bG$. Obviously, this $\bK$-algebra is separable. 
	
  Furthermore,  $\Bbbk\bG^1$ is a $\Bbbk\bG^0$-bimodule with respect to the structure induced by the multiplication. It is well known that the map $a_1\cdots a_n\mapsto a_1\ot\cdots\ot a_n$ is an isomorphism of connected $\Bbbk\bG^0$-rings between $\bK\bG$ and $T_{\Bbbk\bG^0}^a(\Bbbk\bG^1)$. Since $(\bK\bG^1)^{\ot n}\cong\bK\bG^n$, by applying Corollary \ref{cor:hdim_triv_ext} one proves the following.

\begin{corollary}\label{co:Hdim_trivial_extensions_of_quivers}
  Let $\bG$ denote a quiver with a finite number of vertices. The Hochschild dimension of the trivial extension $A_{\bK\bG^0}(\bK\bG^1)$ is given by $\Hd A_{\bK\bG^0}(\bK\bG^1)=\sup\{n\mid \bG^n\neq\emptyset\}.$
\end{corollary}

\begin{fact}[Multiparametric quantum spaces.]\label{fa:O_q}
  We are going to apply the results that we have obtained in order  to give an alternative proof of some known homological properties (cf. for instance \cite{GG2} and \cite{Wa}) of the multiparametric  quantum spaces $O_q(\mathbb{A}^n)$. Recall from \cite{AMS} that the affine quantum spaces can be defined as the symmetric $\Bbbk$-algebra $S_\Bbbk(V, \fc)$, where $V$ is an $n$-dimensional vector space and $\mathfrak{c}:V\ot_{\Bbbk}V\to V\ot_{\Bbbk}V$ is the solution of the braid equation given by $\mathfrak{c}(x_j\otimes x_i)= q_{ij} x_i \otimes x_j$. Here, the set $\{x_1,\dots,x_n\}$ denotes a basis on $V$, and the family of quantum parameters $q = \{q_{ij}\}_{i,j}$ defining $O_q(\mathbb{A}^n)$ is assumed to satisfy the conditions $q_{ij}q_{ji}=1$, for all $i,j=1,\dots,n.$

  One can think of $O_q(\mathbb{A}^n)$ as the algebra with generators $\{x_1,\dotsc, x_n\}$ and relations $x_jx_i = q_{ij}x_ix_j$ for all $1\leq i,j\leq n$, with the natural grading. Note that $x_i^2=0$, provided that $q_{ii}=-1$ and $\mathrm{char}(\Bbbk)\neq 2$. Therefore, if the later relation holds for all $i\in\{1,\dots,n\}$ then $O_q(\mathbb{A}^n)$ is a finite dimensional algebra of dimension $2^n$, and the homogeneous component of degree $n$ is one dimensional. Clearly, in this case all other components of higher degree vanish. On the other hand, if $q_{ii}=1$ for a certain $i$, then the subalgebra generated by $x_i$ is a polynomial ring, so $O_q(\mathbb{A}^n)$  is infinite dimensional.

  In view of Theorem \ref{te: Koszul Pair}, the pair $\left(O_q(\mathbb{A}^n),O_{-q}(\mathbb{A}^n)\right)$ is Koszul. By Theorem \ref{thm:Hdim}, it follows that the Hochschild dimension of $O_q(\mathbb{A}^n)$ is finite if and only if $O_{-q}(\mathbb{A}^n)$ is finite dimensional, \emph{i.e.} $q_{ii}=1$ for all $1\leq i\leq n$. In this situation, $\Hd(O_q(\mathbb{A}^n))=n$ and on the basis
\begin{equation*}
  \{ x_{i_1} x_{i_2} \dotsb x_{i_r} \ |\ 1\leq i_1< i_2 < \dotsb < i_r \leq n \}
\end{equation*}
  the comultiplication for $T(O_{q}(\mathbb{A}^n))\cong O_{-q}(\mathbb{A}^n)$ is defined by
\begin{equation*}
  \Delta_{p,q}(x_{i_1} x_{i_2} \dotsb x_{i_{p+q}}) = \sum\limits_{\nu \in \Sh(p,q)} q_\nu^{-1} x_{i_{\nu(1)}} \dotsb x_{i_{\nu(p)}} \otimes x_{i_{\nu(p+1)}} \dotsb x_{i_{\nu(p+q)}},
\end{equation*}
  where $\nu$ ranges in the set $\Sh(p,q)$ of all shuffles of type $(p,q)$ and the constant $q_\nu$ is the $q$-sign of the  shuffle $\nu$, defined as
\begin{equation*}
  q_\nu = \prod_{i<j,\, \nu(j) < \nu(i)}(-q_{\nu(j)\nu(i)}).
\end{equation*}
  Indeed, $O_{-q}(\mathbb{A}^n)=S_R(V,-\mathfrak{c})$ is a braided bialgebra. Let us denote its braiding by $\mathfrak{c}_O$. As the comultiplication of the multiparametric quantum space is an algebra map from $O_{-q}(\mathbb{A}^n)$ to the twisted tensor product $O_{-q}(\mathbb{A}^n)\otimes_{\mathfrak{c}_O} O_{-q}(\mathbb{A}^n)$, the desired formula for $\Delta_{p,q}$ can be easily proved by induction using the fact that each $x_i$ is a primitive element.
\end{fact}

\begin{fact}[The incidence algebra of the power set of a finite set.] \label{fa:d-cube}
  Let $\bK$ be a field of characteristic $0$ and let $X:=\{1,\dots,d\}$. The power set $\mathcal{P}(X)$ of $X$ is a poset with respect to the order relation given by inclusion. We are now going to show that the incidence algebra $A(X)$ of this poset is a Koszul $R$-ring, where $R:=\bK^{2^d}$. By definition, $A(X)$ has as a basis the set $\{e_{[I,J]}\mid I\subseteq J\}$, where the interval $[I,J]$ contains all subsets $K$ such that $I\subseteq K\subseteq J$. The multiplication of $A(X)$ is defined by the relation 
  \[
  e_{[I,J]} \cdot e_{[I',J']}=\delta_{J,I'} e_{[I,J']}.
  \]
  The unit of $A(X)$ is $\sum_{I\subseteq X}e_{[I,I]}$. The incidence algebra of $X$ is a graded $R$-ring. Its $n$-degree component $A^n(X)$ is spanned by all $e_{[I,J]}$ with $|J\setminus I|=n$. In particular, $A(X)$ is connected as an $R$-ring, as $A^0(X)$ is  generated by  $\{e_{[I,I]}\mid I\subseteq X\}$, which is a complete set of orthogonal idempotents. 
	
  The Koszulity of $A(X)$ will follow as a consequence of the fact that this $R$-ring is an example of braided symmetric $R$-bialgebra. Let $V:=A^1(X)$. Obviously, $V$ is an $R$-bimodule as it is a homogeneous component of  a connected $R$-ring. The set of all tensor monomials $e_{I_0\dots I_n}:=e_{[I_0,I_1]}\ot\cdots \ot e_{[I_{n-1},I_n]}$ with $I_k \subseteq I_{k+1}$ and $|I_{k+1}\setminus I_k|=1$ is a basis of $V^{\ot n}$ regarded as a vector space (recall that by $\ot$ we mean $\ot_R$).  

  Our goal now is to construct an involutive solution $\fc:V\ot V\to V\ot V$ of the braid equation. As a $\bK$-linear map, $\fc$ is uniquely defined by the elements $\fc(e_{I_0I_1I_2})$, where each $I_k$ is a predecessor of $I_{k+1}$ in the poset $\mathcal{P}(X)$, that is $I_k$ is a subset of $I_{k+1}$ and $|I_k|=|I_{k+1}|-1$. Hence  $I_1=I_0\bigcup\{i_1\}$ and $I_2=I_0\bigcup\{i_1,i_2\}$, where $i_1$ and $i_2$ are distinct elements which do not belong to $I_0$. Let $I_1':=I_0\bigcup\{i_2\}$. We can now define $\fc$ by
\begin{equation}\label{def:braiding}
 	\fc(e_{I_0I_1I_2}):=e_{I_0I_1'I_2}.
\end{equation}
  Let us note that the Hasse diagram of $\mathcal{P}(X)$ can be identified with the unit cube in $\mathbb{R}^d$, regarded as a quiver $\bU_d$. The set of vertices of this quiver is the set $\{0,1\}^d\subseteq \mathbb{R}^d$, and an arrow of this quiver has the source $(i_1,\dots,i_d)$ and the target $(j_1,\dots,j_d)$ if and only if the former vertex is the successor of the latter with respect to the  lexicographic ordered on $\{0,1\}^d$. Thus, from a geometric point of view, the $\bK$-linear map $\fc$ interchanges any pair of oriented $2$-paths having the same source and the same target.

  In view of the above geometric interpretation of the Hasse diagram, we shall say that $I_0$  (respectively $I_n$) is the source (respectively the target) of $e_{I_0\dots I_n}\in V^{\ot n}$. Since in the equation \eqref{def:braiding} both elements of the basis on $V^{\ot 2}$ have the same source and the same target it follows that $\fc$ is a morphism of $R$-bimodules. On the other hand, if $e_{I_0I_1I_2I_3}\in V^{\ot 3}$ and $I_k=I_0\bigcup \{i_1\dots,i_k\}$, then
\begin{equation*}
  (\fc_1\circ\fc_2\circ\fc_1)(e_{I_0I_1I_2I_3})=e_{I_0I_1'I_2'I_3}=(\fc_2\circ\fc_1\circ{}\fc_2)(e_{I_0I_1I_2I_3}),
\end{equation*}
  where $ I_1'=I_0\bigcup\{i_3\}$ and $I_2'=I_0\bigcup\{i_2,i_3\}$. In conclusion, $(V,\fc)$ is a braided $R$-bimodule. Clearly, $\fc$ is involutive, so it makes sense to consider the braided $R$-bialgebras $S_R(V,\fc)$ and $S_R(V,-\fc)$.
\end{fact}
\begin{theorem}\label{thm:power_set}
  There is an isomorphism  of graded $R$-rings $A(X)\cong S_R(V,\fc)$ In particular, $A(X)$ is a $3^d$-dimensional Koszul ring of Hochschild dimension $d$.
\end{theorem}
\begin{proof}
  Let $T:=T_R^a(V)$ be the braided $R$-bialgebra with braiding $\fc_T$, see \S\ref{braided}. For $n>d$, the $n$-degree homogeneous component $T^n$ of $T$ vanishes, as any increasing sequence $I_0\varsubsetneq \cdots\varsubsetneq I_k$ has length $k\leq d$. As $S:=S(V,\fc)$ is a quotient braided $R$-bialgebra of $T$ we deduce that $S^n=0$ for any $n>d$.  

  We claim that $\dim S^n=2^{d-n}\binom{d}{n}$ for any $n\leq d$. We start the proof of this relation by recalling that the \emph{involutive} braiding $\fc$ induces an action of the symmetric group $\Sigma_n$ on $T^n$ such that the transposition $\sigma_i:=(i,i+1)$ acts on $v_1\ot\cdots\ot v_n$ by
\[
  \sigma_i\cdot(v_1\ot\cdots\ot v_n)=v_1\ot\cdots\ot\fc(v_i\ot v_{i+1}) \ot\cdots\ot v_n,
\]
  for any $v_1,\dots,v_n\in V$. If $e_{I_0\dots I_n}$ is an element of the basis on $T^n$, with $I_{k}= I_{0}\bigcup\{i_1,\dots i_{k}\}$, then 
\begin{equation}\label{eq:action}
  \sigma\cdot e_{I_0\dots I_n}=e_{J_0\dots J_n},
\end{equation} 
  where $J_0=I_0$ and $J_k=J_0\bigcup\{i_{\sigma(1)},\dots,i_{\sigma(k)}\}$, for any $k=1,\dots,n$.  

  Regarding each permutation $\sigma\in\Sigma_n$ as an $R$-bilinear automorphism of $V^{\ot n}$ and taking into account the definition of the braided symmetric $R$-bialgebra $S$, we get 
\begin{equation*}\label{eq:A^n}
  S^n=\frac{V^{\ot n}}{\sum_{i=1}^n \im(1-\sigma_i)}=\frac{V^{\ot n}}{\sum_{\sigma\in\Sigma_n}\im(1-\sigma)}.
\end{equation*}
  Note that the second equation is a consequence of the relation
\begin{equation*}
  \Id_{V^{\ot n}}-\sigma_{i_1}\cdots\sigma_{i_n}= (\Id_{V^{\ot n}}-\sigma_{i_n})+(\sigma_{i_n}- \sigma_{i_{n-1}}\sigma_{i_n})+\cdots+(\sigma_{i_2} \cdots\sigma_{i_n}-\sigma_{i_{1}}\sigma_{i_2} \cdots\sigma_{i_n})
\end{equation*}
  and of the fact that $\sigma_1,\dots,\sigma_{n-1}$ generate $\Sigma_n$. Hence, $S^n$ coincides with the coinvariant quotient space $\left(V^{\ot n}\right)_{\Sigma_n}$. Since, by assumption $\bK$ is a field of characteristic zero, the canonical linear map from the invariant subspace $\left(V^{\ot n}\right)^{\Sigma_n}$ to $S^n$ is a $\bK$-linear isomorphism.

  In conclusion, we have to show that $\dim \left(V^{\ot n}\right)^{\Sigma_n}=2^{d-n}\binom{d}{n}$. For, we split the representation $T^n=V^{\ot n}$ as a direct sum of sub-representations
\[
	T^n=\oplus_{I\subseteq J}T^n(I,J),
\]
	where $T^n(I,J)$ denotes the vector space spanned by the elements $e_{I_0,\dots,I_n}$ with source $I$ and target $J$. The summands are indexed by all pairs $(I,J)$ such that $I\subseteq J\subseteq \{1,\dots,d\}$ and $J\setminus I$ is a set with $n$ elements. Thus the above decomposition has
\begin{equation}\label{eq:binom}
  \sum_{k=n}^d\binom{d}{k}\binom{k}{n}=2^{d-n}\binom{d}{n}
\end{equation}
  terms. On the other hand $\dim T^n(I,J)=n!$ and, by relation \eqref{eq:action}, it follows that the action of $\Sigma_n$ on $T^n(I,J)$ is transitive for any $I\subseteq J$ with $|J\setminus I|=n$. We deduce that $T^n(I,J)^{\Sigma_n}$ is a vector space of dimension $1$. Hence our claim has been proved, as $\dim S^n=\dim \oplus_{I\subseteq J}T^n(I,J)^{\Sigma_n}=2^{d-n}\binom{d}{n}$.   
	
  We can now prove that $S$ and $A(X)$ are isomorphic. Let $\varphi:T\to A(X)$ be the canonical morphism of graded $R$-rings that extends the identity map $\varphi^0:T^0\to A^0(X)$ and the $R$-bimodule morphism $\varphi^1:T^1\to A^1(X)$ mapping $e_{IJ}$ to $e_{[I,J]}$, for any $I\subseteq J$ with $J$ a successor of $I$. Since $\varphi^n(e_{I_0\dots I_n})=e_{[I_0,I_n]}$, it is easy to see that $\varphi$ vanishes on the ideal generated by the image of $\Id_{T^2}-\fc$. Hence $\varphi$ induces a surjective graded ring morphism $\overline{\varphi}:S\to A(X)$. 
	
  To prove that $\overline{\varphi}$ is an isomorphism we notice that  $\{e_{[I,J]}\mid I\subseteq J\text{ and } |J\setminus I|=n \}$ is a basis of $A^n(X)$. Hence $\dim A^n(X)=2^{d-n}\binom{d}{n}$ by the proof of equation \eqref{eq:binom}. Therefore, every $\overline{\varphi}^n$ is bijective and 
\[
  \dim A(X)=\sum_{n=0}^d \binom{d}{n}2^{d-n}=(1+2)^d=3^d.
\]
  It remains to show that $\Hd A(X)=d$.  Proceeding as above we can show that $S^n(V,-\fc)=0$ for $n>d$. On the other hand, $\dim S^n(V,-\fc)=2^{d-n}\binom{d}{n}$, for $n\leq d$. We conclude applying Theorem \ref{thm:Hdim}.
\end{proof}

\begin{fact}[A generalization of Fr\"{o}berg Theorem.]
  Let $\bK$ be a field. By a result of Fr\"{o}berg \cite{Fr}, the quotient of the free algebra $\bK\langle X_1,\dots,X_n\rangle$ by the ideal generated by a set of non-commuting monomials of degree $2$ is a Koszul $\bK$-algebra. As an application of our results on Koszul pairs, we shall prove a similar result for the quotient of a path algebra by an ideal which is generated by $2$-paths.

  We fix a quiver $\bG$ with a finite number of vertices and a set $\varPhi$ of $2$-paths. The complement of $\varPhi$ in $\bG^2$ will be denoted by $\varPhi'$. We shall also use the following notation: $R:=\bK\bG^0$ and $V:=\bK\bG^1$. We  define the connected $R$-ring $A(\bG,\varPhi)$ to be the quotient of the path algebra $\bK\bG$ by the ideal generated by $\varPhi$.

  For $n=0$ and $n=1$ we set $\bG^n_{\varPhi}:=\bG^n$. On the other hand, if $n\geq 2$ let $\bG_{\varPhi}^n$ denote the set of $n$-paths $\gamma=a_1\cdots a_n$ such that $a_ia_{i+1}\in\varPhi$, for all $i$. The sets $\bG^n_{\varPhi'}$ are defined in a similar way. 
	
  The linear transformation that maps a path in $\bG^n_{\varPhi'}$ to its equivalence class in $A(\bG,\varPhi)^n$ is an isomorphism, for all $n$. This property allows us to identify $A(\bG,\varPhi)$ with the connected $R$-ring $\oplus_{n\geq 0}\bK\bG^n_{\varPhi'}$, whose multiplication is given for  $a_1\cdots a_n\in \bG^n_{\varPhi'}$ and $a'_1\cdots a'_m \in \bG^m_{\varPhi'}$ by
\[
  (a_1\cdots a_n)\cdot (a'_1\cdots a'_m) =
  \left\{
\begin{array}{ll}
a_1\cdots a_na'_1\cdots a'_m, & \text{if } t(a_n)=s(a'_1) \text{ and } a_na'_1\in\varPhi';
				\\
0, & \text{otherwise}.
\end{array}
\right.
\]
  From now on we shall regard the $R$-ring $A(\bG,\varPhi)$ as a subspace of $\bK\bG$ with respect to this multiplication.

  The $R$-bimodule isomorphisms $\Delta_{p,q}:\bK\bG^{p+q}\to\bK\bG^{p}\ot\bK\bG^{q}$
\begin{equation}\label{eq:D_{p,q}}
  \Delta_{p,q}(a_1\cdots  a_{p+q})=a_1\cdots a_p\ot a_{p+1}\cdots a_{p+q}
\end{equation}
  define an $R$-coring structure on $\bK\bG$ (recall that $\ot =\ot_R$). Let $C(\bG,\varPhi)_n$ be the linear subspace spanned by $\bG^n_{\varPhi}$.  By definition, $C(\bG,\varPhi):=\oplus_{n\geq 0}C(\bG,\varPhi)_n$ is a graded $R$-subcoring of $\bK\bG$.
\end{fact}

\begin{theorem}\label{thm:Froberg}
  The pair $(A(\bG,\varPhi),C(\bG,\varPhi))$ is Koszul and $\Hd A(\bG,\varPhi)=\sup\{n\mid \bG^n_\varPhi\neq\emptyset\}$.
\end{theorem}

\begin{proof}
  Let $A:=A(\bG,\varPhi)$ and $C:=C(\bG,\varPhi)$. Clearly, by construction, $A$ is a connected $R$-ring and $C$ is a connected $R$-coring. Let $\theta_{C,A}:=\Id_V$. If $\gamma=aa'$ is a $2$-path in $\varPhi$ then $m^{1,1}\circ\Delta_{1,1}(\gamma)=0$. Indeed, $\Delta_{1,1}(\gamma)=a\ot a'$ and the product in $A$ of $a$ and $a'$ is $0$, since $aa'\not\in\varPhi'$. Thus $(A,C)$ is almost-Koszul.

  By the definition of Koszul pairs, it is enough to prove that $\K^l_\ast(A,C,m)$ is exact for every $m>0$. Recall that $\K^l_n(A,C,m)=A^{m-n}\ot C_n$, for any $n\leq m$, and there are no nontrivial $n$-chains in higher degrees, see \S\ref{K(A,C,m)}. Note that the set of tensor monomials $\gamma'\ot \gamma''$, with $\gamma'\in\bG^{p}_{\varPhi'}$ and $\gamma''\in\bG^{q}_{\varPhi}$ satisfying the condition $t(\gamma')=s(\gamma'')$, is a basis $\bG^{p, q}_{\varPhi',\varPhi}$ on $A^{p}\ot C_q$. Thus $\bG^{m-n,n}_{\varPhi',\varPhi}$ is a basis on $\K_n^l(A,C,m)$.

  The complex $\K^l_\ast(A,C,m)$ is exact in degree $0$, as $A$ is strongly graded and $d^l_1$ is induced by the multiplication. We now assume that $0<n<m$. Let $\omega$ be an $n$-cycle.  Thus
\[
  \omega=\sum_{i=1}^r\alpha_i\lambda_i\ot \lambda'_i,
\]
  for some $\alpha_i\in\bK$ and  $\lambda_i\ot \lambda'_i\in\bG^{m-n,n}_{\varPhi',\varPhi}$. We may assume without loss of generality that $\lambda_i\ot \lambda'_i$ and $\lambda_j\ot \lambda'_j$ are distinct for $i\neq j$. Let $\lambda_i=\gamma_ia_i$, with $\gamma_i \in\bG^{m-n-1}$ and $a_i\in\bG^{1}$. Similarly, $\lambda_i'=a'_i\gamma'_i$, with $a_i'\in\bG^{1}$ and $\gamma_i' \in\bG^{n-1}$. By definition of $d_n^l$ and relation \eqref{eq:D_{p,q}} we get
\[
  d^l_n(\omega)=\sum_{i=1}^r\alpha_i\gamma_ia_i\cdot a'_i\ot \gamma'_i.
\]
  Let $I$ denote the set of all $i$ such that $a_ia'_i\in\varPhi$. Then, for any $i\in I$, we have $\gamma_ia_i\cdot a'_i\ot \gamma'_i=0$.  On the other hand, if $i\not\in I$ then  $\gamma_ia_i\cdot a'_i\ot \gamma'_i=\gamma_ia_i a'_i\ot \gamma'_i$ is an element of $\bG^{m-n+1,n-1}_{\varPhi',\varPhi}$. Moreover, if $i$ and $j$ are distinct elements which do not belong to $I$, then $\gamma_ia_i a'_i\ot \gamma'_i$ and $\gamma_ja_j a'_j\ot \gamma'_j$ are distinct as well. Since $d^l_n(\omega)=0$ it follows that $\alpha_i=0$, for any $i\not\in I$. Let us remark that $\gamma_i\ot a_ia'_i\gamma'_i$ is an element of $\bG^{m-n-1,n+1}_{\varPhi',\varPhi}$, for any $i\in I$, so it is a chain of degree $n+1$. Since
\[
  \omega=\sum_{i\in I}\alpha_i\gamma_ia_i\ot a'_i\gamma'_i=d^l_{n+1}(\sum_{i\in I}\alpha_i\gamma_i\ot a_ia'_i\gamma'_i)
\]
  we deduce that any $n$-cycle is a boundary. As $d_m^l$ maps $1\ot a_1\cdots a_m\in\bG^{0,m}_{\varPhi',\varPhi}$ to $a_1\ot a_2\cdots a_m\in \bG^{1,m-1}_{\varPhi',\varPhi}$, this function is injective. In conclusion $\K^l_\ast(A,C,m)$ is exact.

  The computation of the Hochschild dimension of $A$ follows by Theorem \ref{thm:Hdim}.
\end{proof}

\section*{Acknowledgments}

  The first and the second named authors were partially supported by the research project MTM2010-20940-C02-01. The third named author was financially supported by CNCS, Contract 560/2009 (CNCS code ID 69) and Contract 253/05.10.2011  (CNCS code ID 0635)


\begin{thebibliography}{00}
 
	\bibitem{AMS}     A. Ardizzoni, C. Menini and D. \c{S}tefan, \emph{Braided bialgebras of Hecke-type}, J. Algebra \textbf{321} (2009), 847-65.

	\bibitem{Ba}     J.C. Baez, \emph{Hochschild homology in a braided tensor category}, Trans. Amer. Math. Soc. \textbf{334} (1994), 885--906.

	\bibitem{BG}     A.~Braverman and D.~Gaitsgory, \emph{Poincar\'{e}-Birkhoff-Witt theorem for quadratic algebras of Koszul type}, J. Algebra \textbf{181} (1998), 315--328.

	\bibitem{BM}  A. Borowiec and W. Marcinek, \emph{On crossed product of  algebras}, J. Math. Phys. \textbf{41} (2000), no. 10, 6959--6975.

	\bibitem{BGS} A. Beilinson, V. Ginzburg, and W. Soergel, \emph{Koszul duality patterns in representation theory}, J. Amer. Math. Soc. \textbf{9} (1996), 473--527.

	\bibitem{Br}  T.~Brzezi\'{n}ski, \emph{Comodules and corings}, Handbook of Algebra, vol.~6, Elsevier, 2009, pp.~237--318.

	\bibitem{Ci1} C. Cibils, \emph{Rigidity of truncated quiver algebras}, Advances in Math. \textbf{79} (1990), 18--42.

	\bibitem{Ci2} C. Cibils, \emph{Rigid monomial algebras}, Math. Ann. \textbf{289} (1991), 95--109.

	\bibitem{Ci3} C. Cibils, \emph{Hochschild cohomology algebra of radical square zero algebras}, Algebras and modules, II (Geirenger, 1996), 93-101, CMS Conf. Proc. 24, Amer. Math. Soc. Providence, RI, 1998. 

	\bibitem{CQ} J.  Cuntz and D. Quillen, \emph{Algebra extensions and nonsingularity}, J. Amer. Math. Soc. \textbf{8} (1995), 251--289.

	\bibitem{Cap95a}  A.~Cap, H.~Schichl and J.~Van\v{z}ura, \emph{On twisted tensor products of algebras}, Comm. Algebra \textbf{23} (1995), 4701--4735.
	
	\bibitem{Fr} R.~Fr\"{o}berg, \emph{Determination of a class of Poincar\'{e} series}, Math. Scand. \textbf{37} (1975), 29--39.

	\bibitem{GG1} J. A. Guccione and J. J. Guccione, \emph{Hochschild and cyclic homology of Ore extensions and some quantum examples}, K-Theory \textbf{12} (1997), 259--276.

	\bibitem{GG2}  J. A. Guccione and J. J. Guccione, \emph{Hochschild homology of twisted tensor products}, K-Theory \textbf{18} (1999), no. 4, 363--400.

	\bibitem{JaraUNa} P.~Jara~Mart\'inez, J.~L\'opez Pe\~na, F.~Panaite, and F.~Van~Oystaeyen, \emph{On  iterated twisted tensor products of algebras},    International Journal of Mathematics, \textbf{19} (2008), 1053--1101.

	\bibitem{Kassel95a}  C.~Kassel, \emph{Quantum groups},  Graduate Texts in Math., 155, Springer Verlag, Berlin, 1995.

	\bibitem{PP} A.~Polishchuk and L.~Positselski, \emph{Quadratic algebras}, University lecture series, vol.~37, American Mathematical Society, Providence, R.I., 2005.

	\bibitem{Pr} S.~Priddy, \emph{Koszul resolutions}, Trans. Amer. Math. Soc. \textbf{152} (1970), 39--60.
	
	\bibitem{RS} V.~Reiner and D. Stamate, \emph{Koszul incidence algebras, affine semigroups, and Stanley-Reisner ideals}, Adv. Math. \textbf{224} (2010), 2312--2345.
	\bibitem{ShW1} A. V. Shepler and S. Witherspoon, \emph{Hochschild cohomology and graded Hecke algebras}, Trans. Amer. Math. Soc. \textbf{360} (2008), no. 8, 3975–4005.

	\bibitem{ShW2} A.V. Shepler and S. Witherspoon,  \emph{Finite groups acting linearly: Hochschild cohomology and the cup product}, Adv. Math. \textbf{226} (2011), 2884--2910.
	
	\bibitem{ShW3} A.V. Shepler and S. Witherspoon,  \emph{Group actions on algebras and the graded Lie structure of Hochschild cohomology}, J. Algebra \textbf{351} (2012), 350--381.   
	
	\bibitem{Ta} M. Takeuchi, \emph{Survey of braided Hopf algebras}, Contemp. Math. \textbf{267}  (2000), 301--324.


	\bibitem{Wa} M. Wambst, \emph{Complexes de Koszul quantiques}, Ann. Inst. Fourier \textbf{43,} 4(1993), 1089--1156.

	\bibitem{We} C.~Weibel, \emph{An introduction to homological algebra}, Cambridge University Press, Cambridge, 1997. 
	
		
%
%
%
%
%
%
%
%
%
%
%
%
%
%
%
%
%
%
%
%

\end{thebibliography}
\end{document}